\documentclass[10pt,oneside,a4paper]{article} 				

\usepackage[fleqn]{amsmath} 								
\usepackage{amsthm,amsfonts,latexsym,amssymb,amscd} 		
\usepackage[mathscr]{eucal}   								
\usepackage[all,2cell]{xy} \UseAllTwocells 					
\usepackage{stmaryrd} 										
\usepackage{framed}
\usepackage{color}
\definecolor{cmyk}{cmyk}{.0,.4,.9,.5}						
\usepackage{txfonts}
\usepackage{mathtools} 										

\usepackage{bbold}

\usepackage{scalerel} 	

\usepackage{tikz}
\usetikzlibrary{svg.path}

\definecolor{orcidlogocol}{HTML}{A6CE39}
\tikzset{
  orcidlogo/.pic={
    \fill[orcidlogocol] svg{M256,128c0,70.7-57.3,128-128,128C57.3,256,0,198.7,0,128C0,57.3,57.3,0,128,0C198.7,0,256,57.3,256,128z};
    \fill[white] svg{M86.3,186.2H70.9V79.1h15.4v48.4V186.2z}
                 svg{M108.9,79.1h41.6c39.6,0,57,28.3,57,53.6c0,27.5-21.5,53.6-56.8,53.6h-41.8V79.1z M124.3,172.4h24.5c34.9,0,42.9-26.5,42.9-39.7c0-21.5-13.7-39.7-43.7-39.7h-23.7V172.4z}
                 svg{M88.7,56.8c0,5.5-4.5,10.1-10.1,10.1c-5.6,0-10.1-4.6-10.1-10.1c0-5.6,4.5-10.1,10.1-10.1C84.2,46.7,88.7,51.3,88.7,56.8z};
  }
}

\newcommand\orcidicon[1]{\href{https://orcid.org/#1}{\mbox{\scalerel*{
\begin{tikzpicture}[yscale=-1,transform shape]
\pic{orcidlogo};
\end{tikzpicture}
}{|}}}}

\usepackage[draft=false,breaklinks]{hyperref} 					

\usepackage{cite}

\newcommand{\E}{\mathcal{E}}

\newcommand{\X}{\mathcal{X}}

\newcommand{\Ag}{\mathfrak{A}}\newcommand{\Bg}{\mathfrak{B}}
\newcommand{\Dg}{\mathfrak{D}}
\newcommand{\Fg}{\mathfrak{F}}\newcommand{\Gg}{\mathfrak{G}}
\newcommand{\Hg}{\mathfrak{H}}
\newcommand{\Ig}{\mathfrak{I}} 
\newcommand{\Lg}{\mathfrak{L}}
\newcommand{\Mg}{\mathfrak{M}}
\newcommand{\Pg}{\mathfrak{P}}

\newcommand{\Ug}{\mathfrak{U}}\newcommand{\Wg}{\mathfrak{W}}

\newcommand{\NN}{{\mathbb{N}}}

\newcommand{\As}{{\mathscr{A}}}\newcommand{\Bs}{{\mathscr{B}}}\newcommand{\Cs}{{\mathscr{C}}}
\newcommand{\Ds}{{\mathscr{D}}}
\newcommand{\Es}{{\mathscr{E}}}

\newcommand{\Ms}{{\mathscr{M}}}

\newcommand{\Ws}{{\mathscr{W}}}
\newcommand{\Xs}{{\mathscr{X}}}

\DeclareFontFamily{U}{rsfs}{\skewchar\font127 }
\DeclareFontShape{U}{rsfs}{m}{n}{%
   <5> <6> rsfs5
   <7> rsfs7
   <8> <9> <10> <10.95> <12> <14.4> <17.28> <20.74> <24.88> rsfs10
}{}
\DeclareSymbolFont{rsfs}{U}{rsfs}{m}{n} 
\DeclareSymbolFontAlphabet{\scr}{rsfs}

\newcommand{\Bf}{\scr{B}}\newcommand{\Cf}{\scr{C}}\newcommand{\Ef}{\scr{E}}

\newcommand{\Kf}{\scr{K}}\newcommand{\Mf}{\scr{M}}\newcommand{\Of}{\scr{O}} 
\newcommand{\Qf}{\scr{Q}}

\DeclareMathOperator{\id}{Id}

\DeclareMathOperator{\Hom}{Hom}

\DeclareMathOperator{\Par}{Par}

\renewcommand{\emph}{\textbf} 										
\newcommand{\cj}[1]{\overline{#1}}									
\newcommand{\imp}{\Rightarrow}										
\newcommand{\st}{\ : \ }											

\newcommand{\hlink}[2]{\href{#1}{\texttt{#2}}} 						

\newcommand{\xqedhere}[2]{%
  \rlap{\hbox to#1{\hfil\llap{\ensuremath{#2}}}}}

\newcommand{\xqed}[1]{%
  \leavevmode\unskip\penalty9999 \hbox{}\nobreak\hfill
  \quad\hbox{\ensuremath{#1}}}

\theoremstyle{plain}
\newtheorem{theorem}{Theorem}[section]							
  
\newtheorem{corollary}[theorem]{Corollary}

\newtheorem{lemma}[theorem]{Lemma}
 
\newtheorem{proposition}[theorem]{Proposition}

\newtheorem{definition}[theorem]{Definition}

\theoremstyle{definition} 
\newtheorem{remark}[theorem]{Remark}

\numberwithin{equation}{section}  									
\title{\textbf{Involutive Weak Globular $\omega$-categories}}

\author{\normalsize  
\orcidicon{0000-0002-3891-7717} Paratat Bejrakarbum$^a$, 
\orcidicon{0000-0002-1387-9283} Paolo Bertozzini$^b$   
\\  
\normalsize \textit{Department of Mathematics and Statistics, Faculty of Science and Technology,}
\\
\normalsize \textit{Thammasat University, Pathumthani 12121, Thailand}
\\ 
\normalsize e-mail: 
\ 
$^a$ \texttt{paratat.bejr@dome.tu.ac.th}
\quad 
$^b$ \texttt{paolo.th@gmail.com} 
}

\date{\normalsize{26 March 2023
\quad 
started: 01 October 2020 
}}

\begin{document}

\maketitle

\begin{center}
\textit{
Dedicated to the memory of Francis William Lawvere
} 
\end{center} 

\begin{abstract} \noindent 
We investigate the notion of \textit{involutive} weak globular $\omega$-categories via T.Leinster's approach: as algebras for the initial contracted globular 
operad in the bicategory of globular collections induced by the Cartesian monad of the free \textit{involutive} strict $\omega$-category functor on globular $\omega$-sets. 
An apparently more restrictive notion of involutive weak globular $\omega$-categories as algebras for the initial \textit{operadic-contraction} in the bicategory of globular \textit{contracted-collections} induced by the previous Cartesian monad (where here the operadic multiplications and units satisfy further compatibility axioms with the contractions) is also considered. 

\medskip


\smallskip

\noindent
\emph{MCS-2020:} 
					18N65, 
					18N70, 
					18M40, 
					18N30, 
					18N99. 
\end{abstract}

\tableofcontents

\newpage 

\section{Introduction and Motivation} \label{sec: intro}

Category theory, since its inception in~\cite{EM45}, has always been evolving in very close connection with algebraic topology. The interlink between these two subjects became even more substantial with the development of higher category theory (among the several resources available, see~\cite{ChLa04}, \cite[pages 19-30]{Lei04} for an introductory discussion, \cite{Ba97,BaDo95} for original motivations also from physics and the wiki-site~\href{http://ncatlab.org/nlab}{http://ncatlab.org/nlab} for further details). 

\medskip 

Attempts to test the architecture of higher category theory within non-commutative topology are still in their prenatal stage (see for example~\cite{BCLS20} where only some ``non-commutative'' variants of strict $n$-categories have been considered). 

\medskip

Non-commutative topology is notoriously dominated by the central role of C*-algebras as an arena generalizing the well-known Gel'fand-Na\u\i mark duality between commutative unital C*-algebras and compact Hausdorff topologies.  
As a first minimal attempt to elaborate categorical environments capable of supporting non-commutative homotopy/cobordism and define weak notions of (higher) C*-categories, one would like to axiomatize the existence of (higher) involutions, vertically categorifying, in a weak environment, several already known notions of involutive categories (see~\cite{Ya20} and~\cite[section~4]{BCLS20} for further references).  

\medskip 

In our previous joint paper~\cite{BeBe17} we have been providing a definition of weak involutive $\omega$-category as an algebra for the free involutive $\omega$-category monad in the spirit of J.Penon's algebraic definition of weak-$\omega$-categories~\cite{Pe99}. 
Our ideological point of view is to consider involutions not (only) as symmetries of a (higher) categorical structure, but as unary operations on the very same footing of the binary compositions and nullary identities present in ordinary categories.

\begin{itemize}
\item[$\rightsquigarrow$]
\textit{
The purpose of this work is to produce an algebraic definition of involutive weak globular $\omega$-categories following an operadic definition in the style of~\cite{Lei04}. 
}
\end{itemize}

For now (for simplicity) we limited ourselves to the usual axiomatic setting of (weak) higher categories, although we plan to further develop our work in the direction of non-commutative exchange~\cite[section 3.3]{BCLS20} in view of application to operator algebraic structures. 

\medskip 

In the original treatment from T.Leinster, the operadic and contraction structures introduced onto a given globular $T$-collection are essentially independent; on the other side, terminal objects in the category of globular $T$-collections seem to be naturally contracted $T$-operads that further satisfy compatibility axioms between operadic multiplication/unit and contraction.  

\begin{itemize}
\item[$\rightsquigarrow$]
\textit{
We put forward a more restrictive notion of weak (involutive) globular $\omega$-category as an algebra for an initial ``operadic-contraction'', that is universal among those contracted operads whose multiplication and unit satisfy additional compatibility conditions with the contraction. 
}
\end{itemize}

\medskip 

As a first-aid motivation for readers that might not be familiar with the intricacies of operadic definitions of weak $\omega$-categories, we provide here below a brief synopsis of the construction:
\begin{itemize}
\item[$\blacktriangleright$]
one first introduces strict (involutive) $\omega$-categories and constructs the Cartesian monad $\hat{T}$ (respectively $\hat{T}^\star$ in our involutive case) induced by the free (involutive) $\omega$-category functor, 
\item[$\blacktriangleright$] 
the monad $\hat{T}$ (respectively $\hat{T}^\star$) applied to the the terminal globular $\omega$-set $\bullet$ specifies the input-type ``arity'' of general operations to be axiomatized via operads, 
\item[$\blacktriangleright$]
to the Cartesian monad $\hat{T}$ (respectively $\hat{T}^\star$) a bicategory $\Ef_{\hat{T}}$ (respectively $\Ef_{\hat{T}^\star}$) is associated whose \hbox{1-cells} ${E}\xleftarrow{t_M}M\xrightarrow{s_M}\hat{T}(E)$ (respectively ${E}\xleftarrow{t_M}M\xrightarrow{s_M}\hat{T}^\star(E)$) represent systems ``labeling the multi-input one-target operations'' with source parametrized by $\hat{T}(E)$ (respectively by $\hat{T}^\star(E)$) and target in $E$, 
\item[$\blacktriangleright$] 
generalized $\hat{T}$-multicategories (respectively $\hat{T}^\star$-multicategories) are defined as monads in the previous bicategory and generalized $\hat{T}$-operads are just generalized $\hat{T}$-multicategories whose labeling is provided by $\hat{T}(\bullet)$ (respectively by $\hat{T}^\star(\bullet)$), 
\item[$\blacktriangleright$] 
the actual unbiased description of the evaluation of all the operations involved into the definition of a weak globular $\omega$-category and of their coherence structure are together uniquely specified by a choice of contraction on a $\hat{T}$-operad (respectively on a $\hat{T}^\star$-operad),  
\item[$\blacktriangleright$] 
the monad $L$ (respectively $L^\star$) that is the initial/universal contracted $\hat{T}$-operad (respectively contracted $\hat{T}^\star$-operad) is supposed to specify the labeling of operations, in a weak (involutive) globular $\omega$-category, with certain $\hat{T}(\bullet)$ (respectively $\hat{T}^\star(\bullet)$) inputs and describe their formal compositions and identities, 
\item[$\blacktriangleright$] 
algebras for the initial contracted $\hat{T}$-operad $L$ (respectively for initial contracted $\hat{T}^\star$-operad $L^\star$) are the actual weak (involutive) $\omega$-categories,  
\item[$\blacktriangleright$]  
any contracted $\hat{T}$-operad (respectively contracted $\hat{T}^\star$-operad) $P$ produces a strict-functor from algebras over $P$ to algebras over $L$ (respectively over $L^\star$), that by definition are the weak (involutive) $\omega$-categories: to give an example of weak (involutive) $\omega$-category is equivalent to provide an algebra for a contracted $\hat{T}$-operad (respectively contracted $\hat{T}^\star$-operad) $P$. 
\end{itemize}

Notice that in our involutive case: the involution is used to specify the input type monad $\hat{T}^\star$, 
it is not used to compose formal labeled operations and hence (apart from the labeling input $\hat{T}^\star(\bullet)$) it does not modify the definition of the monad underlying the definition of initial operad $L^\star$: there is usually no involution on the collection of operations making up $L^\star$ (the operad only takes care of the nesting of operations); 
involutions and their evaluations are instead hidden in the choice of contraction that explicitly depends on the base labeling via $\hat{T}^\star(\bullet)$ in place of $\hat{T}(\bullet)$. 

\bigskip 

The content of the paper consists of this brief motivational introduction section~\ref{sec: intro} followed by a section~\ref{sec: pre} of preliminaries, where we recall (in a notation compatible with our previous work) already available material on strict $\omega$-categories, monads in bicategories and T.Leinster's construction of weak $\omega$-categories as algebras for a certain generalized operad. 

\medskip 

Section~\ref{sec: invo} of the paper opens recalling our previously developed definition of \textit{involutive} strict $\omega$-category and continues exposing the new material on an operadic definition of \textit{involutive weak $\omega$-categories} as algebras for a generalized initial  $\hat{T}^\star$-operad in the bicategory of $\hat{T}^\star$-collections, where $\hat{T}^\star$ is the monad of the free involutive strict $\omega$-category construction presented in~\cite[propositions~3.1 and 3.2]{BeBe17}. 

\medskip 

Our main existence theorem~\ref{th: L*} is obtained from a direct procedure, detailed in theorem~\ref{th: free-c-T*}, explicitly constructing by recursion a \textit{free contracted $\hat{T}^\star$-operadic magma} and quotienting it in order to obtain a \textit{free contracted $\hat{T}^\star$-operad} over a $\hat{T}^\star$-collection. 

\medskip 

The more restrictive notion of \textit{operadic-contraction} (in place of the more general contracted-operads considered in~\cite[definition 9.2.1]{Lei04}) is introduced in remark~\ref{rem: ope-con}, where we also mention the possibility to utilize them to define a tighter variant of Leinster's algebraic notion of weak globular $\omega$-categories. 
A parallel treatment of this issue in the involutive case is described in remarks~\ref{rem: ope-con*} and~\ref{rem: ope-con-fin}; the terminal operadic-contraction $\hat{T}^\star(\bullet)$ is examined in detail in remark~\ref{rem: q-cong}.

\medskip 

We close the work in section~\ref{sec: out} with some outlook on possible further work in the direction of weak higher C*-categories and higher categorical non-commutative geometry. 

\section{Preliminaries}\label{sec: pre}

This section is dedicated to a description of all the long background material necessary to formulate algebraic operadic notions of weak $\omega$-categories; most of the material is directly inspired by~\cite{Lei04}.  

\medskip 

Before starting, a foundational disclaimer: although no set-theoretical contradiction will emerge in this work, formally (especially in section~\ref{sec: multicategories}) we will use monads internal to a bicategory of non-small categories \footnote{
A simple solution would be to consider a set-theory based on classes of at least ``3 types'' (2-classes consisting of elements called \hbox{1-classes}, whose elements are called 0-classes and identified as sets) suitably formulating the axiom of ``class-formation'' in such a way that, for $k\in\{2,1,0\}$, ``proper classes'' of level $k$ cannot be elements of classes of level strictly less than $k$.
}

\subsection{Strict Globular $\omega$-categories} 

We recall the formalism and definition of strict globular $\omega$-categories as used in~\cite{BeBe17}. 
\begin{definition}\label{def: omega-cat}
An \emph{$\omega$-quiver} 
$Q^0 \overset{s^0}{\underset{t^0}{\leftleftarrows}} Q^1 \overset{s^1}{\underset{t^1}{\leftleftarrows}} \cdots
\overset{s^{n-2}}{\underset{t^{n-2}}{\leftleftarrows}} Q^{n-1}\overset{s^{n-1}}{\underset{t^{n-1}}{\leftleftarrows}} Q^n \overset{s^n}{\underset{t^n}{\leftleftarrows}} \cdots$ 
consists of a sequence of sets $(Q^k)_{k\in\NN}$ equipped with a sequences of \emph{source} $(s^k)_{k\in\NN}$ and \emph{target} $(t^k)_{k\in\NN}$ maps. 

\medskip 

A \emph{globular $\omega$-set} is an $\omega$-quiver that satisfies the \emph{globularity conditions}:
\begin{equation*}
s^{k}\circ s^{k+1}=s^{k}\circ t^{k+1},
\quad 
t^{k}\circ s^{k+1}=t^{k}\circ t^{k+1}, 
\quad
\forall k\in\NN. 
\end{equation*}
For any $k\in\NN$, an element $x\in Q^k$ is called a \emph{globular $k$-cell} of the globular $\omega$-set. 

\medskip 

A globular $\omega$-set is \emph{reflexive} if it is equipped with a sequence $(\iota^k)_{k\in\NN}$ of maps 
$Q^0 \overset{\iota^0}{\rightarrow} Q^1 \overset{\iota^1}{\rightarrow} \cdots \overset{\iota^{n-1}}{\rightarrow}
Q^n \overset{\iota^n}{\rightarrow} \cdots$ such that $s^k\circ \iota^k=\id_{Q^k}=t^k\circ \iota^k$ for every $k\in \NN$.

\medskip 

A \emph{(reflexive) globular $\omega$-magma} is a (reflexive) globular $\omega$-set equipped with a family of \emph{compositions} 
\begin{gather*}
\circ^m_p: Q^m\times_{Q^p} Q^m\to Q^m, \quad 
(x',x)\mapsto x'\circ^m_p x, 
\quad \forall m\in\NN_0, \quad 0\leq p<m, 
\quad \text{defined on the sets} 
\\
Q^m\times_{Q^p} Q^m:=\Big\{(x',x)\in Q^m\times Q^m \ | \ t^p\circ t^{p+1}\circ\cdots\circ t^{m-1}(x)=s^p\circ s^{p+1}\circ \cdots\circ s^{m-1}(x')\Big\},
\end{gather*} 
such that the following conditions hold: if $m\in\NN_0$,  $0\leq p<m$ and $(x',x)\in Q^m\times_{Q^p} Q^m$,
\begin{itemize}
\item[$\blacktriangleright$] 
$(s^q\circ s^{q+1}\circ\cdots \circ s^{m-1})(x'\circ^m_p x)=
\begin{cases}
(s^q\circ s^{q+1}\circ \cdots \circ s^{m-1})(x')\circ^q_p (s^q\circ s^{q+1}\circ \cdots \circ s^{m-1})(x), & q>p; 
\\
(s^q\circ s^{q+1}\circ \cdots \circ s^{m-1})(x'), & q\leq p. 
\end{cases}
$
\item[$\blacktriangleright$] 
$(t^q\circ t^{q+1}\circ \cdots \circ t^{m-1})(x'\circ^m_p x)=
\begin{cases}
(t^q\circ t^{q+1}\circ \cdots\circ t^{m-1})(x')\circ^q_p (t^q\circ t^{q+1}\circ \cdots\circ t^{m-1})(x), & q>p; \\
(t^q\circ t^{q+1}\circ \cdots \circ t^{m-1})(x), & q\leq p. 
\end{cases}
$
\end{itemize}
A \emph{strict glubular $\omega$-category} $(\Cs,s,t,\iota,\circ)$ is a reflexive globular $\omega$-magma 
\begin{equation*}
\xymatrix{
\Cs^0 \ar[rr]|{\ \iota^0\ } & & \lltwocell^{t^0}_{s^0}{\omit} \Cs^1 \ar[rr]|{\ \iota^1\ }
& & \lltwocell^{t^1}_{s^1}{\omit} \Cs^1 \ar@{.}[r] &  
\Cs^{n} \ar[rr]|{\ \iota^{n}\ } & & \lltwocell^{t^n}_{s^n}{\omit} \Cs^{n+1} \ar@{.}[r] &  
}, 
\quad
\xymatrix{
\Cs^n\times_{\Cs^p}\Cs^n  \ar[r]^{\circ^n_p} & \Cs^n, 
}
\end{equation*}
that satisfies the following list of algebraic axioms: 
\begin{itemize}
\item[$\blacktriangleright$]
(\emph{associativity}) 
for all $p,m\in\NN$, such that $0\leq p<m$, and all $x,y,z\in \Cs^m$ with $(z,y),(y,x)\in \Cs^m\times_{\Cs^p}\Cs^m$: 
\begin{equation*}
(z\circ^m_py)\circ^m_px=z\circ^m_p(y\circ^m_px),
\end{equation*}
\item[$\blacktriangleright$] 
(\emph{unitality}) 
for all $p,m\in\NN$, such that $0\leq p<m$, and all $x\in \mathscr{C}^m$: 
\begin{equation*}
(\iota^{m-1}\circ \cdots\circ \iota^p\circ t^p\circ\cdots\circ t^{m-1})(x)\circ^m_px=x=x\circ^m_p (\iota^{m-1}\circ\cdots\circ \iota^p\circ s^p\circ\cdots\circ s^{m-1})(x),
\end{equation*}
\item[$\blacktriangleright$] 
(\emph{functoriality of identities}) 
for all $q,p\in\NN$, such that $0\leq q<p$, and all 
$(x',x)\in\mathscr{C}^p \times_{\mathscr{C}^q}\mathscr{C}^p$:
\begin{equation*}
\iota^p(x')\circ^{p+1}_q \iota^p(x)=\iota^{p}(x'\circ^p_qx),
\end{equation*} 
\item[$\blacktriangleright$] 
(\emph{binary exchange}) 
for all $q,p,m\in\NN$ such that $0\leq q<p<m$ and all $x,x',y,y'\in \mathscr{C}^m$ with
$(y',y),(x',x)\in \mathscr{C}^m\times_{\mathscr{C}^p}\mathscr{C}^m$ and
$(y',x'),(y,x)\in \mathscr{C}^m\times_{\mathscr{C}^q}\mathscr{C}^m$: 
\begin{equation*}
(y'\circ^m_py)\circ^m_q(x'\circ^m_px)=(y'\circ^m_qx')\circ^m_p(y\circ^m_qx), \quad 
\xymatrix{
\bullet \ruppertwocell{x} \rlowertwocell{x'} \ar[r] & \bullet \ruppertwocell{y} \rlowertwocell{y'} \ar[r] 
& \bullet 
}. 
\end{equation*}
\end{itemize}
A \emph{morphism of $\omega$-quivers} $Q\xrightarrow{\phi}\hat{Q}$ (respectively, of globular $\omega$-sets) is a sequence $(\phi^k)_{k\in\NN}$ of maps $Q^k\xrightarrow{\phi^k}\hat{Q}^k$ that, for any $q\in\NN$, satisfies any one of the following two alternative properties:
\begin{gather}\label{eq: cov}
\text{\emph{$q$-covariance}}:\phantom{ntra} \quad \quad \hat{s}^q\circ\phi^{q+1}=\phi^q\circ s^q, \quad \quad \hat{t}^q\circ\phi^{q+1}=\phi^q\circ t^q, 
\\ \label{eq: con}
\text{\emph{$q$-contravariance}}: \quad \quad \hat{t}^q\circ\phi^{q+1}=\phi^q\circ s^q, \quad \quad \hat{s}^q\circ\phi^{q+1}=\phi^q\circ t^q.  
\end{gather}
An index $q\in \NN$ satisfying~\eqref{eq: cov} (respectively~\eqref{eq: con}) is a \emph{$\phi$-covariance} 
(respectively \emph{$\phi$-contravariance}) index. 
 
\medskip 

A \emph{morphism of reflexive $\omega$-quivers} (respectively of reflexive globular $\omega$-sets) is also required to satisfy: 
\begin{equation*}
\forall k\in\NN \st \hat{\iota}^k\circ\phi^k=\phi^{k+1}\circ\iota^k.
\end{equation*} 

A \emph{morphism of (reflexive) globular $\omega$-magmas} is a morphism of (reflexive) globular $\omega$-sets that, for all $k,q\in \NN$ such that $0\leq q<k$, further satisfies:  
\begin{gather*}
\text{whenever $q$ is $\phi$-covariance index:\phantom{ntra}} \quad 
\phi^k(x\circ^k_q x') = \phi^k(x)\ \hat{\circ}^k_q\ \phi^k(x'), \quad  
\forall (x,x')\in Q^k\times_{Q^q}Q^k, 
\\ 
\text{whenever $q$ is $\phi$-contravariance index:} \quad 
\phi^k(x\circ^k_q x')=\phi^k(x')\ \hat{\circ}^k_q\ \phi^k(x), \quad   
\forall (x,x')\in Q^k\times_{Q^q}Q^k. 
\end{gather*}

\medskip 

An \emph{$\omega$-functor} between two strict globular $\omega$-categories is a morphism of their reflexive globular $\omega$-magmas.  
\end{definition}

\begin{remark}
Each one of the previous notions of ``morphism'' provides a strict 1-category where, given a ``composable pair of morphisms'' $Q\xrightarrow{\psi}\hat{Q}\xrightarrow{\phi}\tilde{Q}$, their composition $Q\xrightarrow{\phi\circ\psi}\tilde{Q}$ is defined componentwise: 
$(\phi^k)_{k\in\NN}\circ (\psi^k)_{k\in\NN}:=(\phi^k\circ\psi^k)_{k\in\NN}$; 
and, for any object $Q:=(Q^k)_{k\in\NN}$, its ``identity morphism'' is defined by $\iota(Q):=(\id_{Q^k})_{k\in\NN}$. 
\xqed{\lrcorner}
\end{remark}

The following result is well-known, see for example~\cite{Pe99} or~\cite[appendix~F]{Lei04}. 
\begin{proposition}\label{prop: free-omega-cat}
Let $\Qf$ denote the strict 1-category of covariant morphisms between globular $\omega$-sets and let $\Cf$ be the strict 1-category of covariant $\omega$-functors between strict globular $\omega$-categories.  

\medskip 

For any globular $\omega$-set $Q$ in $\Qf$, a \emph{free strict globular $\omega$-category over $Q$} is a morphism of globular $\omega$-sets $Q\xrightarrow{\eta_Q}\Ug(\Cs)$, into the underlying globular $\omega$-set $\Ug(\Cs)$ of a strict globular $\omega$-category $\Cs$, satisfying the following universal factorization property: 
for any morphism of globular $\omega$-sets $Q\xrightarrow{\phi}\Ug(\hat{\Cs})$, into the underlying globular $\omega$-set $\Ug(\hat{\Cs})$ of a strict globular $\omega$-category $\hat{\Cs}$, there exists a unique $\omega$-functor $\Cs\xrightarrow{\hat{\phi}}\hat{\Cs}$ such that $\phi=\hat{\phi}\circ\eta_Q$. 

\medskip 

The \emph{forgeful functor} $\Cf\xrightarrow{\Ug}\Qf$ (forgetting compositions and identities of objects in $\Cf$) admits a left-adjoint $\Fg\dashv \Ug$ \emph{free strict globular $\omega$-category functor} $\Cf\xleftarrow{\Fg}\Qf$ that is uniquely determined via a specific construction of free strict globular $\omega$-category $Q\xrightarrow{\eta}\Ug(\Cs)$ of a globular $\omega$-set $Q$ above.  
\end{proposition}

\subsection{Monads in Bicategories and Cartesian Monads}

Algebraic definitions of weak $\omega$-categories in the several approaches available~\cite{Pe99,Ba98,Lei98,Ba02} make use of algebras/modules over certain (generalized) monads. 

\medskip 

In this subsection we review the basic preliminaries on bicategories, introducing monads (respectively  algebras over them) as internal monoids in a bicategory (respectively modules over such monoids). 
For completeness, categorical adjunctions and some of their well-known relations to monads in the bicategory of categories are also recalled, following~\cite{Ri16}. 
Finally Leinster's definition of Cartesian monad is presented. 

\subsubsection{Bicategories} 

Bicategories~\cite{Be67}, \cite[section~I.7.7]{Bo94}, \cite{Lei98}, \cite[section~I.1.5]{Lei04}, are a horizontal categorification of the well-known notion of weak monoidal category (where a monoidal category is just a strict 2-category with one object). There are alternative possible equivalent definitions of this structure, we present here a version that is adapted to our notation for globular $\omega$-quivers in definition~\ref{def: omega-cat}. 

\medskip 

\begin{definition}\label{def: bicat}
A \emph{bicategory} $(\Bf,\circ,\iota,\alpha,\lambda,\rho)$ is a reflexive globular 2-magma  $\Bf^0\leftleftarrows\Bf^1\leftleftarrows\Bf^2$ such that: 
\begin{itemize}
\item[$\blacktriangleright$] 
the 1-magma $\Bf^1\leftleftarrows\Bf^2$ is a strict 1-category with the vertical composition $\circ^2_1$ and vertical identity $\iota^1$; \footnote{
This condition implies that $\Bf^2$, and $\Bf^2\times_{\Bf^0}\Bf^2$ are both bundles of 1-categories over the product $\Bf^0\times\Bf^0$ of discrete categories (with projections that are 1-functors) and that $(\iota^1\circ\iota^0):\Bf^0\to\Bf^2$
is also a functor, where $\Bf^0$ is considered as a discrete category.}
\item[$\blacktriangleright$] 
$(\circ^2_0,\circ^1_0):\Bf^2\times_{\Bf^0}\Bf^2\to\Bf^2$ is a covariant 1-functor; \footnote{
This condition is equivalent to the strict axioms of exchange and functoriality of identities. 
} 
\end{itemize}
that is further equipped with: 
\begin{itemize}
\item[$\blacktriangleright$] 
an \emph{associator natural isomorphism} 
$\left((-\circ^2_0 -)\circ^2_0 -\right)\ \xRightarrow{\alpha}\ \left(-\circ^2_0(-\circ^2_0-)\right)$ between the functors 
\begin{equation*}
(\Bf^2\times_{\Bf^0}\Bf^2)\times_{\Bf^0}\Bf^2\xrightarrow{(-\circ^2_0 -)\circ^2_0 -}\Bf^2, 
\quad 
\Bf^2\times_{\Bf^0}(\Bf^2\times_{\Bf^0}\Bf^2) \xrightarrow{-\circ^2_0(-\circ^2_0-)} \Bf^2, 
\end{equation*}
over the naturally isomorphic 1-categories 
$(\Bf^2\times_{\Bf^0}\Bf^2)\times_{\Bf^0}\Bf^2 \xrightarrow{\alpha}\Bf^2\times_{\Bf^0}(\Bf^2\times_{\Bf^0}\Bf^2)$; 
\item[$\blacktriangleright$]  
a \emph{right unitor natural isomorphism} 
$\left(-\circ^2_0-\right)\ \xRightarrow{\rho}\ \Ig_{\Bf^2}$ between the two 1-functors \footnote{
Here $\Ig_{\Bf}$ denotes the identity functor of the category $\Bf$. 
}  
\begin{equation*}
\Bf^2\times_{\Bf^0}\iota(\Bf^1) \xrightarrow{-\circ^2_0-} \Bf^2, 
\quad \quad 
\Bf^2 \xrightarrow{\Ig_{\Bf^2}}\Bf^2,
\end{equation*}
over the naturally isomorphic 1-categories $\Bf^2\times_{\Bf^0}\iota(\Bf^1)\xrightarrow{\rho}\Bf^2$; 
\footnote{
Here $\iota(\Bf^1)$ is a bundle over $\Bf^0\times\Bf^0$ of discrete categories (consisting only of identity morphisms in $\Bf^2$); it is a distinguished terminal object in the category of bundles over $\Bf^0\times\Bf^0$ of 1-categories with fiberwise 1-functors as morphisms.  
}
\item[$\blacktriangleright$] 
a \emph{left unitor natural isomorphism}
$\left(-\circ^2_0-\right)\ \xRightarrow{\lambda}\ \Ig_{\Bf^2}$
between the two 1-functors 
\begin{equation*}
\iota(\Bf^1)\times_{\Bf^0}\Bf^2 \xrightarrow{-\circ^2_0-} \Bf^2, 
\quad \quad 
\Bf^2 \xrightarrow{\Ig_{\Bf^2}}\Bf^2, 
\end{equation*}
over the naturally isomorphic 1-categories $\iota(\Bf^1)\times_{\Bf^0}\Bf^2\xrightarrow{\lambda}\Bf^2$; 
\end{itemize}
such that any possible diagram involving (iterated) applications of the $\circ^2_0$ composition functor to the associator isomorphism $\alpha$, the left/right unitor isomorphisms $\lambda,\rho$ and their inverses, is commuting. 
\end{definition}
\begin{remark}
Coherence theorems for weak monoidal and bicategories assure that the condition on the commuting diagrams in the previous definition is satisfied as long and the following pentagonal and triangular diagrams commute under $\circ^2_1$-composition: 
\begin{gather*}
\text{\emph{associator coherence}: for all $A\xrightarrow{z}B\xrightarrow{y}C\xrightarrow{x}D\xrightarrow{w}E$ in $\Bf^1$} \\
\xymatrix{
((w\circ^1_0x)\circ^1_0y)\circ^1_0z \ar[rr]^{\alpha_{w,x,y}\ \circ^2_0\ \iota^1(z)} \ar[d]_{\alpha_{(w\circ^1_0x),y,z}} & & \ar[rr]^{\alpha_{w,(x\circ^1_0y),z}} (w\circ^1_0(x\circ^1_0y))\circ^1_0z  & & w\circ^1_0((x\circ^1_0y)\circ^1_0z) \ar[d]^{\iota^1(w)\circ^2_0 \alpha_{x,y,z}}
\\
(w\circ^1_0 x)\circ^1_0(y\circ^1_0z) \ar[rrrr]_{\alpha_{w,x,(y\circ^1_0z)}} & & & & w\circ^1_0(x\circ^1_0(y\circ^1_0z))
}
\\
\text{\emph{unitors coherence}: for all $A\xrightarrow{g}B\xrightarrow{f}C$ in $\Bf^1$} \\ 
\xymatrix{
(f\circ^1_0\iota^0(B))\circ^1_0 g \ar[rrrr]^{\alpha_{f,\ \iota^0(B),\ g}} \ar[drr]_{\rho_f\circ^2_0\iota^1(g)} & & & & f\circ^1_0(\iota^0(B)\circ^1_0 g) \ar[dll]^{\iota^1(f)\circ^2_0\lambda_g} 
\\
& & f\circ^1_0 g & & 
}
\end{gather*}
We refer to the respective entry \cite{n-Lab: coh} for references and details about the proof of this result. 
\xqed{\lrcorner}
\end{remark}

\subsubsection{Monads in a Bicategory}

The notion of \textit{monad} (in a strict 2-category) originated in a concrete adjunction case in~\cite{Go58}, with the name ``standard construction''; it is a powerful instrument that allows to generalize algebraic structures. 

The abstract notion of formal monad over an object of a strict 2-category and in a bicategory are introduced in~\cite{St72} and more recently discussed, for example, in~\cite{Chi15}. 

\begin{definition}\label{def: bicat-monad}
Let $(\Bf,\circ,\iota,\alpha,\lambda,\rho)$ be a bicategory and $\Bs\in\Bf^0$ an object of $\Bf$. 

A \emph{monad $(T,\mu,\eta)$ over $\Bs$} consists of a 1-cell $\Bs\xrightarrow{T}\Bs$ together with a pair of 2-cells $\mu,\eta\in\Bf^2$:  
\begin{itemize}
\item[$\blacktriangleright$]
the \emph{monadic multiplication} $T\circ^1_0 T\xRightarrow{\mu\ } T$,
\item[$\blacktriangleright$]
the \emph{monadic unit} $\iota^0(\Bs)\xRightarrow{\eta\ } T$,  
\end{itemize}
such that the following unitality and associativity diagrams of $\circ^2_1$-compositions of 2-cells are commuting: 
\begin{equation} \label{eq: monadic-ax}
\begin{aligned}
\xymatrix{
\iota^0(\Bs)\circ^1_0 T \ar[d]_{\lambda} \ar[r]^{\ \ \eta\circ^2_0\iota^1(T)} & T\circ^1_0 T \ar[d]_{\mu} & T \circ^1_0 \iota^0(\Bs) \ar[d]^{\rho} \ar[l]_{\iota^1(T)\circ^2_0\eta\ \ }
\\
T \ar[r]_{\iota^1(T)} & T & T \ar[l]^{\iota^1(T)}
}
\quad 
\xymatrix{
(T\circ^1_0 T)\circ^1_0 T \ar[r]^{\alpha} \ar[d]^{\mu\ \circ^2_0\ \iota^1(T)} & T\circ^1_0 (T\circ^1_0 T) \ar[rr]^{\iota^1(T)\ \circ^2_0\ \mu} & & T\circ^1_0 T \ar[d]^\mu
\\
T\circ^1_0 T \ar[rrr]_\mu & & & T
}
\end{aligned}
\end{equation}
Given two objects $\As,\Bs\in\Bf^0$ and two monads, $(T,\mu,\eta)$ over $\As$ and $(S,\mu',\eta')$ over $\Bs$, a \emph{$T$-$S$ bimodule} is a 1-cell $\As\xleftarrow{\Ms}\Bs$ in $\Bf^1$ together with a pair of \emph{left/right evaluations} 2-cells $T\circ^1_0 \Ms \xRightarrow{\theta}\Ms$ and $\Ms \circ^1_0 S\xRightarrow{\vartheta} \Ms$ such that the following diagrams involving vertical composition of 2-cells all commute: 
\begin{gather*}
\xymatrix{
T\circ^1_0 (T \circ^1_0 \Ms)  \ar[d]_{\iota^1(T)\circ^2_0\theta} & \ar[l]_{\alpha_{T,T,\Ms}} (T\circ^1_0 T) \circ^1_0 \Ms \ar[rr]^{\mu\circ^2_0 \iota^1(\Ms)} & & T\circ^1_0 \Ms \ar[d]^{\theta}
\\
T \circ^1_0 \Ms \ar[rrr]_{\theta} & & & \Ms
}
\quad
\xymatrix{
\iota^0(\As)\circ^1_0\Ms \ar[d]_{\lambda_\As} \ar[rr]^{\eta \circ^2_0\iota^1(\Ms)} & & T\circ^1_0 \Ms \ar[d]^{\theta}
\\
\Ms \ar[rr]_{\iota^1(\Ms)} & & \Ms
} 
\\
\xymatrix{
(\Ms \circ^1_0 S)\circ^1_0 S  \ar[d]_{\vartheta \circ^2_0\iota^1(S)}\ar[r]^{\alpha_{\Ms, S, S}} & \Ms \circ^1_0 (S\circ^1_0 S) \ar[rr]^{\iota^1(\Ms)\circ^2_0\mu'} & & \Ms \circ^1_0 S \ar[d]^{\vartheta}
\\
\Ms\circ^1_0 S \ar[rrr]_{\vartheta} & & & \Ms
}
\quad 
\xymatrix{
\Ms\circ^1_0\iota^0(\Bs) \ar[rr]^{\iota^1(\Ms)\circ^2_0\eta'} \ar[d]_{\rho_\Bs}  & & \Ms \circ^1_0 S \ar[d]^{\vartheta}
\\
\Ms \ar[rr]_{\iota^1(\Ms)} & & \Ms
}
\\
\xymatrix{
(T\circ^1_0 \Ms)\circ^1_0 S \ar[rr]^{\alpha_{T,\Ms,S}} \ar[d]_{\theta \circ^2_0 \iota^1(S)} & & 
T\circ^1_0 (\Ms\circ^1_0 S)  \ar[rrr]^{\iota^1(T)\circ^2_0\vartheta} & & &  T \circ^1_0\Ms \ar[d]^{\theta}
\\
\Ms \circ^1_0 S \ar[rrrrr]_{\vartheta} & & & & & \Ms 
}
\end{gather*}
\end{definition}

We will need to apply monads two consecutive times: in the first case it will be a monad $\hat{T}$, over $\Qf$ (respectively $\hat{T}^\star$ over $\Qf^\star$), in the strict bicategory $\Cf$ of natural transformations between covariant functors between (not necessarily small) strict 1-categories; 
in the second case it will be a monad, over a terminal object, in the bicategory $\Ef_{\hat{T}}$ (respectively $\Ef_{\hat{T}^\star}$) that will be subsequently defined in proposition~\ref{prop: T-bicat}. 

\begin{remark}
When applied to the 2-category $\Cf$ of natural transformations between covariant functors between (small) strict 1-categories, the previous definition of monad over an object $\Cs\in\Cf^0$ reproduces the traditional monad endofunctor $T\in [\Cs;\Cs]$ with its multiplication and unit natural transformations. 

Given a small strict 1-category $\Cs\in\Cf^0$, a monad $(T,\mu,\eta)$ over $\Cs$ is an endofunctor $\Cs\xrightarrow{T}\Cs$, $T\in\Cf^1$ equipped with natural transformations $T\circ T\xRightarrow{\mu^T}T$, the monad multiplication $\mu^T\in\Cf^2$ and $\iota^1(\Cs)\xRightarrow{\eta^T} T$, the monad unit $\eta^T\in\Cf^2$, that satisfy, for every object $X\in\Cs^0$, the following properties:
\begin{equation}\label{eq: monad-asso-u}
\mu^T_X\circ T^1(\eta^T_X)=\iota^0_{T^0(X)}=\mu^T_X\circ \eta^T_{T^0(X)}, \quad \quad 
\mu^T_X\circ T^1(\mu^T_X)=\mu^T_X\circ \mu^T_{T^0(X)}.
\end{equation} 
An \emph{algebra for the monad} $(T,\mu^T,\eta^T)$ over the 1-category $(\Cs,\circ,1)$ consists of an object $A\in \Cs^0$ together with an evaluation morphism $T(A) \xrightarrow{\theta^A} A$, $\theta^A\in\Cs^1$,  such that $\theta^A \circ \eta^A =\iota^0_A$ and $\theta^A \circ T^1(\theta^A) = \theta^A \circ \mu^A$. 

Considering a functor $\bullet\xrightarrow{A}\Cs$ from a terminal 1-category $\bullet\in\Cf^0$, the definition of bimodule left for the monad $T$ over $\Cs$ and right for the identity endofunctor of $\bullet$ reproduces the usual definition of $T$-algebra. 
\xqed{\lrcorner}
\end{remark}

We recall these essential properties of monads and adjunctions, see for example~\cite[chapter 5]{Ri16}. 
\begin{remark}\label{remark: monadic} 
Every adjunction $\xymatrix{\Cf \rtwocell^{\Ug}_{\Fg}{'} & \Qf}$, $\Fg\dashv \Ug$, between small 1-categories $\Qf,\Cf$, with unit $\id_\Qf\xRightarrow{\eta}\Ug\circ\Fg$ and co-unit $\Fg\circ\Ug\xRightarrow{\epsilon}\id_{\Cf}$, induces a monad $T:=\Ug\circ\Fg$, in the strict 2-category of natural transformations between functors, over the 1-category $\Qf$, with multiplication 
$\mu:=\Ug \circ \epsilon \circ \Fg$ and unit $\eta$ (see for example~\cite[lemma 5.1.3]{Ri16} for further details). 

\medskip 

Every monad $(T,\mu,\eta)$ on a category $\Cs$, induces an adjunction $\xymatrix{\Cs^T \rtwocell^{\Ug^T}_{\Fg^T}{'} & \Cs}$ where $\Fg^T$ is the free $T$-algebra functor, left-adjoint to the forgetful functor $\Ug^T$ defined on the category $\Cs^T$ of $T$-algebras.   
The monad $T$ coincides with the monad induced by the above adjunction $\Fg^T\dashv \Ug^T$ (see for example~\cite[lemma 5.2.8]{Ri16}). 

\medskip 

Given a monad $(T,\mu,\eta)$ over $\Cs$, the adjunction $\xymatrix{\Cs^T \rtwocell^{\Ug^T}_{\Fg^T}{'} & \Cs}$ is a terminal object (see \cite[proposition~5.2.12]{Ri16}) in the category whose objects are adjunctions $\xymatrix{\Ds \rtwocell^{\Ug}_{\Fg}{'} & \Cs}$, $\Fg\dashv\Ug$, over $\Cs$ and whose morphisms 
$\xymatrix{\Cs \rtwocell^{\Fg}_{\Ug}{'} & \Ds} \xrightarrow{\Hg} \xymatrix{\hat{\Ds} \rtwocell^{\hat{\Ug}}_{\hat{\Fg}}{'} & \Cs}$
are functors $\Ds\xrightarrow{\Hg}\hat{\Ds}$ such that $\hat{\Fg}=\Hg\circ\Fg$ and $\hat{\Ug}\circ\Hg=\Ug$. 

A \emph{monadic adjunction} $\Fg\dashv \Ug$ is an adjunction $\xymatrix{\Ds \rtwocell^{\Ug}_{\Fg}{'} & \Cs}$ such that the terminal morphism $\Ds\xrightarrow{!}\Cs^T$, in the previous category of adjunctions over $\Cs$, is an equivalence of categories (see \cite[definition~5.3.1]{Ri16}). 
A \emph{monadic functor} is a functor $\Ds\xrightarrow{\Ug}\Cs$ with a left adjoint $\Ds\xleftarrow{\Fg}\Cs$ such that the adjunction $\Fg\dashv\Ug$ is monadic.
\xqed{\lrcorner}
\end{remark}

\medskip 

From the adjunctions described in the previous proposition~\ref{prop: free-omega-cat} we have the following. 
\begin{corollary}\label{cor: free-monads}
On the 1-category $\Qf$ of morphisms of globular $\omega$-sets, we have the following  
\begin{itemize}
\item[$\blacktriangleright$]
\emph{free strict globular $\omega$-category monad} $\hat{T}:=\Ug\circ\Fg$.  
\end{itemize}
\end{corollary} 
Certain conditions of ``closure under pull-backs'' are required to define Leinster's generalized operads. 

\begin{definition}\cite[definition~II.4.4.1]{Lei04}
A category $\Ef$ is \emph{Cartesian} if it admits all pull-backs: for any co-span $A\xrightarrow{\alpha}X\xleftarrow{\beta}B$ in $\Ef$, there exists a span $A\xleftarrow{\hat{\alpha}}\hat{X}\xrightarrow{\hat{\beta}}B$ in $\Ef$, with $\alpha\circ\hat{\alpha}=\beta\circ\hat{\beta}$, satisfying the following universal factorization property: 
for any other span $A\xleftarrow{\alpha'}P\xrightarrow{\beta'}B$ in $\Ef$, with $\alpha\circ\alpha'=\beta\circ\beta'$, there exists a unique $P\xrightarrow{\gamma} \hat{X}$ such that $\alpha'=\hat{\alpha}\circ\gamma$ and $\beta'=\hat{\beta}\circ\gamma$. 

\medskip 

A functor $\Ef\xrightarrow{\Gg}\hat{\Ef}$ between 1-categories is a \emph{Cartesian functor} if it preserves pull-backs. 

\medskip 

A natural transformation $\xymatrix{\Es \rtwocell^{\Gg}_{\hat{\Gg}}{\zeta} & \hat{\Es}}$ is a \emph{Cartesian natural transformation} if, for any 1-arrow $A\xrightarrow{x}B$ in $\Es$, the 
span $\Gg(B)\xleftarrow{\Gg(x)}\Gg(A)\xrightarrow{\zeta_A}\hat{\Gg}(A)$ in $\hat{\Es}$ is a pull-back of the co-span $\Gg(B)\xrightarrow{\zeta_B}\hat{\Gg}(B)\xleftarrow{\hat{\Gg}(x)}\hat{\Gg}(A)$ in $\hat{\Es}$. 

\medskip 

A monad $(T,\mu,\eta)$ on a category $\Cs$ is a \emph{Cartesian monad} if the category $\Cs$ the functor $T$ and the two natural transformations $\mu$ and $\eta$ are all Cartesian in the previously defined senses. 
\end{definition}

For the free strict globular $\omega$-category adjunction in proposition~\ref{prop: free-omega-cat} and its associated monad in corollary~\ref{cor: free-monads} we have these Cartesianity conditions as a consequence of \cite[theorem F.2.2]{Lei04}.  
\begin{proposition}  \label{prop: Q-cart}
The 1-category $\Qf$ of small globular $\omega$-sets with morphisms of globular $\omega$-sets 
is Cartesian. 
The 1-category $\Cf$ of small strict globular $\omega$-categories with $\omega$-functors is Cartesian. 
The forgetful 1-functor $\Cf\xrightarrow{\Ug}\Qf$ and the free strict globular $\omega$-category 1-functor $\Cf\xleftarrow{\Fg}\Qf$ are Cartesian. 
The unit $\Ig_{\Qf}\xrightarrow{\eta} \Ug\circ\Fg$ of the free strict globular $\omega$-category functor is a Cartesian natural transformation. 
The free strict globular $\omega$-category monad $\hat{T}:=\Ug\circ\Fg$ is Cartesian. 
\end{proposition}

\subsection{Leinster Weak Globular $\omega$-categories} 

Weak $\omega$-categories in the Leinster's approach~\cite[section~III.9.2]{Lei04} are 
\textit{algebras for a certain monad $L$}; the monad $L$ here utilized is a certain 
\textit{universal (initial) generalized contracted $T$-operad}, where $T$ is the Cartesian free strict globular $\omega$-category monad $\hat{T}$ of corollary~\ref{cor: free-monads}. 

\medskip 

We proceed here to recall all the essential above-mentioned ingredients that are still missing. 

\subsubsection{Generalized Multicategories and Generalized Operads} \label{sec: multicategories}

A monad $T$ on a Cartesian category $\Ef$ allows to formulate a generalized notion of multi-category, where the ``arity'' of the multicategory arrows is specified by $T(\bullet)$, with $\bullet$ an object in $\Ef$. The strategy behind such definition originates in~\cite{Bu71}, \cite{He97} and has been further described in~\cite[section~II.4]{Lei04} whose exposition we are closely following. 
 
\begin{proposition}\cite[section~II.4.2]{Lei04}. \label{prop: T-bicat}
Let $T$ be a Cartesian monad on the Cartesian category $\Ef$. 

Given a span $E_1 \xrightarrow{i_1} X \xleftarrow{i_2}E_2$ in $\Ef$, let us denote by 
$E_1\xleftarrow{\pi_1}E_1\diamond_X E_2\xrightarrow{\pi_2}E_2$ a choice of pull-back. \footnote{
Assume that a specific choice of pull-backs has been done via axiom of choice, for example we will use $E_1\diamond_X E_2:=E_1\times_X E_2$.  
} 

\medskip 

There is a bicategory $(\Ef_T,\circ_T,\iota_T,\alpha_T,\lambda_T,\rho_T)$ defined as follows: 
\begin{itemize}
\item[$\blacktriangleright$] 
$0$-cells in $\Ef_T^0$ are just objects $E$ in $\Ef^0$, 
\item[$\blacktriangleright$] 
$1$-cells $E_2\xleftarrow{(t_P,P,s_P)}E_1$ in $\Ef_T^1$ are spans in $\Ef$ of the form:  $E_2\xleftarrow{t_P}P\xrightarrow{s_P}T^0(E_1)$

\item[$\blacktriangleright$] 
$2$-cells $\xymatrix{E_2 & \ltwocell_{(t_P,P,s_P)}^{(t_Q,Q,s_Q)}{^\phi} E_1}$ consist of commuting diagrams in $\Ef$ of the form: 
$\vcenter{\xymatrix{
&  P \ar[dd]_{\phi} \ar[dr]^{s_P} \ar[ld]_{t_P}  & 
\\
 E_2 &  & T^0(E_1)
\\
 & Q \ar[ur]_{s_Q} \ar[lu]^{t_Q} &  
}}
$ 
\item[$\blacktriangleright$] 
vertical composition $\circ^2_1$ is just the usual composition of 1-arrows in $\Ef$, 
\item[$\blacktriangleright$] 
vertical identities $\iota^1_{\Ef_T}\left(E_2\xleftarrow{t_P}P\xrightarrow{s_P}T(E_1)\right)$ are just identities $\iota^0_\Ef(P)$ in $\Ef$, 
\item[$\blacktriangleright$] 
horizontal composition $\left(E_1\xleftarrow{t_{P_1}}P_1\xrightarrow{s_{P_1}}T(E_2)\right)\circ^1_0\left(E_2\xleftarrow{t_{P_2}}P_2\xrightarrow{s_{P_2}}T(E_3)\right)$ of $1$-arrows in $\Ef_T$ is given by: 
\begin{gather*}
\left( E_1 \xleftarrow{t_{P_1}\circ\pi_1} P_1\diamond_{T(E_2)} T(P_2)
\xrightarrow{\mu^T_{E_3}\circ T(s_{P_2})\circ\pi_2} T(E_3) \right), 
\ \text{as specified by the following diagram}  
\\
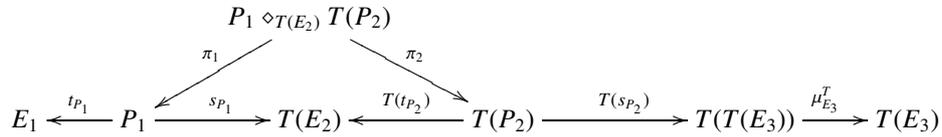
\xymatrix{
& & \ar[dl]_{\pi_1} P_1\diamond_{T(E_2)} T(P_2) \ar[dr]^{\pi_2} & & &
\\
E_1  &\ar[r]^{s_{P_1}} P_1 \ar[l]_{t_{P_1}} & T(E_2) &\ar[rr]^{T(s_{P_2})} T(P_2) \ar[l]_{T(t_{P_2})} & & \ar[r]^{\mu^T_{E_3}} T(T(E_3))& T(E_3)
}
\end{gather*}
horizontal composition of 2-cells 
$\xymatrix{E_1 & E_2 \ltwocell_{(t_{P_1},P_1,s_{P_1})}^{(t_{Q_1},Q_1,s_{Q_1})}{^{\phi_1}} & E_3 \ltwocell_{(t_{P_2},P_2,s_{P_2})}^{(t_{Q_2},Q_2,s_{Q_2})}{^{\phi_2}} }$ 
is given by the 2-cell 
$\xymatrix{E_1 & &  \lltwocell_{(t_{P_1},P_1,s_{P_1})\circ^1_0(t_{P_2},P_2,s_{P_2})}^{(t_{Q_1},Q_1,s_{Q_1})\circ^1_0(t_{Q_2},Q_2,s_{Q_2})}{^{\quad \quad \phi_1\circ^2_0\phi_2}} E_3}$, where $\phi_1\circ^2_0\phi_2$ is defined as the unique morphism $P_1\diamond_{T(E_2)}T(P_2)\xrightarrow{\phi_1\circ^2_0\phi_2} Q_1\diamond_{T(E_2)}T(Q_2)$ in $\Ef$ induced by the universal factorization property of the pull-back $Q_1 \xleftarrow{\pi'_1} Q_1\diamond_{T(E_2)}T(Q_2)\xrightarrow{\pi'_2} T(Q_2)$ via the span $Q_1\xleftarrow{\phi_1\circ\pi_1}P_1\diamond_{T(E_2)}T(P_2)\xrightarrow{T(\phi_2)\circ \pi_2}T(Q_2)$ as specified in this commuting diagram: 
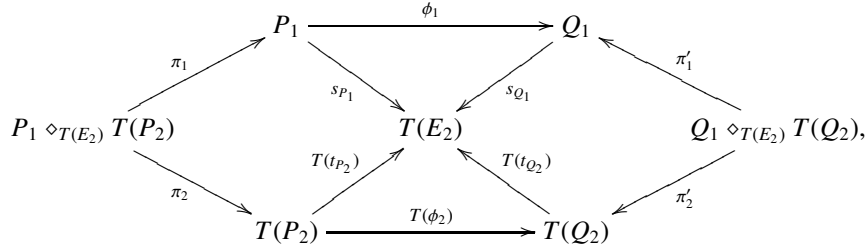
\begin{equation*}
\xymatrix{
& P_1 \ar[rr]^{\phi_1} \ar[dr]_{s_{P_1}} & & Q_1 \ar[dl]^{s_{Q_1}} &  
\\
P_1\diamond_{T(E_2)} T(P_2) \ar[ur]^{\pi_1} \ar[dr]_{\pi_2}  & & T(E_2) & & Q_1\diamond_{T(E_2)} T(Q_2), \ar[ul]_{\pi'_1} \ar[dl]^{\pi'_2}
\\
& T(P_2) \ar[ru]^{T(t_{P_2})} \ar[rr]^{T(\phi_2)} & & T(Q_2) \ar[lu]_{T(t_{Q_2})}  & 
}
\end{equation*}
\item[$\blacktriangleright$] 
horizontal identities $\iota_T^0(E):=\left(E \xleftarrow{\iota^0(E)} E\xrightarrow{\eta^T_E} T(E)\right)$, for all objects $E\in\Ef_T^0$.  
\item[$\blacktriangleright$]
left-unitors $\xymatrix{E_2 & \ltwocell_{\iota^1(E_2)\circ^1_0P}^{P}{^\ \ \lambda_P} E_1}$ are uniquely determined by universal factorization property of the pull-backs:
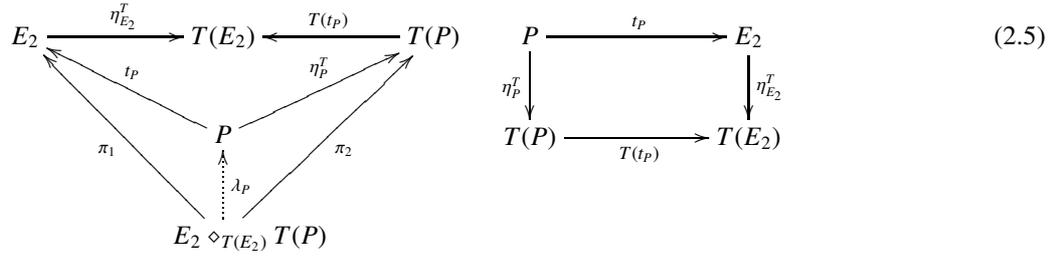
\begin{equation} \label{eq: bicat-left-unitors}
\xymatrix{
E_2 \ar[r]^{\eta^T_{E_2}} & T(E_2) & \ar[l]_{T(t_P)} T(P) 
\\
& \ar[ul]_{t_P} P\ar[ur]^{\eta^T_P} & 
\\
& \quad \quad E_2\diamond_{T(E_2)}T(P) \ar@{.>}[u]_{\lambda_P} \ar[uur]_{\pi_2} \ar[uul]^{\pi_1}& 
}
\quad 
\xymatrix{
P \ar[rr]^{t_P} \ar[d]_{\eta^T_P} & & E_2 \ar[d]^{\eta^T_{E_2}}
\\
T(P) \ar[rr]_{T(t_P)} & & T(E_2)
}
\end{equation}
(where the second pull-back-diagram above is assured by Cartesianity of $\eta^T$) and by the diagram
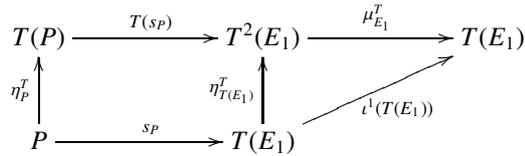
\begin{equation*}
\xymatrix{
T(P) \ar[rr]^{T(s_P)} & & T^2(E_1) \ar[rr]^{\mu^T_{E_1}} & & T(E_1)
\\
P \ar[rr]^{s_P} \ar[u]^{\eta^T_P} & & T(E_1) \ar[u]^{\eta^T_{T(E_1)}} \ar[rru]_{\iota^1(T(E_1))} & & 
}
\end{equation*}
that commutes by Cartesianity of $\eta^T$ and the unital property of the monad $T$; 
\item[$\blacktriangleright$]
right-unitors $\xymatrix{E_2 & \ltwocell_{P\circ^1_0\iota^1(E_1)}^{P}{^\ \ \rho_P} E_1}$ are also determined by universal factorization property of the pull-backs:
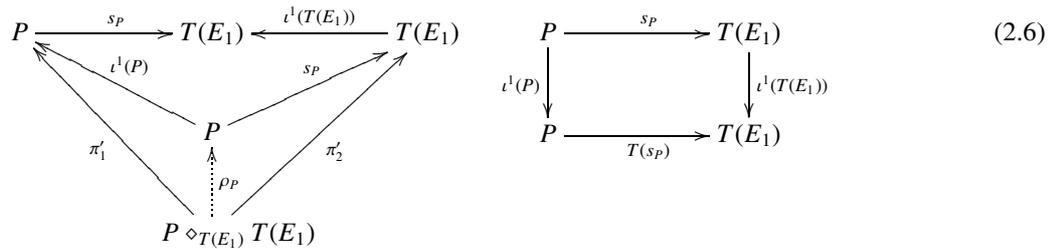
\begin{equation} \label{eq: bicat-right-unitors}
\xymatrix{
P \ar[r]^{s_P} & T(E_1) & \ar[l]_{\iota^1(T(E_1))} T(E_1) 
\\
& \ar[ul]_{\iota^1(P)} P\ar[ur]^{s_P} & 
\\
& \quad \quad P\diamond_{T(E_1)}T(E_1) \ar@{.>}[u]_{\rho_P} \ar[uur]_{\pi'_2} \ar[uul]^{\pi'_1}& 
}
\quad 
\xymatrix{
P \ar[rr]^{s_P} \ar[d]_{\iota^1(P)} & & T(E_1) \ar[d]^{\iota^1(T(E_1))}
\\
P \ar[rr]_{T(s_P)} & & T(E_1)
}
\end{equation}
(where the second pull-back diagram is due to Cartesianity of $T$) and by the commuting diagram
\begin{equation*}
\xymatrix{
P \ar[rr]^{s_P} \ar[rrd]_{s_P} & & \ar[d]|{\iota^1(T(E_1))} T(E_1) \ar[rr]^{T(\eta^T_{E_1})} & & T^2(E_1) \ar[lld]^{\mu^T_{E_1}}
\\
& & T(E_1) & & 
}
\end{equation*}
that is again commuting because of the unital property of the monad $T$; 
\item[$\blacktriangleright$] 
associators $\xymatrix{E_4 & & \lltwocell_{(P_1\circ^1_0P_2)\circ^1_0 P_3}^{P_1\circ^1_0(P_2\circ^1_0 P_3)}{^\quad\quad  \alpha_{P_1P_2P_3}} E_1}$ for the compositions $E_4 \xleftarrow{P_1}E_3\xleftarrow{P_2}E_2\xleftarrow{P_3}E_1$ are uniquely determined by universal factorization property of the following pair of pull-backs over $T(E_2)$: \footnote{
Here, with a little abuse of notation, we utilize the same symbols $\pi_1,\pi_2$ to denote all the  pull-back projections of compositions. 
} 
\begin{gather}\label{eq: bicat-associators-unitors}
\vcenter{\xymatrix{
\ar[rr]^{\mu^T_{E_2}\circ T(s_{P_2})\circ\pi_2} P_1\diamond T(P_2)  & & T(E_2) & & T(P_3)\ar[ll]_{T(t_{P_3})}
\\
& & P_1\diamond T(T(P_2)\diamond T(P_3)) \ar[urr]|{\mu^T_{P_3}\circ T(\pi_1)\circ\pi_1} \ar[ull]|{\pi_2\circ\theta} & & 
\\
& & \ar@{.>}[u]|{\alpha_{P_1P_2P_3}} (P_1\diamond T(P_2))\diamond T(P_3) \ar[uurr]_{\pi_2} \ar[uull]^{\pi_1}& &
}}
\end{gather}
where the morphism $\theta$ is uniquely determined by universal factorization property in the following diagram, where all the squares (due to composition or to the Cartesianity of $T$ and $\mu^T$) are pull-backs: 
\begin{gather*}
\xymatrix{
P_1\diamond T(P_2\diamond T(P_3)) \ar@/_1cm/[dd]_{\pi_1} \ar@{.>}[d]_{\theta} \ar[rr]^{\pi_2} & & \ar[d]_{T(\pi_1)} T(P_2\diamond T(P_3)) \ar[rr]^{T(\pi_2)}  & & \ar[d]_{T^2(t_{P_3})} T^2(P_3)\ar[rr]^{\mu^T_{P_3}} & & T(P_3) \ar[d]_{T(t_{P_3})}
\\
P_1\diamond T(P_2) \ar[d]_{\pi_1} \ar[rr]_{\pi_2} & & T(P_2) \ar[rr]_{T(s_{P_2})} \ar[d]_{T(t_{P_2})} & & T^2(E_2)\ar[rr]_{\mu^T_{E_2}} & & T(E_2) 
\\ 
P_1 \ar[rr]^{s_{P_1}}& & T(E_3) & & & & 
}
\end{gather*}
and where the following triangles diagrams, with target $E_4$ and source $T(E_1)$ 
\begin{equation}\label{eq: op-asso-T*}
\vcenter{
\xymatrix{
E_4 & & & T^2(P_3) \ar[d]_{\mu^T_{P_3}} \ar[r]^{T^2(s_{P_3})} & T^3(E_1) \ar[r]^{T(\mu^T_{E_1})} \ar[d]_{\mu^T_{T(E_1)}} & T^2(E_1) \ar[d]^{\mu^T_{E_1}} & 
\\
& & \ar[llu]|{t_{P_1}\circ\pi_1} P_1\diamond T(T(P_2)\diamond T(P_3))\ar[ur]_{T(\pi_2)\circ\pi_2}
& T(P_3) \ar[r]_{T(s_{P_3})} & T^2(E_1) \ar[r]^{\mu^T_{E_3}} & T(E_1) 
\\
& & \ar[uull]^{t_{P_1}\circ\pi_1\circ\pi_1} (P_1\diamond T(P_2))\diamond T(P_3) \ar[u]|{\alpha_{P_1P_2P_3}} \ar[ur]_{\pi_2}
}}
\end{equation}
are commuting because of Cartesianity of $T$ and $\mu^T$ and by the associativity property of the monad $T$. 
\end{itemize}
\end{proposition}

\begin{definition}
Given a Cartesian monad $(T,\mu^T,\eta^T)$ on a Cartesian category $\Ef$, a \emph{$T$-multicategory on $E$} is a monad $(P,\mu_P,\eta_P)$ over an object $E$ in the bicategory $\Ef_T$ defined in proposition~\ref{prop: T-bicat}. 

\medskip 

A \emph{$T$-operad} $(P,\mu_P,\eta_P)$ is a $T$-multicategory on a terminal object $\bullet$ in $\Ef$. 
\end{definition}

\begin{remark}
In practice a $T$-multicategory consists of: an object $E\in\Ef^0$ (indexing the inputs and outputs of multicategory arrows); a 1-cell $E\xleftarrow{t_P}P\xrightarrow{s_P} T^0(E)$ in $\Ef_T$ (specifying all the multicategory arrows with multi-input as an element of $T^0(E)$ and only one target in $E$); a composition of multicategory arrows specified by the monadic multiplication $\mu_P$; and multicategory identity, specified by the monadic unit $\eta_P$. 
\xqed{\lrcorner}
\end{remark}

\begin{proposition}\label{prop: operad-cat}
For any object $E\in\Ef$, we have  a 1-category, denoted by $\Ef^T_E$, whose objects are 1-arrows in $\Ef_T$ with source and target $E$ and whose morphisms are 2-arrows in $\Ef_T$. 

\medskip 

For every object $E\in\Ef$, there is a \emph{category $\Cf^T_E$ of $T$-multicategories on $E$} that is the subcategory of $\Ef^T_E$ with: 
\begin{itemize}
\item[$\blacktriangleright$] 
objects of $\Cf^T_E$ are $T$-multicategories $(P,\mu_P,\eta_P)$ in $\Ef_T$,
\item[$\blacktriangleright$]
morphisms $(P,\mu_P,\eta_P)\xrightarrow{\phi}(Q,\mu_{Q},\eta_{Q})$ in $\Cf^T_E$ are morphisms in $\Ef_T$ such that: 
\begin{equation*}
\xymatrix{
P\circ^1_0 P \ar[rr]^{\mu_P} \ar[d]_{\phi\circ^2_0\phi} & & P \ar[d]^{\phi}
\\
Q\circ^1_0Q \ar[rr]^{\mu_{Q}} & & Q
}\quad \quad 
\xymatrix{
\iota^0_T(E) \ar[rr]^{\eta_P} \ar[d]_{\iota^1(E)} & & P \ar[d]^{\phi}
\\
\iota^0_T(E) \ar[rr]^{\eta_{Q}} & & Q
}
\end{equation*}
\item[$\blacktriangleright$] 
composition and identity coincide with those in $\Ef_T$. 
\end{itemize}  
\end{proposition}
\begin{definition}\label{def: ET}
For a terminal object $\bullet$ in a Cartesian category $\Ef$ with a Cartesian monad $T$ we denote by: 
\begin{itemize}
\item[$\blacktriangleright$]
$\Ef^T_\bullet$ the  \emph{category of $T$-collections} over $\bullet$
\item[$\blacktriangleright$] 
$\Of^T_\bullet:=\Cf^T_\bullet$ the \emph{category of $T$-operads} over $\bullet$. 
\end{itemize}
\end{definition}

\subsubsection{Contractions and Globular Collections} 

We now specialize the discussion to the bicategory $\Qf_{\hat{T}}$ constructed from the Cartesian category $\Ef:=\Qf$ of globular $\omega$-sets equipped with the free strict globular $\omega$-category Cartesian monad $\hat{T}$. 

\medskip  

The codification of the algebraic axioms for weak globular $\omega$-categories via ``contractions'' originates in~\cite{Pe99}; the following notion of contraction, for globular $\omega$-sets, appears in~\cite[section~II.5]{Lei01} and~\cite[section~III.9.1]{Lei04} and is actually used to formalize, at the same time, the (evaluation of) operations and the algebraic and coherence axioms for weak globular $\omega$-categories. 
\begin{definition}\label{def: L-contraction}
Let $Q\in\Qf^0$ be a globular $\omega$-set; we say that $x_1,x_2\in Q^k$ are \emph{parallel $k$-cells} if either $k=0$ or 
\begin{gather*}
s^{k-1}(x_1)=s^{k-1}(x_2), \quad \quad t^{k-1}(x_1)=t^{k-1}(x_2). 
\end{gather*}
We denote by $\Par_Q$ the family of pairs of parallel cells of the globular $\omega$-set $Q$. 

\medskip 

Let $Q_1\xrightarrow{\pi}Q_2$ be a covariant morphism in the category $\Qf$ of globular $\omega$-sets. 

\medskip 

we define $\Par(\pi):=\Big\{(x^+,y,x^-) \ | \ (x^+,x^-)\in\Par_{Q_1}, \ y\in Q_2^n, \ n\in\NN_0,\ \pi(x^+)=t^{n-1}(y), \ s^{n-1}(y)=\pi(x^-) \Big\}$. 

\medskip 

A \emph{Leinster contraction on $\pi$} is a map $\kappa:\Par(\pi)\to Q_1$ such that:  
\begin{equation}\label{eq: L-con}
s(\kappa(x^+,y,x^-))=x^-, \quad t(\kappa(x^+,y,x^-))=x^+, \quad \pi(\kappa(x^+,y,x^-))=y, \quad \forall (x^+,y,x^-)\in\Par(\pi). 
\end{equation}
\end{definition}

\begin{definition}\label{def: contracted_T-coll}
Let $\bullet\in\Qf$ denote a terminal object in the category of globular $\omega$-sets and $T$ a monad on $\Qf$.

\medskip 

A \emph{globular $T$-collection} is a morphism $Q\xrightarrow{\pi} T(\bullet)$ in $\Qf$; 
a \emph{globular contracted $T$-collection}  
consists of a Leinster contraction $\kappa$ on a globular $T$-collection 
$\Par(\pi) \xrightarrow{\kappa} Q \xrightarrow{\pi} T(\bullet)$. 
\end{definition}

\begin{lemma} \label{lem: parallel}
Given a commuting diagram of covariant functors in the category $\Qf$ of globular $\omega$-sets: 
\begin{equation*}
\xymatrix{
Q_1  \ar[rrrr]^{\phi} \ar[rrd]_{\pi_1} & & & & \ar[dll]^{\pi_2} \hat{Q}_1
\\
& & Q_2 & & 
}
\end{equation*}
there is a well-defined induced map $\Par(\pi_1)\xrightarrow{\Par_\phi}\Par(\pi_2)$ given by $\Par_\phi:(x^+,y,x^-)\mapsto (\phi(x^+),y,\phi(x^-))$. 
\end{lemma}

We recall here the notion of category of globular contracted operads from~\cite[definition~III.9.2.3]{Lei04}.
\begin{proposition} \label{def: g-c-T-operads}
For any terminal object $\bullet\in\Qf^0$, we have the following categories: 
\begin{itemize}
\item[$\blacktriangleright$]
the \emph{category $\Qf^{\hat{T}}_\bullet$ of globular $\hat{T}$-collections over $\bullet$}, specified as in definition~\ref{def: ET}; 
\item[$\blacktriangleright$] 
the \emph{category $\Qf^{\hat{T},\kappa}_\bullet$ of globular contracted $\hat{T}$-collections over $\bullet$} whose objects are globular contracted $\hat{T}$-collections and whose morphisms $(P,\pi_P,\kappa_P)\xrightarrow{\phi}(Q,\pi_Q,\kappa_Q)$ are globular $\omega$-functors $P\xrightarrow{\phi}Q$ such that 
\begin{equation*}
\vcenter{
\xymatrix{
\Par(\pi_P) \ar[d]_{\kappa_P} \ar[rrr]^{\Par_\phi} & &  & \Par(\pi_Q) \ar[d]^{\kappa_Q}
\\
P \ar[rrr]^{\phi} & & & Q  
}
} \quad \quad 
\phi\circ \kappa_P=\kappa_Q\circ \Par_\phi;
\end{equation*}  
\item[$\blacktriangleright$] 
the \emph{category $\Of^{\hat{T}}_\bullet$ of globular $\hat{T}$-operads over $\bullet$}, specified as in definition~\ref{def: ET}; 
\item[$\blacktriangleright$] 
the \emph{category $\Of^{\hat{T},\kappa}_\bullet$ of globular contracted $\hat{T}$-operads over $\bullet$}
with objects $(P,\pi_P,\kappa_P,\mu_P,\eta_P)$ such that 
$(P,\pi_P,\mu_P,\eta_P)\in\Of^{\hat{T}}_\bullet$ is a globular $\hat{T}$-operad and $(P,\pi_P,\kappa_P)\in\Qf^{\hat{T}}_\bullet$ is a globular contracted $\hat{T}$-collection, \footnote{
\label{foo: ope-con}
Notice that at this point we are not imposing any compatibility axiomatic requirement between the contraction $\kappa$ and the operadic unit $\eta$ and multiplication $\mu$ simultaneously defined on a $\hat{T}$-collection; we will address this issue in remark~\ref{rem: ope-con} here below. 
} 
and with morphisms $(P,\pi_P,\kappa_P,\mu_P,\eta_P)\xrightarrow{\phi}(Q,\pi_Q,\kappa_Q,\mu_Q,\eta_Q)$ such that $(P,\pi_P,\mu_P,\eta_P)\xrightarrow{\phi}(Q,\pi_Q,\mu_Q,\eta_Q)$ is a morphism in $\Of^{\hat{T}}_\bullet$ and $(P,\pi_P,\kappa_P)\xrightarrow{\phi}(Q,\pi_Q,\kappa_Q)$ is a morphism in $\Qf^{\hat{T},\kappa}_\bullet$. 
\end{itemize}
There is a forgetful functor $\Qf^{\hat{T},\kappa}_\bullet\xrightarrow{\Wg}\Qf^{\hat{T}}_\bullet$ defined on object by $(Q,\pi,\kappa)\mapsto (Q,\pi)$, for every globular contracted $\hat{T}$-collection $(Q,\pi,\kappa)$ and as identity map on morphisms.

\medskip 

There is a forgetful functor $\Of^{\hat{T},\kappa}_\bullet \xrightarrow{\Ug} \Of^{\hat{T}}_\bullet$ defined on object by $(P,\pi,\kappa,\mu,\eta)\mapsto (P,\pi,\mu,\eta)$, for every globular contracted $\hat{T}$-operad $(P,\pi,\kappa,\mu,\eta)$ and as identity map on morphisms.
\end{proposition}

\begin{remark}\label{rem: w-inv-omega}
Although $T$-operads are defined for any operad $T$ on a Cartesian category $\Ef$, Leinster's contractions are defined only on globular $T$-collections, where $T$ is an arbitrary monad on the Cartesian category $\Qf$ of globular $\omega$-sets.  
As a consequence, proposition~\ref{def: g-c-T-operads} holds actually for any Cartesian operad $T$ on the Cartesian category of globular $\omega$-sets $\Qf$. We will use this fact, substituting the free strict globular $\omega$-category monad $\hat{T}$ with the free involutive strict globular $\omega$-category monad $\hat{T}^\star$ in the next section~\ref{sec: w-inv-omega}. 
\xqed{\lrcorner}
\end{remark}

\begin{remark} \label{rem: ope-con}
As anticipated in footnote~\ref{foo: ope-con}, following the treatment in~\cite[definition 9.2.1]{Lei04}, we did not include in our definition of globular contracted $\hat{T}$-operad, in the last point of proposition~\ref{def: g-c-T-operads}, any further compatibility axioms between contractions and operadic units/multiplications. 

\medskip 

We introduce here a more restrictive notion of \emph{globular $\hat{T}$-operadic contraction} over $\bullet$, consisting of data $(P,\pi_P,\kappa_P,\eta_P,\mu_P)$ such that $(P,\pi_P,\kappa_P)$ is a globular contracted $\hat{T}$-collection in $\Qf^{\hat{T},\kappa}_\bullet$ and $(P,\pi_P,\eta_P,\mu_P)$ is a globular $\hat{T}$-operad in $\Of^{\hat{T}}_\bullet$ that furthermore satisfy the commutativity of the following diagrams in $\Qf^{\hat{T}}_\bullet$: 
\begin{gather*}
\xymatrix{
\Par(\pi_\bullet) \ar[d]_{\kappa_\bullet} \ar[rrr]^{\Par_{\eta_P}} & & &  \Par(\pi_P) \ar[d]^{\kappa_P}
\\
\bullet \ar[rrr]^{\eta_P} & & & P
}
\quad \quad 
\xymatrix{
\Par(\pi_P\circ^2_0\pi_P) \ar[d]_{\kappa_{P}\circ^2_0\kappa_P} \ar[rrr]^{\Par_{\mu_P}} & & & \Par(\pi_P) \ar[d]^{\kappa_P}
\\
P\circ^1_0 P \ar[rrr]^{\mu_P} & & & P
}
\end{gather*}
where $\bullet\xrightarrow{\pi_\bullet}\hat{T}(\bullet)$ is $\pi_\bullet:=\eta^T_\bullet$, we have $P\circ^1_0P\xrightarrow{\pi_P\circ^2_0\pi_P}\hat{T}(\bullet)\circ^1_0\hat{T}(\bullet)$ and lemma~\ref{lem: parallel} provides $\Par_{\eta_P}$ and $\Par_{\mu_P}$. \footnote{
Notice that the contraction $\kappa_\bullet:(x^+,y,x^-)\mapsto y$ is uniquely determined by the defining axioms of contraction and the definition of $\pi_\bullet$. 
}

\medskip 

The category $\Of\Kf^{\hat{T}}_\bullet$ of globular $\hat{T}$-operadic contractions over $\bullet$ is just the full subcategory of $\Of^{\hat{T},\kappa}_\bullet$ determined by the objects defined above. 

\medskip 

It is perfectly viable to use the category $\Of\Kf^{\hat{T}}_\bullet$ instead of $\Of^{\hat{T},\kappa}_\bullet$ in order to define a slightly more restrictive notion of weak globular $\omega$-category as an algebra for the initial object in $\Of\Kf^{\hat{T}}_\bullet$, whose existence can be obtained following similar steps and in the case examined in the subsequent subsection~\ref{sec: L-w-T}. 
\xqed{\lrcorner}
\end{remark}

\subsubsection{Weak $\omega$-categories}\label{sec: L-w-T}

The subsequent fundamental result is proved in~\cite[appendix~G]{Lei04}.
\begin{theorem}\label{prop: L}
The category $\Of^{\hat{T},\kappa}_\bullet$ of globular contracted $\hat{T}$-operads has initial objects. 
\end{theorem}
The proof provided in the reference above is not direct, and is based on the following definitions and results. 
\begin{definition}
A strict 1-category $\Cs$ is \emph{filtered} if any finite diagram $\Ds\xrightarrow{\Dg}\Cs$ admits a co-cone. \footnote{ 
This is equivalent to requiring: 
a) that there exists at least an object in $\Cs^0$ (the vertex of a co-cone on the empty diagram); 
b) for any two objects $A_1,A_2\in\Cs^0$, there exists a co-cone $A_1\xrightarrow{\alpha_1}C\xleftarrow{\alpha_2}A_2$; 
c) for any finite diagram consisting of pair of parallel arrows $A_1\xrightarrow{f,g}A_2$, there exists a co-cone $A_1\xrightarrow{\alpha_1}C\xleftarrow{\alpha_2}A_1$ and hence $\alpha_2\circ f=\alpha_1=\alpha_2\circ g$. 
} 

\medskip 

A \emph{filtered co-limit} is a co-limit for a diagram $\Ds \xrightarrow{\Dg}\Bs$ with filtered 1-category $\Ds$. 

\medskip 

A \emph{finitary functor} is a functor $\As\xrightarrow{\Fg}\Bs$ that preserves filtered co-limits. 

\medskip 

A \emph{finitely presentable object} is an object $A\in\Cs^0$ whose covariant $\Hom$-functor $\Hom_\Cs(A,-)$ is finitary. 

\medskip 
A category $\Cs$ is \emph{locally finite presentable} if the following properties are satisfied: 
\begin{itemize}
\item[$\blacktriangleright$]
$\Cs$ is co-complete (it admits all co-limits of diagrams $\Ds\xrightarrow{\Dg}\Cs$ with small $\Ds$),  
\item[$\blacktriangleright$]
the full subcategory of finitely presentable objects of $\Cs$ is essentially small, \footnote{
Recall that in a \emph{locally small category} $\Cs$, for all objects $A,B\in\Cs^0$, $\Hom_\Cs(A,B)$ is a set; a locally small category is \emph{small} if and only if $\Cs^0$ is a set; a category is \emph{essentially small} if and only if it equivalent to a small category. 
}
\item[$\blacktriangleright$]
every object of $\Cs$ is a filtered co-limit of at least one diagram of finitely presentable objects. 
\end{itemize}
\end{definition}

This first lemma is a consequence of \cite[theorems~6.5.1, 6.5.2 and 6.5.4 in appendix D]{Lei04}. 
\begin{lemma}
If $T$ is a finitary Cartesian monad on a Cartesian category $\Ef$, the forgetful functor  $\Of^T_\bullet\xrightarrow{\Ug}\Ef^T_\bullet$ has a left adjoint and the adjunction is monadic. 

In particular, from proposition~\ref{prop: Q-cart}, we have that $\Of^{\hat{T}}_\bullet\xrightarrow{\Ug}\Qf^{\hat{T}}_\bullet$ is monadic. 

\medskip 

Furthermore $\Ug$ is finitary \cite[appendix G page 353]{Lei04}. 
\end{lemma}
This second lemma follows from a direct construction contained in~\cite[appendix G page 352]{Lei04}. 
\begin{lemma}
The forgetful functor $\Qf^{\hat{T},\kappa}_\bullet\xrightarrow{\Wg}\Qf^{\hat{T}}_\bullet$ is finitary, has a left adjoint and the adjunction is monadic. 
\end{lemma}

This third lemma is dealt with in~\cite[appendix G page 352]{Lei04} using that $\Qf^{\hat{T}}_\bullet$ is a pre-sheaf category. 
\begin{lemma} 
The category $\Qf^{\hat{T}}_\bullet$ is locally finitely presentable.  
\end{lemma}

\begin{lemma}\cite[proposition 27.1]{Ke80} \label{lem: kelly}
In the (big) category $\Cf$ of functors between 1-categories, consider the co-span $\As\xrightarrow{\Ag}\Xs\xleftarrow{\Bg}\Bs$ and assume that: 
\begin{itemize}
\item[$\blacktriangleright$] 
$\Xs$ is a locally finitely presentable 1-category, 
\item[$\blacktriangleright$] 
the functors $\Ag$ and $\Bg$ are both finitary and monadic,  
\item[$\blacktriangleright$] 
the span $\As\xleftarrow{\hat{\Ag}}\Ws\xrightarrow{\hat{\Bg}}\Bs$ is a strict pull-back of of the diagram  $\As\xrightarrow{\Ag}\Xs\xleftarrow{\Bg}\Bs$
\begin{equation*}
\xymatrix{
\Ws \ar[rrrr]^{\hat{\Ag}} \ar[d]_{\hat{\Bg}} & & & & \Bs \ar[d]^{\Bg}
\\
\As \ar[rrrr]_{\Ag} & & & & \Xs.
}
\end{equation*}.
\end{itemize}
The functor $\Ws\xrightarrow{\Ag\circ\hat{\Ag}=\Bg\circ\hat{\Bg}}\Xs$ is also monadic. 

\medskip 

As a consequence~\cite[corollary~G.1.2]{Lei04} the category $\Ws$ has an initial object. 
\end{lemma}

Theorem~\ref{prop: L} follows applying the previous lemmata in the case of the strict pull-back of forgetful functors: 
\begin{equation*}
\xymatrix{
\Of^{\hat{T},\kappa}_\bullet \ar[rrrr]^{\hat{\Ug}_\kappa} \ar[d]_{\hat{\Ug}_\Of} & & & & \Of^{\hat{T}}_\bullet \ar[d]^{\Ug_\Of}
\\
\Qf^{\hat{T},\kappa}_\bullet \ar[rrrr]_{\Ug_\kappa} & & & & \Qf^{\hat{T}}_\bullet.
}
\end{equation*}

\medskip 

Here is the basic notion of weak $\omega$-category from~\cite[definition~III.9.2.3]{Lei04}. 
\begin{definition}\label{def: w-omega}
A \emph{weak $\omega$-category} is an algebra for any monad $L$, that is an initial object in $\Of^{\hat{T},\kappa}_\bullet$. 
\end{definition}

\subsubsection{Examples} \label{sec: ex-T}

The basic way to provide examples of weak globular $\omega$-categories in Leinster's approach is described in~\cite[Example 9.2.4]{Lei04}: any contracted $\hat{T}$-operad $(P,\kappa,\mu,\eta)\in\Of^{\hat{T},\kappa}_\bullet$ (this means for every choice of a contraction on a monad over $\bullet$ in the bicategory $\Qf_{\hat{T}}$) induces (since $L$ is initial) a unique morphism $L\xrightarrow{!}P$ in $\Of^{\hat{T},\kappa}$ that produces a strict-functor from algebras over $P$ to algebras over $L$ (that are, by definition, weak globular $\omega$-categories). 
Hence an example of weak globular $\omega$-category is obtained as soon as we are given:
\begin{itemize}
\item[$\blacktriangleright$]
$(P,\pi,\kappa,\mu,\eta)$ a contracted $\hat{T}$-operad (that is just a contraction on a monad in $\Qf_{\hat{T}}$), 
\item[$\blacktriangleright$] 
an algebra $X\in \Qf_{\hat{T}}$ for the contracted $\hat{T}$-operad $P$ (that is just an algebra for the previous monad).  
\end{itemize}

\medskip 

Notable examples of weak globular $\omega$-categories in Leinster's approach include:
\begin{itemize}
\item[$\blacktriangleright$] 
strict globular $\omega$-categories (see example~\cite[9.2.5]{Lei04}): these are just algebras for the terminal contracted $\hat{T}$-operad $\hat{T}(\bullet)\in\Of^{\hat{T},\kappa}$, \footnote{
The terminal object in $\Of^{\hat{T},\kappa}_\bullet$ is $\bullet\xleftarrow{!}\hat{T}(\bullet)\xrightarrow{\iota(\hat{T}(\bullet))}\hat{T}(\bullet)$, we will describe in more detail this point for our new monad $\hat{T}^\star$ in remark~\ref{rem: q-cong}. 
} 
\item[$\blacktriangleright$] 
the globular $\omega$-homotopy groupoid $\Pi_\omega(S)$ of a topological space $S$: in~\cite[example 9.2.7]{Lei04} it is shown that, for every topological space $S$, the $\omega$-homotopy groupoid $\Pi_\omega(S)$ is an algebra for a certain contracted-$\hat{T}$-operad. 
\end{itemize}

\section{Involutive Leinster Weak Globular $\omega$-categories} \label{sec: invo}

In the present section we proceed to study how to introduce weak involutions on a weak globular $\omega$-category. 

\medskip 

We start in subsection~\ref{sec: st-inv-omega} from involutions on strict globular $\omega$-categories, as defined in our previous paper~\cite{BeBe17} that extended to the $\omega$-category case the notion of fully involutive strict globular $n$-category introduced in~\cite[section~4]{BCLS20}. 

\medskip 

Subsection~\ref{sec: w-inv-omega} contains our tentative definition of involutive weak globular $\omega$-category based on an initial generalized operad $L^\star$ in the bicategory $\Qf_{\hat{T}^\star}$ induced by the free involutive globular omega category monad $\hat{T}^\star$ over the same Cartesian bicategory $\Qf$ of globular $\omega$-sets. 

\subsection{Involutive Strict Globular $\omega$-categories} \label{sec: st-inv-omega}

We quickly recall the basic definitions and contructions of free strict involutive globular $\omega$-categories, closely following our previous paper~\cite{BeBe17}. 
\begin{definition}
A (reflexive) $\omega$-quiver $Q:=Q^0 \overset{s^0}{\underset{t^0}{\leftleftarrows}} Q^1 \overset{s^1}{\underset{t^1}{\leftleftarrows}} \cdots
\overset{s^{n-2}}{\underset{t^{n-2}}{\leftleftarrows}} Q^{n-1}\overset{s^{n-1}}{\underset{t^{n-1}}{\leftleftarrows}} Q^n \overset{s^n}{\underset{t^n}{\leftleftarrows}} \cdots$ is \emph{fully self-dual} if it is equipped with a family $(\gamma_q)_{q\in\NN}$ endomorphisms $Q\xrightarrow{\gamma_q}Q$, for all $q\in\NN$, such that: 
\begin{itemize}
\item[$\blacktriangleright$] 
for all $q\in \NN$, $\gamma_q$ is $q$-contravariant: \quad 
$\gamma^q_q \circ s^q=t^q\circ \gamma_q^{q+1}$,  \quad 
$\gamma^q_q \circ t^q=s^q\circ \gamma_q^{q+1}$,
\item[$\blacktriangleright$]
for all $p,q\in\NN$ such that $p\neq q$, $\gamma_q$ is $p$-covariant: 
\quad 
$\gamma^p_q \circ s^p=s^p\circ \gamma_q^{p+1}$,  \quad 
$\gamma^p_q \circ t^p=t^p\circ \gamma_q^{p+1}$. 
\end{itemize} 
A (reflexive) globular $\omega$-set (respectively a (reflexive) globular $\omega$-magma) is said to be self-dual whenever its underlying $\omega$-quiver is self-dual in the previous sense.\footnote{Here we interpret $\circ, \iota, \gamma$ as partially defined binary, nullary, unary operations, and we do not impose any algebraic axiom for magmas.} 

\medskip 

A strict globular $\omega$-category $(\Cs,\circ,\iota)$ that is also a fully self-dual $\omega$-quiver $(\Cs,*)$ is a \emph{fully involutive strict globular $\omega$-category} if its family $(*_q)_{q\in\NN}$ of self-duality endomorphisms (here denoted by $x\mapsto x^{*^k_q}$, for $x\in\Cs^k$ and $k\in\NN$) satisfies the following algebraic axioms:
\begin{itemize}
\item[$\blacktriangleright$] 
(\emph{$q$-contravariance}): \quad 
$(x\circ^k_qy)^{*^k_q}=y^{*^k_q}\circ^k_q x^{*^k_q}$, \ $\forall k,q\in\NN$ such that $0\leq q<k$ and $(x,y)\in\Cs^k\times_{\Cs^q}\Cs^k$, 
\\
\phantom{(\emph{$q$-contravariance}):} \quad 
$(x\circ^k_qy)^{*^k_p}=x^{*^k_p}\circ^k_q y^{*^k_p}$, \ $\forall k,p,q\in\NN$ such that $0\leq q\neq p <k$ and $(x,y)\in\Cs^k\times_{\Cs^q}\Cs^k$, 
\item[$\blacktriangleright$] 
(\emph{unitality}): \quad $(\iota^k(x))^{*^{k+1}_q}=\iota^k(x^{*^k_q})$, \quad 
for all $k,q\in\NN$ and for all $x\in\Cs^k$, 
\item[$\blacktriangleright$]
(\emph{involutivity}): \quad $(x^{*^k_q})^{*^k_q}=x$, \quad for all $k,q\in\NN$, and all $x\in\Cs^k$, 
\item[$\blacktriangleright$] 
(\emph{commutativity}): \quad $(x^{*^k_q})^{*^k_p}=(x^{*^k_p})^{*^k_q}$, \quad 
for all $k,p,q\in\NN$, and all $x\in\Cs^k$, 
\item[$\blacktriangleright$]
(\emph{$q$-grounding}): \quad $x^{*^k_q}=x$, \quad for all $k,q\in\NN$ such that $k\leq q$ and all $x\in\Cs^k$. 
\end{itemize}

A \emph{self-dual morphism} between self-dual $\omega$-quivers (and similarly between self-dual (reflexive) globular $\omega$-sets or self-dual (reflexive) globular $\omega$-magmas) is a morphism $(Q,\gamma)\xrightarrow{\phi}(\hat{Q},\hat{\gamma})$ such that $\hat{\gamma}\circ\phi=\phi\circ\gamma$. 

\medskip 

An \emph{involutive $\omega$-functor} $\Cs\xrightarrow{\phi}\hat{\Cs}$ between fully involutive strict globular $\omega$-categories is just a self-dual morphism for the underlying self-dual globular $\omega$-magmas. 
\end{definition}

\medskip 

This involutive version of proposition~\ref{prop: free-omega-cat} was proved in~\cite[propositions 3.1, 3.2]{BeBe17}.
\begin{proposition}\label{prop: free-inv-omega-cat}
Let $\Cf^\star$ denote the strict 1-category of covariant involutive $\omega$-functors between strict globular fully involutive $\omega$-categories.  

\medskip 

For any globular $\omega$-set $Q$ in $\Qf$, a \emph{free strict involutive globular $\omega$-category over $Q$} is a morphism of globular $\omega$-sets $Q\xrightarrow{\eta^\star_Q}\Ug^\star(\Cs)$, into the underlying globular $\omega$-set $\Ug^\star(\Cs)$ of a strict involutive globular $\omega$-category $\Cs$, satisfying the following universal factorization property: 
for any morphism of globular $\omega$-sets $Q\xrightarrow{\phi}\Ug^\star(\hat{\Cs})$, into the underlying globular $\omega$-set $\Ug^\star(\hat{\Cs})$ of a strict involutive globular $\omega$-category $\hat{\Cs}$, there exists a unique involutive $\omega$-functor $\Cs\xrightarrow{\hat{\phi}}\hat{\Cs}$ of fully involutive strict globular $\omega$-categories such that $\phi=\hat{\phi}\circ\eta^\star_Q$. 

\medskip 

The \emph{forgeful functor} $\Cf^\star\xrightarrow{\Ug^\star}\Qf$ (forgetting compositions, identities and involutions of objects in $\Cf^\star$) admits a left-adjoint $\Fg^\star\dashv \Ug^\star$ \emph{free strict involutive globular $\omega$-category functor} $\Cs^\star\xleftarrow{\Fg^\star}\Qf$ that is uniquely determined via a specific construction of free involutive strict globular $\omega$-category $Q\xrightarrow{\eta^\star_Q}\Ug^\star(\Cs)$ over a globular $\omega$-set $Q$. 
\end{proposition}
\begin{proof}
For convenience of the reader, we recall an explicit construction of the free strict globular involutive  $\omega$-category $Q\xrightarrow{\eta^\star_Q}\Fg^\star(Q)$ over a given globular $\omega$-set $Q$. 
\begin{itemize}
\item[a.] 
We first produce the free globular self-dual relfexive $\omega$-magma $Q\xrightarrow{\zeta_Q}\Mg(Q)$, over the globular $\omega$-set $Q$, with respect to all the partial binary compositions, the unary involutions and the nullary identities involved in the definition of strict globular involutive $\omega$-category. 
\item[b.]
Then we consider in $\Mg(Q)$ the smallest congruence relation of globular self-dual reflexive $\omega$-magmas $\Xi\subset \Mg(Q)\times \Mg(Q)$ containing all the pairs of terms appearing into the algebraic axioms that are involved in the definition of strict globular involutive $\omega$-category. 
\item[c.]
Finally, considering the quotient morphism $\Mg(Q)\xrightarrow{\varpi}\Mg(Q)/\Xi$ of globular self-dual reflexive $\omega$-magmas and the map of globular $\omega$-sets $Q\xrightarrow{\varpi\circ\zeta_Q}\Mg(Q)/\Xi$, we notice that $\Fg^\star(Q):=\Mg(Q)/\Xi$ is a strict involutive globular \hbox{$\omega$-category} and that $\eta^\star_Q:=\varpi\circ\zeta_Q$ satisfies the universal factorization property for free involutive globular $\omega$-categories over $Q$. 
\end{itemize}

\medskip 

a. 

The free globular self-dual reflexive $\omega$-magma $Q\xrightarrow{\zeta_Q}\Mg(Q)$ over the globular $\omega$-set $Q$ is obtained by a recursive definition. 
We first provide a 1-quiver $\Mg(Q)^0\leftleftarrows\Mg(Q)^1$. 

\medskip 

Starting from $Q^0$, we construct $\Mg(Q)^0:=Q^0$. 
We then introduce $Q^0_\iota:=\{(x,\iota^0) \ | \ x\in Q^0\}$ (that is a disjoint copy of $Q^0$ representing freely added identities of elements in $Q^0$) and $\Mg(Q)^1[1]^0:=Q^1\uplus Q^0_\iota$ with source/target given by $s^0_{\Mg(Q)}(x):=s^0_Q(x)$, $t^0_{\Mg(Q)}(x):=t^0_Q(x)$, for all $x\in Q^1$ and $s^0_{\Mg(Q)}(x,\iota):=x=:t^0_{\Mg(Q)}(x,\iota)$, for all $x\in Q^0$. 
If $\Mg(Q)^1[1]^j$ has been already constructed, we define $\Mg(Q)^1[1]^{j+1}:=\{(x,\gamma_0) \ | \ x\in\Mg(Q)^1[1]^j\}$ (that introduces freely added self-dualities of elements in $\Mg(Q)^1[1]^j$) with source $s^0_{\Mg(Q)}(x,\gamma_0):=t^0_{\Mg(Q)}(x)$ and target $t^0_{\Mg(Q)}(x,\gamma_0):=s^0_{\Mg(Q)}(x)$; defining $\Mg(Q)^1[1]:=\bigcup_{j=1}^{+\infty}\Mg^1(Q)[1]^j$, we get a 1-quiver $\Mg(Q)^0\leftleftarrows \Mg(Q)^1[1]$. 

We proceed to define $\Mg(Q)^1[2]^0:=\Big\{(x,0,y) \ | \ (x,y)\in\Mg(Q)^1[1]\times_{\Mg^0(Q)}\Mg(Q)^1[1] \Big\}$ (whose elements represent freely added compositions) with sources and targets 
$s_{\Mg(Q)}^0(x,0,y):=s_{\Mg(Q)}^0(y)$ and $t_{\Mg(Q)}^0(x,0,y):=t_{\Mg(Q)}^0(x)$.
Exactly as done before, we define $\Mg(Q)^1[2]:=\bigcup_{j=1}^{+\infty}\Mg(Q)^1[2]^j$, that freely introduces arbitrary iterations of $\gamma_0$-self-dualities of the 1-cells in $\Mg(Q)^1[2]^0$, obtaining a new 1-quiver $\Mg(Q)^0\leftleftarrows\Mg(Q)^1[1]\cup\Mg(Q)^1[2]$. 

Supposing, by recursion, that the 1-quiver $\Mg(Q)^0\leftleftarrows\bigcup_{n=1}^k\Mg(Q)^1[n]$ is already defined, we extend it freely adding compositions $\Mg(Q)^1[k+1]^0:=\{(x,0,y) \ | \ (x,y)\in\Mg(Q)^1[i]\times_{\Mg(Q)^0}\Mg(Q)^1[j], \ i+j=k+1 \}$, with sources/targets $s_{\Mg(Q)}^0(x,0,y):=s_{\Mg(Q)}^0(y)$, $t_{\Mg(Q)}^0(x,0,y):=t_{\Mg(Q)}^0(x)$, and freely adding arbitrary iterations of $\gamma_0$-self-dualities to get $\Mg(Q)^1[k+1]:=\bigcup_{j=1}^{+\infty}\Mg(Q)^1[k]^j$ and the 1-quiver $\Mg(Q)^0\leftleftarrows\Mg(Q)^1$, where we have $\Mg(Q)^1:=\bigcup_{k=1}^{+\infty}\Mg(Q)^1[k]$. 

\medskip 

The previous 1-quiver, becomes reflexive defining $\iota^0:\Mg(Q)^0\to\Mg(Q^1)$ by $x\mapsto (x,\iota^0)$; it further becomes self-dual defining $\gamma_0:\Mg(Q)^1\to\Mg(Q)^1$ by $x\mapsto (x,\gamma_0)$; a binary composition $\circ^1_0:\Mg(Q)^1\times_{\Mg(Q)^0}\Mg(Q^1)\to\Mg(Q)^1$ is also present defining  $\circ^1_0:(x,y)\mapsto (x,0,y)$ and hence we have a reflexive self-dual (globular) 1-magma.  

\medskip 

Supposing by recursion that a reflexive self-dual globular magma for the $n$-quiver $\Mg(Q)^0\leftleftarrows\cdots\leftleftarrows\Mg(Q)^n$ has been already defined, we will obtain a reflexive self-dual globular magma $\Mg(Q)^0\leftleftarrows\cdots\leftleftarrows\Mg(Q)^n\leftleftarrows\Mg(Q)^{n+1}$. 

\medskip 

We start with  $\Mg(Q)^{n+1}[0]^0:=Q^{n+1}\cup\Mg(Q)^n_\iota$, where $\Mg(Q)^n_\iota:=\{ (x,\iota^n) \ | \ x\in\Mg(Q)^n\}$ with sources and targets $s_{\Mg(Q)}(x,\iota^n):=x=:t_{\Mg(Q)}(x,\iota^n)$; we then introduce $\Mg(Q)^{n+1}[0]^{j+1}:=\{(x,\gamma_q) \ | \ x\in\Mg(Q)^{n+1}[0]^j, \ q=0,\dots,n\}$, with sources/targets $s_{\Mg(Q)}/t_{\Mg(Q)}(x,\gamma_q)=s_{\Mg(Q)}/t_{\Mg(Q)}(x)$, if $q< n$, and, for $q=n$,  $s_{\Mg(Q)}(x,\gamma_n)=t_{\Mg(Q)}(x)$, $t_{\Mg(Q)}(x,\gamma_n)=s_{\Mg(Q)}(x)$; and we get $\Mg(Q)^{n+1}[0]:=\bigcup_{j=1}^{+\infty}\Mg(Q)^{n+1}[0]^j$. 
Next, assuming for recursion that $\Mg(Q)^{n+1}[k]$ (with its source and target maps) have been already defined, we introduce free depth-$p$ compositions $\Mg(Q)^{n+1}[k+1]^0:=\Big\{(x,p,y) \ | \ p=0,\dots,n, \ (x,y)\in\Mg(Q)^{n+1}[i]\times_{\Mg(Q)^p}\Mg(Q)^{n+1}[j], \ i+j=k+1 \Big\}$, with source targets 
$s_{\Mg(Q)}(x,p,y):=(s_{\Mg(Q)}(x),p,s_{\Mg(Q)}(y))$ and $t_{\Mg(Q)}(x,p,y):=(t_{\Mg(Q)}(x),p,t_{\Mg(Q)}(y))$, whenever $0\leq p<n$, and otherwise $s_{\Mg(Q)}(x,n,y):=s_{\Mg(Q)}(y)$ and $t_{\Mg(Q)}(x,p,y):=t_{\Mg(Q)}(x)$. 
Subsequently we introduce free iterated self-duals of the previous $(n+1)$-cells by  $\Mg(Q)^{n+1}[k+1]:=\bigcup_{j=0}^{+\infty}\Mg(Q)^{n+1}[k+1]^j$ where, as above, we consider   $\Mg(Q)^{n+1}[k=1]^{j+1}:=\Big\{(x,\gamma_q) \ | \ q=0,\dots,n, \ x\in\Mg(Q)^{n+1}[k+1]^{j} \Big\}$ with similarly defined source/target maps. 
Finally we set $\Mg(Q)^{n+1}:=\bigcup_{k=0}^{+\infty}\Mg(Q)^{n+1}[k]$ obtaining the required globular $(n+1)$-quiver that is reflexive, with the map $\iota^n(x):=(x,\iota^n)$, for $x\in\Mg(Q)^n$; self-dual with the maps $x\mapsto (x,\gamma_q)$, for all $x\in\Mg(Q)^{n+1}$ and all $q\in\{0,\dots,n\}$; and a magma for the partial compositions $x\circ^{n+1}_py:=(x,p,y)$, for all $(x,y)\in\Mg(Q)^{n+1}\times_{\Mg(Q)^p}\Mg(Q)^{n+1}$ and all $p\in\{0,\dots,n\}$. 
 
\medskip 

The recursive construction of the globular reflexive self-dual $\omega$-magma $\Mg(Q)^0\leftleftarrows\cdots\leftleftarrows\Mg(Q)^n\leftleftarrows\cdots$ is now completed and we further produce the morphism $Q\xrightarrow{\eta_Q}\Mg(Q)$ of globular $\omega$-sets by the inclusion of $Q^n$ into $\Mg(Q)^n[0]^0\subset \Mg(Q)^n$, for all $n\in\NN$. 

\medskip 

We only need to check the universal factorization property for the morphism $Q\xrightarrow{\eta_Q}\Mg(Q)$. 
Suppose that $Q\xrightarrow{\phi}\hat{M}$ is a morphism of reflexive self-dual globular $\omega$-sets into a reflexive self-dual globular $\omega$-magma. Any grade-preserving map $\Mg(Q)\xrightarrow{\hat{\phi}}\hat{M}$ such that $\phi=\hat{\phi}\circ\eta_Q$ must necessarily satisfy $x\mapsto \phi(x)$, for all $x\in Q^n\subset\Mg(Q)^n$. Using the fact the $\hat{\phi}$ must be a morphism of reflexive self-dual globular $\omega$-magmas, by induction, we obtain that, for all $n\in\NN$,  $\hat{\phi}(x,\iota^n)=\hat{\iota}_{\hat{M}}^n(\phi(x))$, for all $x\in \Mg(Q)^{n}$, $\hat{\phi}(x,p,y)=\phi(x)\hat{\circ}^n_p\phi(y)$, for all $(x,y)\in\Mg(Q)^n\times_{\Mg(Q)^p}\Mg(Q)^n$ and $0\leq p<n$, and $\hat{\phi}(x,\gamma_q)=\phi(x)^{*^n_q}$, for all $x\in\Mg(Q)^n$ and $0\leq q<n$. 
This uniquely defined map $\Mg(Q)\xrightarrow{\hat{\phi}}\hat{M}$ is a morphism of relfexive self-dual globular $\omega$-magmas as required.

\bigskip 

b. 

From~\cite[section 3.2]{BeBe17} we recall that, given globular $\omega$-sets $(Q,s_Q,t_Q)$, $(\hat{Q},s_{\hat{Q}},t_{\hat{Q}})$ the \textit{Cartesian product of globular $\omega$-sets} is the globular $\omega$-set $(Q\times \hat{Q},s_{Q\times\hat{Q}},t_{Q\times\hat{Q}})$ defined, for all $n\in\NN$, as $(Q\times \hat{Q})^n:=Q^n\times \hat{Q}^n$, with source/targets $s^n_{Q\times\hat{Q}}:=(s^n_Q,s^n_{\hat{Q}})$ and $t^n_{Q\times\hat{Q}}:=(t^n_Q,t^n_{\hat{Q}})$ acting componentwise; and that whenever $(Q,s_Q,t_Q,\iota_Q,\gamma_Q,\circ_Q)$, $(\hat{Q},s_{\hat{Q}},t_{\hat{Q}},\iota_{\hat{Q}},\gamma_{\hat{Q}},\circ_{\hat{Q}})$ are globular self-dual reflexive $\omega$-magmas, also their Cartesian product $Q\times\hat{Q}$ is a globular self-dual reflexive $\omega$-magma with the componentwise defined nullary $\iota^{n}_{Q\times\hat{Q}}:=(\iota^n_Q,\iota^n_{\hat{Q}})$, unary $\gamma^{(Q\times\hat{Q})^n}_q:=(\gamma^{Q^n}_q,\gamma_q^{\hat{Q}^n})$ and binary 
$\circ_p^{(Q\times\hat{Q})^n}:=(\circ_p^{Q^n},\circ_p^{\hat{Q}^n})$ operations. 

\medskip

We also recall that a \textit{congruence} $X$ in a globular self-dual reflexive $\omega$-magma $M$ is a globular self-dual reflexive $\omega$-magma $X$ such that $X^n\subset M^n\times M^n$, for all $n\in\NN$, in such a way that the inclusion $X\xrightarrow{\phi}M\times M$ is a morphism of globular self-dual reflexive $\omega$-magmas. 

\medskip 

Inside the free globular self-dual reflexive $\omega$-magma $\Mg(Q)$ over $Q$ constructed in a.~above we consider the congruence $\Xi\subset\Mg(Q)\times\Mg(Q)$ generated \footnote{
Since intersection of congruences is a congruence, $\Xi$ is just the intersection of all the conguences containing the given pairs.
} 
by the union of all of the following families of pairs: 
\begin{align*}
&\Big\{(x\circ^n_p(y\circ^n_pz),\ (x\circ^n_py)\circ^n_p z) 
\quad \big| \quad   
n>p\in\NN, \ (x,y,z) \in \Mg(Q)^n\times_{\Mg(Q)^p}\Mg(Q)^n\times_{\Mg(Q)^p}\Mg(Q)^n \Big\} 
\\
&\Big\{((\iota^{n-1}\circ\cdots\iota^p\circ t^p\circ\cdots\circ t^{n-1} (x))\circ^n_p x,\ x) 
\quad \big| \quad 
n>p \in \NN, \  x\in \Mg(Q)^n \Big\}  
\\ 
&\Big\{(x,\ x\circ^n_p (\iota^{n-1}\circ\cdots\iota^p\circ s^p\circ\cdots\circ s^{n-1} (x))) 
\quad \big| \quad 
n>p \in \NN,\  x\in \Mg(Q)^n \Big\} 
\\ 
&\Big\{(\iota^n(x)\circ^{n+1}_p\iota^n(y),\ \iota^n(x\circ^n_p y)) 
\quad \big| \quad 
n>p\in \NN, \ (x,y)\in \Mg(Q)^n\times_{\Mg(Q)^q}\Mg(Q)^n \Big\}
\\ 
&\begin{aligned}
\Big\{((x\circ^n_p y)\circ^n_q(z\circ^n_p w),\ \  &(x\circ^n_q z)\circ^n_p(y\circ^n_q w)) 
\quad \big| \quad 
n>p,q\in\NN, \ 
\\
&(x,y),(z,w)\in \Mg(Q)^n\times_{\Mg(Q)^p}\Mg(Q)^n, \ (x,z),(y,w)\in \Mg(Q)^n\times_{\Mg(Q)^q}\Mg(Q)^n\Big\} 
\end{aligned}
\\ 
&\Big\{(\gamma^n_q(\gamma^n_q(x)),\ x)
\quad \big| \quad  
n>q\in\NN, \ x\in \Mg(Q)^n 
\Big\} 
\quad \quad 
\Big\{
(\gamma_q^n(\gamma_p^n(x)), \ \gamma_p^n(\gamma_q^n(x))) 
\quad \big| \quad 
n>q,p\in\NN, \ x\in \Mg(Q)^n \Big\} 
\\ 
&\Big\{
(\gamma^n_q(x\circ^n_p y), \ \gamma^n_q(x)\circ^n_p \gamma^n_q(y) )
\quad \big| \quad 
n>p\neq q\in\NN, \ (x,y)\in \Mg(Q)^n\times_{\Mg(Q)^p}\Mg(Q)^n \Big\} 
\\ 
&\Big\{
(\gamma^n_q(x\circ^n_p y), \ \gamma^n_q(y)\circ^n_p \gamma^n_q(x)) 
\quad \big| \quad 
n>p=q\in\NN, \ (x,y)\in \Mg(Q)^n\times_{\Mg(Q)^p}\Mg(Q)^n \Big\} 
\\ 
&\Big\{
(\iota^{n}(\gamma^n_q(x)), \ \gamma^{n+1}_q(\iota^n_q(x))) 
\quad \big| \quad
n>q\in\NN, \  
x\in \Mg^n(Q) \Big\} 
\quad \quad 
\Big\{
(\gamma^n_q(x), \ x) 
\quad \big| \quad 
n\leq q\in\NN, \ x\in \Mg(Q)^n 
\Big\}. 
\end{align*}

\medskip 

c. 

As every \textit{quotient of a globular self-dual relfexive $\omega$-magma by a globular $\omega$-congruence}, the quotient $\Mg(Q)/\Xi$ is a globular $\omega$-set with $(\Mg(Q)/\Xi)^n:=\Mg(Q)^n/\Xi^n=\Big\{[x]_{\Xi^n} \ | \ x\in\Mg(Q)^n \Big\}$ and $s^n_{\Mg(Q)/\Xi}([x]_{\Xi^{n+1}}):=[s^n_{\Mg(Q)}(x)]_{\Xi^{n}}$,  
$t^n_{\Mg(Q)/\Xi}([x]_{\Xi^{n+1}}):=[t^n_{\Mg(Q)}(x)]_{\Xi^{n}}$, for $n\in\NN$ and $x\in\Mg(Q)^{n+1}$; and $\Mg(Q)/\Xi$ is actually a globular self-dual reflexive $\omega$-magma with well-defined compositions given by $[x]_{\Xi^n}\circ^{(\Mg(Q)/\Xi)^n}_p[y]_{\Xi^n}:=[x\circ^{\Mg(Q)^n}_p y]_{\Xi^n}$, 
self-dualties $\gamma_q([x]_{\Xi^n}):=[\gamma_q(x)]_{\Xi^n}$, and reflexive maps 
$\iota^n_{\Mg(Q)/\Xi}([x]_{\Xi^n}):=[\iota^n_{\Mg(Q)}(x)]_{\Xi^{n+1}}$, 
for $n>p\in\NN$, $q\in\NN$, $x,y\in\Mg(Q)^n$. 
Funrthermore, since all the algebraic axioms for strict globular involutive $\omega$-category are already included in $\Xi$, we have that the quotient $\Mg(Q)/\Xi$ is actually a strict globular involutive $\omega$-category and the quotient morphism $\varpi:\Mg(Q)\to\Mg(Q)/\Xi$, defined as $\varpi^n(x):=[x]_{\Xi^n}$, for all $n\in\NN$ and $x\in \Mg(Q)^n$, is a morphism of globular self-dual reflexive $\omega$-sets. 
 
\medskip

We only need to check that the morphism of globular $\omega$-sets $Q\xrightarrow{\eta^\star_Q}\Mg(Q)/\Xi$, with $\eta^\star_Q:=\varpi\circ\zeta_Q$, into the globular involutive $\omega$-category $\Mg(Q)/\Xi$,  satisfies the universal factorization property. 
For any other morphism $Q\xrightarrow{\phi}\Cs$ of globular self-dual reflexive $\omega$-sets, into a strict globular involutive $\omega$-category $\Cs$, since $Q\xrightarrow{\zeta_Q}\Mg(Q)$ is a free globular self-dual reflexive $\omega$-magma, there exists a unique morphism $\Mg(Q)\xrightarrow{\cj{\phi}}\Cs$ of globular self-dual reflexive $\omega$-magmas such that $\phi=\cj{\phi}\circ\zeta_Q$. 
Consider now the \textit{kernel globular $\omega$-congruence} $\Xi_{\cj{\phi}}$ induced by the morphism $\cj{\phi}$ of $\omega$-magmas: for all $n\in\NN$, we have $\Xi_{\cj{\phi}}^n:=\Big\{(x,y)\in\Mg(Q)^n\times\Mg(Q)^n \ | \  \cj{\phi}(x)=\cj{\phi}(y)\Big\}$ and, since all the axioms of strict globular involutive $\omega$-category are already satisfied in $\Cs$, we have $\Xi\subset\Xi_{\cj{\phi}}$ and hence $\Mg(Q)/\Xi_{\cj{\phi}}$ is already a strict globular involutive $\omega$-category, furthermore the assignment $\tilde{\phi}:[x]_{\Xi_{\cj{\phi}}}\mapsto \cj{\phi}(x)$ is a well-defined covariant involutive $\omega$-functor $\Mg(Q)/\Xi_{\cj{\phi}}\xrightarrow{\tilde{\phi}}\Cs$, that is actually the unique such that $\cj{\phi}=\tilde{\phi}\circ\varpi_\phi$, where $\Mg(Q)\xrightarrow{\varpi_\phi}\Mg(Q)/\Xi_{\cj{\phi}}$ is the quotient morphism of globular self-dual reflexive $\omega$-magmas defined as usual by $\varpi_\phi:x\mapsto [x]_{\Xi_{\cj{\phi}}}$, for $x\in\Mg(Q)$. From $\Xi\subset\Xi_{\cj{\phi}}$, we obtain a unique involutive $\omega$-functor $\Mg(Q)/\Xi_{\cj{\phi}}\xrightarrow{\theta}\Mg(Q)/\Xi$, $\theta:[x]_\Xi\mapsto[x]_{\Xi_{\cj{\phi}}}$, for $x\in\Mg(Q)$, such that $\varpi_\phi\circ\varpi=\theta$. Defining $\hat{\phi}:=\tilde{\phi}\circ\theta$, we have that $\Mg(Q)/\Xi\xrightarrow{\hat{\phi}}\Cs$ is the unique involutive $\omega$-functor such that $\hat{\phi}\circ\eta^\star_Q=\tilde{\phi}\circ\theta\circ\varpi\circ\zeta_Q=\tilde{\phi}\circ\varpi_\phi \circ\eta_Q=\cj{\phi}\circ\eta_Q=\phi$. 
\end{proof}

\begin{remark}
The existence of algebras (with a given signature) that are free over a set is known; \footnote{
See for example section~6 in the n-Lab entry: \href{https://ncatlab.org/nlab/show/variety+of+algebras\#literature}{https://ncatlab.org/nlab/show/variety+of+algebras\#literature}. 
} 
propositions~\ref{prop: free-omega-cat} and~\ref{prop: free-inv-omega-cat} are essentially special cases of a general existence theorem for \textit{``$\omega$-algebras'' that are free over a globular $\omega$-set} that (for the $\omega$-globular setting) is vertically categorifying the case of algebras over sets. 
\xqed{\lrcorner}
\end{remark}

In paralled with corollary~\ref{cor: free-monads}, from proposition~\ref{prop: free-inv-omega-cat} we obtain a new monad in the involutive $\omega$-category case.
\begin{corollary}\label{cor: free-inv-monads}
On the 1-category $\Qf$ of morphisms of globular $\omega$-sets, we have the following  
\begin{itemize}
\item[$\blacktriangleright$]
\emph{free involutive strict globular $\omega$-category monad} $\hat{T}^\star:=\Ug^\star\circ\Fg^\star$.  
\end{itemize}
\end{corollary} 

\subsection{Involutive Weak Globular $\omega$-categories} \label{sec: w-inv-omega} 

The following is the ``involutive case'' version of proposition~\ref{prop: Q-cart}.  
\begin{proposition}
The 1-category $\Qf$ of small globular $\omega$-sets with morphisms of globular $\omega$-sets 
is Cartesian. 
The 1-category $\Cf^\star$ of small strict globular involutive $\omega$-categories with involutive $\omega$-functors is Cartesian. 
The forgetful 1-functor $\Cf^\star\xrightarrow{\Ug^\star}\Qf$ and the free strict globular involutive $\omega$-category 1-functor $\Cf^\star\xleftarrow{\Fg^\star}\Qf$ are Cartesian. 
The free strict globular involutive $\omega$-category monad $T^\star:=\Ug^\star\circ\Fg^\star$ is Cartesian. 
\end{proposition}
\begin{proof}
The Cartesianity of the 1-category $\Qf$ of globular $\omega$-sets is already known by proposition~\ref{prop: Q-cart}. Following the exposition in~\cite[section 3.2]{Be23}, we recall that, given a span $Q\xrightarrow{\phi}X\xleftarrow{\psi}R$ of globular $\omega$-sets in $\Qf$, a pull-back can be constructed via the co-span of globular $\omega$-sets $Q\xleftarrow{\hat{\psi}}Q\times_X R\xrightarrow{\hat{\phi}}R$ defined by 
$Q\times_XR:=\left((Q\times_X R)^n \overset{s^n}{\underset{t^n}{\leftleftarrows}} (Q\times_X R)^{n+1}\right)_{n\in\NN}$ where
$(Q\times_X R)^n:=Q^n\times_{X^n}R^n:=\{(q,r)\in Q^n\times R^n \ | \ \phi^n(q)=\psi^n(r)\}$, 
$s^n(q,r):=(s^n_Q(q),s^n_R(r))$ and $t^n(q,r):=(t^n_Q(q),t^n_R(r))$, with $\hat{\phi}^n(q,r):=r$ and $\hat{\psi}^n(q,r):=q$, for $n\in\NN$.

\bigskip 

To prove the Cartesianity of the 1-category $\Cf^\star$, from~\cite[section 3.3]{Be23}, we recall that given any co-span $\As\xrightarrow{\phi}\Xs\xleftarrow{\psi}\Bs$ in $\Cf^\star$, a pull-back can be constructed via the previous span $\As\xleftarrow{\hat{\psi}}
\As\times_\Xs\Bs\xrightarrow{\hat{\phi}}\Bs$ in $\Qf$, noting that the globular $\omega$-set $\As\times_\Xs\Bs$ becomes a strict involutive globular $\omega$-category, with componentwise 
compositions $(a_1,b_1)\circ^n_q(a_2,b_2):=(a_1\ {\circ_\As}^n_q \ a_2,b_1\ {\circ_\As}^n_q\ b_2)$,  
identities $\iota^n(a,b):=(\iota^n_\As(a),\iota_\Bs^n(b))$, involutions $(a,b)^{*_q}:=(a^{*_q^\As},b^{*_q^\Bs})$; and that the above-defined $\hat{\phi}$ and $\hat{\psi}$ turn out to be involutive covariant $\omega$-functors. 

\medskip 

From the previous explicit definitions of pull-backs in $\Qf$ and $\Cf^\star$, it follows that the forgetful functor $\Ug^\star$ is Cartesian, since it associates to the standard pull-back of strict involutive globular $\omega$-categories the standard pull-back of their underlying globular $\omega$-sets. 

\medskip 

In order to prove the Cartesianity of the free strict involutive globular $\omega$-category functor $\Fg^\star$, we simply notice that 
$\Fg^\star(Q)\xleftarrow{\Fg^\star(\hat{\psi})}\Fg^\star(Q\times_X R)\xrightarrow{\Fg^\star(\hat{\phi})}\Fg^\star(R)$ is canonically $\Cf^\star$-isomorphic to the standard $\Cf^\star$-pull-back $\Fg^\star(Q)\xleftarrow{\widehat{\Fg^\star(\psi)}}\Fg^\star(Q)\times_{\Fg^\star(X)}\Fg^\star(R)\xrightarrow{\widehat{\Fg^\star(\phi)}}\Fg^\star(R)$ of the co-span  $\Fg^\star(Q)\xrightarrow{\Fg^\star(\phi)}\Fg^\star(X)\xleftarrow{\Fg^\star(\psi)}\Fg^\star(R)$, via an involutive $\omega$-functor. 

\medskip 

The composition of Cartesian functors is Cartesian, hence the Cartesianity of the monad $T^\star:=\Ug^\star\circ\Fg^\star$. 

\medskip 

For the Cartesianity of the natural transformation $\eta^\star$, we must show that, for any morphism $Q_1\xrightarrow{\phi}Q_2$ of globular $\omega$-sets, the solid square commuting diagram below is a pull-back in $\Qf$; for this purpose, for any span $Q_2\xleftarrow{\beta}R\xrightarrow{\alpha}\hat{T}^\star(Q_1)$ such that $\hat{T}^\star(\phi)\circ\alpha=\eta^\star_{Q_2}\circ\beta$, we must see that there exists a unique morphism $R\xrightarrow{\theta}Q_1$ making commutative the two triangle diagrams below: $\alpha=\eta^\star_{Q_1}\circ\theta$, $\beta=\phi\circ\theta$.  
\begin{equation*}
\xymatrix{
R \ar@{.>}@/^/[drrrrr]^\alpha \ar@{.>}@/_/[ddrr]_{\beta} \ar@{.>}[rrd]|{\theta} & & & & & 
\\
& & Q_1 \ar[rrr]^{\eta^\star_{Q_1}} \ar[d]_{\phi} & & & \hat{T}^\star(Q_1) \ar[d]^{\hat{T}^\star(\phi)}
\\
& & Q_2 \ar[rrr]^{\eta^\star_{Q_2}} & & & \hat{T}^\star(Q_2)}
\end{equation*}
From the explicit construction of the free involutive globular $\omega$-category of a globular $\omega$-set recalled in proposition~\ref{prop: free-inv-omega-cat}, for all $x\in Q$, we have that $\eta^\star_{Q}(x)=\varpi\circ\zeta_{Q}(x)=[x]_\Xi$ is a singleton containing only $x$.  
Hence, for all $r\in R$, $\eta^\star_{Q_2}(\beta(r))=[x_2]_{\Xi_2}$ is a singleton in $\hat{T}^\star(Q_2)$, with $x_2\in Q_2$; since $\hat{T}^\star(\phi)$, as every morphism of globular $\omega$-sets, is ``degree-preserving'', there exists a unique element $\theta(r)\in Q_1^0$ such that $(\hat{T}^\star(\phi))([\theta(r)]_{\Xi_1})=[x_2]_{\Xi_2}$. Such $\theta$ satisfies our requirements. 

\medskip 

For the Cartesianity of the natural transformation $\mu_Q$, 
we must show that, for any morphism $Q_1\xrightarrow{\phi}Q_2$ of globular $\omega$-sets, the solid square commuting diagram below is a pull-back in $\Qf$; for this purpose, for any span $(\hat{T}^\star\circ\hat{T}^\star)(Q_2)\xleftarrow{\beta}R\xrightarrow{\alpha}Q_1$ such that $\phi\circ\alpha=\mu^\star_{Q_2}\circ\beta$, we must see that there exists a unique morphism $R\xrightarrow{\theta}(\hat{T}^\star\circ\hat{T}^\star)(Q_1)$ making commutative the two triangle diagrams below: $\alpha=\mu^\star_{Q_1}\circ\theta$ and $\beta=((\hat{T}^\star\circ\hat{T}^\star)(\phi))\circ\theta$.  
\begin{equation*}
\xymatrix{
R \ar@{.>}@/^/[drrrrr]^\alpha \ar@{.>}@/_/[ddrr]_{\beta} \ar@{.>}[rrd]|{\theta} & & & & & 
\\
& & (\hat{T}^\star\circ \hat{T}^\star)(Q_1) \ar[rrr]^{\mu^\star_{Q_1}}  \ar[d]^{(\hat{T}^\star\circ\hat{T}^\star)(\phi)} & & & \hat{T}^\star (Q_1) \ar[d]_{\hat{T}^\star(\phi)} 
\\
& & (\hat{T}^\star\circ\hat{T}^\star)(Q_2)\ar[rrr]^{\mu^\star_{Q_2}} & & & \hat{T}^\star (Q_2)
}
\end{equation*}
Since for all $r\in R$, we have $\alpha(r)\in \hat{T}^\star(Q_1)$, the only possible element in $(\hat{T}^\star\circ \hat{T}^\star)(Q_1)$ that maps, via $\mu^\star_{Q_1}$ to $\alpha(r)$, must necessarily be $(\alpha(r))$ and the assignment $r\mapsto (\alpha(r))$ is a morphism of globular $\omega$-sets satisfying the required conditions. 
\end{proof}

As direct application of proposition~\ref{prop: T-bicat} to the Cartesian monad $\hat{T}^\star$ on the Cartesian category $\Qf$ we obtain: 
\begin{corollary}
There is a bicategory $\Qf_{\hat{T}^\star}$. 
\end{corollary}

The notion of Leinster contraction in definition~\ref{def: L-contraction} remains unchanged and, as anticipated in remark~\ref{rem: w-inv-omega}, we have a parallel version of definition~\ref{def: contracted_T-coll} and proposition~\ref{def: g-c-T-operads} that reformulate as follows, for the case of the free involutive globular $\omega$-category monad $\hat{T}^\star$. 
\begin{definition}
Let $\bullet\in\Qf$ denote a terminal object in the category of globular $\omega$-sets. 

\medskip 

A \emph{globular $\hat{T}^\star$-collection} is a morphism $Q\xrightarrow{\pi} \hat{T}^\star(\bullet)$ in $\Qf$; 
a \emph{globular contracted $\hat{T}^\star$-collection}  
consists of a Leinster contraction $\kappa$ on a globular $\hat{T}^\star$-collection $\pi$: 
$\Par(\pi) \xrightarrow{\kappa} Q \xrightarrow{\pi} \hat{T}^\star(\bullet)$. 
\end{definition}
Similarly, using the bicategory $\Qf_{\hat{T}^\star}$, for the Cartesian monad $\hat{T}^\star$, definition~\ref{def: ET} already provides the notion of globular (contracted) $\hat{T}^\star$-operad over $\bullet$. 
For technical reasons, in the proof of the subsequent theorem~\ref{th: free-c-T*}, we actually need to introduce the following more general ``magma structure'' internal to $\Qf_{\hat{T}^\star}$. 
\begin{definition} 
A \emph{globular $\hat{T}^\star$-operadic magma} \footnote{
Of course, although the definition is here given in the special case of the bicategory $\Qf_{\hat{T}^\star}$, it remains perfectly valid when applied to an arbitrary bicategory $\Ef_T$, for a given Cartesian monad $T$; furthermore one can define \textit{multicategorical magmas over $E$} using a given object $E$ in place of a terminal $\bullet$ in $\Ef_T$. 
} 
$(M,\pi_M,\eta_M,\mu_M)$   
over $\bullet$ is a 1-cell $\bullet\xrightarrow{M}\bullet$ in the bicategory $\Qf_{\hat{T}^\star}$ (hence a globular $\hat{T}^\star$-collection $M\xrightarrow{\pi_M}\hat{T}^\star(\bullet)$) that is equipped with a unit 2-cell $\bullet\xrightarrow{\eta_M} M$ and a multiplication 2-cell $M\circ^1_0 M\xrightarrow{\mu_M}M$ 
as specified in the following commutative diagrams \footnote{ 
For the description of the notation required in the square diagrams on the right, refer to remark~\ref{rem: q-cong} below. 
} 
in the category $\Qf$ 
\begin{gather} \label{eq: op-mag}
\vcenter{
\xymatrix{\bullet \rtwocell^{\bullet}_{M}{\ \ \eta_M} & \bullet}
}
=
\vcenter{
\xymatrix{
& \bullet \ar[dr]^{!} \ar@{.>}[dd]^{\eta_M} \ar[dl]_{\eta^{\hat{T}^\star}_\bullet} &
\\
\hat{T}^\star(\bullet) & & \bullet
\\
& \ar[ul]^{\pi_M} M \ar[ru]_{!} &  
}
}
=
\vcenter{
\xymatrix{
\bullet \ar[d]_{!}\ar[rr]^{\eta_M} & & M \ar[d]_{\pi_M} 
\\
\bullet \ar[rr]^{\eta_{\hat{T}^\star(\bullet)}} & & \hat{T}^\star(\bullet) 
}
}
\\
\vcenter{
\xymatrix{
\bullet \rtwocell^{M\circ^1_0M}_{M}{\ \ \mu_M} & \bullet
}
}
=
\vcenter{
\xymatrix{
& \ar[dl]_{\mu^{\hat{T}^\star}_\bullet\circ\hat{T}^\star(\pi_M)\circ\pi_2\ \ } M\times_{\hat{T}^\star(\bullet)}\hat{T}^\star(M) \ar[dr]^{!} \ar@{.>}[dd]^{\mu_M} & 
\\
\hat{T}^\star(\bullet) & & \bullet
\\
& \ar[ul]^{\pi_M} M \ar[ur]_{!} & 
}
}
=
\vcenter{
\xymatrix{
M\circ^1_0 M \ar[rr]^{\mu_M} \ar[d]_{\pi_M\circ^2_0\pi_M} & & M \ar[d]^{\pi_M} 
\\
\hat{T}^\star(\bullet)\circ^1_0\hat{T}^\star(\bullet) \ar[rr]^{\mu_{\hat{T}^\star(\bullet)}} & & \hat{T}^\star(\bullet) 
}
}
\end{gather}
that are not necessarily required to satisfy the operadic axioms~\ref{eq: monadic-ax}. 

\medskip 

A \emph{globular $\hat{T}^\star$-operad} is a globular $\hat{T}^\star$-operadic magma with unit and multiplication that satisfy the monadic associativity and unitality axioms~\ref{eq: monadic-ax}. 
\end{definition} 

\begin{proposition} \label{prop: all-*-cat}
For any terminal object $\bullet\in\Qf^0$, we have the following categories: 
\begin{itemize}
\item[$\blacktriangleright$]
the \emph{category $\Qf^{\hat{T}^\star}_\bullet$ of globular $\hat{T}^\star$-collections over $\bullet$}, 
\item[$\blacktriangleright$] 
the \emph{category $\Qf^{\hat{T}^\star,\kappa}_\bullet$ of globular contracted $\hat{T}^\star$-collections over $\bullet$}, 
\item[$\blacktriangleright$] 
the \emph{category $\Mf^{\hat{T}^\star}_\bullet$ of globular $\hat{T}^\star$-operadic magmas} over $\bullet$, 
\item[$\blacktriangleright$]  
the \emph{category $\Of^{\hat{T}^\star}_\bullet$ of globular $\hat{T}^\star$-operads over $\bullet$},  
\item[$\blacktriangleright$] 
the \emph{category $\Mf^{\hat{T}^\star,\kappa}_\bullet$ of globular contracted $\hat{T}^\star$-operadic magmas over $\bullet$}, 
\item[$\blacktriangleright$] 
the \emph{category $\Of^{\hat{T}^\star,\kappa}_\bullet$ of globular contracted $\hat{T}^\star$-operads over $\bullet$}. 
\end{itemize}
There are commuting diagram (in the category of functors between 1-categories) of forgetful functors:
\begin{equation*}
\xymatrix{
\Of^{\hat{T}^\star,\kappa}_\bullet \ar[rrrr]^{\hat{\Ug}^\star_\kappa} \ar[d]_{\hat{\Ug}^\star_\Of} & & & & \Of^{\hat{T}^\star}_\bullet \ar[d]^{\Ug^\star_\Of}
\\
\Qf^{\hat{T}^\star,\kappa}_\bullet \ar[rrrr]_{\Ug_\kappa^\star} & & & & \Qf^{\hat{T}^\star}_\bullet  
}
\quad \quad 
\xymatrix{
\Mf^{\hat{T}^\star,\kappa}_\bullet \ar[rrrr]^{\hat{\Ug}^\star_\kappa} \ar[d]_{\hat{\Ug}^\star_\Mf} & & & & \Mf^{\hat{T}^\star}_\bullet \ar[d]^{\Ug^\star_\Mf}
\\
\Qf^{\hat{T}^\star,\kappa}_\bullet \ar[rrrr]_{\Ug_\kappa^\star} & & & & \Qf^{\hat{T}^\star}_\bullet.  
} 
\end{equation*} 
The categories $\Of^{\hat{T}^\star}_\bullet$, respectively $\Of^{\hat{T}^\star,\kappa}_\bullet$ are full subcategories of $\Mf^{\hat{T}^\star}_\bullet$, respectively of $\Mf^{\hat{T}^\star,\kappa}_\bullet$. 
\end{proposition}

\begin{remark} \label{rem: ope-con*}
Exactly as already noticed in remark~\ref{rem: ope-con}, it is possible to introduce more restrictive notions of globular $\hat{T}^\star$-operadic-magma-contraction and globular $\hat{T}^\star$-operadic-contraction. 

\medskip 

A \emph{globular $\hat{T}^\star$-operadic magma contraction} over $\bullet$ $(M,\pi_M,\kappa_M,\eta_M,\mu_M)$ consists of a globular contracted $\hat{T}^\star$-operadic magma over $\bullet$ that further satisfies the commutativity of the following two diagrams in $\Qf^{\hat{T}^\star}_\bullet$: 
\begin{equation}\label{eq: L*+}
\begin{aligned}
\xymatrix{
\Par(\pi_\bullet) \ar[d]_{\kappa_\bullet} \ar[rrr]^{\Par_{\eta_M}} & & &  \Par(\pi_M) \ar[d]^{\kappa_M}
\\
\bullet \ar[rrr]^{\eta_M} & & & M
}
\quad \quad 
\xymatrix{
\Par(\pi_M\circ^2_0\pi_M) \ar[d]_{\kappa_{M}\circ^2_0\kappa_M} \ar[rrr]^{\Par_{\mu_M}} & & & \Par(\pi_M) \ar[d]^{\kappa_M}
\\
M\circ^1_0 M \ar[rrr]^{\mu_M} & & & M. 
}
\end{aligned}
\end{equation}
A \emph{globular $\hat{T}^\star$-operadic contraction} over $\bullet$ is just an globular contracted $\hat{T}^\star$-operad $(M,\pi_M,\kappa_M,\eta_M,\mu_M)$ that also satisfies the previous commutative diagrams.
 
\medskip 

We can introduce the category $\Mf\Kf^{\hat{T}^\star}_\bullet$, of globular $\hat{T}^\star$-operadic magma contractions over $\bullet$, as the full subcategory of $\Mf^{\hat{T}^\star,\kappa}_\bullet$; similarly the category $\Of\Kf^{\hat{T}^\star}_\bullet$, of globular $\hat{T}^\star$-operadic contractions over $\bullet$, as the full subcategory of $\Of^{\hat{T}^\star,\kappa}_\bullet$.  
\xqed{\lrcorner}
\end{remark}

The following remark is absolutely crucial for us: it identifies the terminal contracted-$\hat{T}^\star$-operad in $\Of^{\hat{T}^\star,\kappa}_\bullet$. 
\begin{remark} \label{rem: q-cong}
Notice that the globular $\omega$-set $\hat{T}^\star(\bullet)$ is naturally a globular $\hat{T}^\star$-collection $\hat{T}^\star(\bullet)\xrightarrow{\pi_{\hat{T}^\star(\bullet)}}\hat{T}^\star(\bullet)$ with projection $\pi_{\hat{T}^\star(\bullet)}$ the identity morphism of globular $\omega$-sets, 
\begin{itemize}
\item[$\blacktriangleright$]
for any $\hat{T}^\star$-collection $Q\xrightarrow{\pi}\hat{T}^\star(\bullet)$, the projection $\pi$ is a morphism in $\Qf_\bullet^{\hat{T}^\star}$, 
\begin{equation*}
\vcenter{
\xymatrix{
Q \ar[rrrr]^{\pi} \ar[drr]_{\pi} & & & & \hat{T}^\star(\bullet) \ar[dll]^{\pi_{\hat{T}^\star(\bullet)}}
\\ 
& & \hat{T}^\star(\bullet) & & 
}
} \quad \quad \pi_{\hat{T}^\star(\bullet)}:x\mapsto x.  
\end{equation*}
\end{itemize}
It also naturally becomes a contracted globular $\hat{T}^\star$-collection 
$\Par(\pi_{\hat{T}^\star(\bullet)})\xrightarrow{\kappa_{\hat{T}^\star(\bullet)}}\hat{T}^\star(\bullet)\xrightarrow{\pi_{\hat{T}^\star(\bullet)}}\hat{T}^\star(\bullet)$ with 
contraction $\kappa_{\hat{T}^\star(\bullet)}:(y^+,y,y^-)\mapsto y$ on 
$\Par(\pi_{\hat{T}^\star(\bullet)})=\Big\{(y^+,y,y^-)\in \hat{T}^\star(\bullet)\times\hat{T}^\star(\bullet)\times\hat{T}^\star(\bullet) \ \big| \ 
\xymatrix{\rtwocell^{y^+}_{y^-}{y}&} 
\Big\}$, 
\begin{itemize}
\item[$\blacktriangleright$]
for any contracted $\hat{T}^\star$-collection $\Par(\pi)\xrightarrow{\kappa}Q\xrightarrow{\pi}\hat{T}^\star(\bullet)$, the projection $\pi$ is a morphism in $\Qf_\bullet^{\hat{T}^\star,\kappa}$: 
\begin{gather*}
\vcenter{
\xymatrix{
\Par(\pi) \ar[d]_{\Par_\pi} \ar[rrrr]^{\kappa} & & & & Q  \ar[d]^{\pi}
\\
\Par(\pi_{\hat{T}^\star(\bullet)}) \ar[rrrr]^{\kappa_{\hat{T}^\star(\bullet)}} & & & & \hat{T}^\star(\bullet)
}
} \quad \quad 
\begin{aligned}
&\kappa_{\hat{T}^\star(\bullet)}:(y^+,y,y^-)\mapsto y, 
\\
&\Par_\pi: (x^+,y,x^-)\mapsto (\pi(x^+),y,\pi(x^-)), 
\end{aligned}
\end{gather*}
\end{itemize}

Furthermore $\hat{T}^\star(\bullet)$ is a globular $\hat{T}^\star$-operad with operadic unit 
$\bullet\xrightarrow{\eta_{\hat{T}^\star(\bullet)}}\hat{T}^\star(\bullet)$ coinciding with the $\hat{T}^\star$-monadic unit $\eta_{\hat{T}^\star(\bullet)}:=\eta^{\hat{T}^\star}_\bullet$,  
and operadic multiplication $\hat{T}^\star(\bullet)\circ^1_0\hat{T}^\star(\bullet) \xrightarrow{\mu_{\hat{T}^\star(\bullet)}}\hat{T}^\star(\bullet)$ given by 
$\mu_{\hat{T}^\star(\bullet)}:=\nu\circ(\pi_{\hat{T}^\star(\bullet)},\mu^{\hat{T}^\star}_\bullet)$, 
$ 
\hat{T}^\star(\bullet)\circ^1_0\hat{T}^\star(\bullet)
=\hat{T}^\star(\bullet)\times_{\hat{T}^\star(\bullet)}\hat{T}^\star(\hat{T}^\star(\bullet))
=\hat{T}^\star(\bullet)\times_{\hat{T}^\star(\bullet)}(\hat{T}^\star)^2(\bullet)
\xrightarrow{(\pi_{\hat{T}^\star(\bullet)},\mu^{\hat{T}^\star}_\bullet)}
\hat{T}^\star(\bullet)\times_{\hat{T}^\star(\bullet)}\hat{T}^\star(\bullet)
\xrightarrow{\nu}\hat{T}^\star(\bullet)$, 
where $(\hat{T}^\star)^2(\bullet)\xrightarrow{\mu^{\hat{T}^\star}_\bullet}\hat{T}^\star(\bullet)$ is the $\hat{T}^\star$-monadic multiplication, $\hat{T}^\star(\bullet)\times_{\hat{T}^\star(\bullet)}\hat{T}^\star(\bullet)
\xrightarrow{\nu}\hat{T}^\star(\bullet)$ is an isomorphism. 

To show that $\hat{T}^\star(\bullet)$ is a $\hat{T}^\star$-operad one verifies (using the definitions~\eqref{eq: bicat-associators-unitors} \eqref{eq: bicat-left-unitors} \eqref{eq: bicat-right-unitors} of associators and unitors via universal factorization property of pull-backs in the Cartesian category $\Qf_{\hat{T}^\star}$ and the equations~\eqref{eq: monad-asso-u} for the monad $\hat{T}^\star$) the associativity and unitality properties already described in~\eqref{eq: monadic-ax}, here in the case of $\hat{T}^\star(\bullet)$:
\begin{gather*}
\mu_{\hat{T}^\star(\bullet)}\circ^2_1 (\eta_{\hat{T}^\star(\bullet)}\circ^2_0\iota^1_{\hat{T}^\star(\bullet)})
= \nu\circ (\iota_{\hat{T}^\star(\bullet)},\mu^{\hat{T}^\star}_\bullet)\circ (\eta_\bullet^{\hat{T}^\star},\ \iota_{\hat{T}^\star(\bullet)})
=\lambda, 
\\
\mu_{\hat{T}^\star(\bullet)}\circ^2_1(\iota^1_{\hat{T}^\star(\bullet)}\circ^2_0\eta_{\hat{T}^\star(\bullet)})
=\nu\circ (\mu^{\hat{T}^\star}_\bullet,\iota_{\hat{T}^\star(\bullet)})\circ (\iota_{\hat{T}^\star(\bullet)},\ \eta_\bullet^{\hat{T}^\star})
=\rho, 
\end{gather*}
can be respectively obtained by the commutativity of the following two diagrams and the unicity of $\lambda$ and $\rho$, 
$
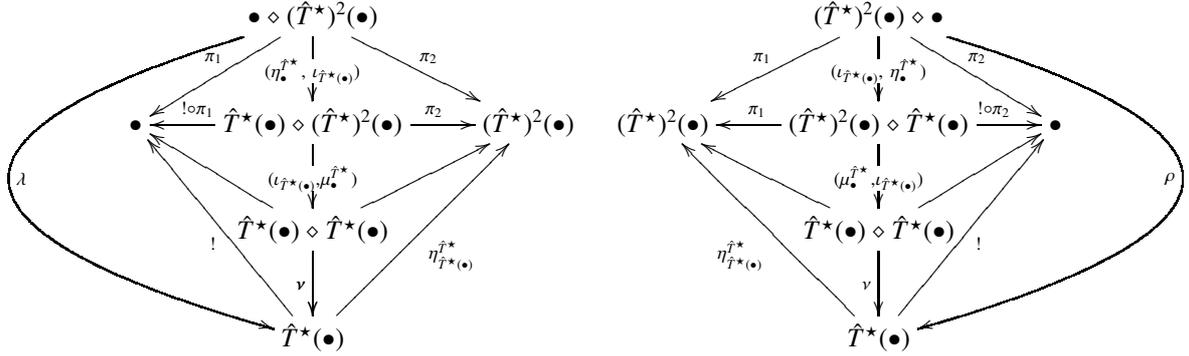
\xymatrix{
&\bullet\diamond (\hat{T}^\star)^2(\bullet) \ar[dr]^{\pi_2}\ar[dl]_{\pi_1} \ar@/_4cm/[ddd]^\lambda \ar[d]|{(\eta_\bullet^{\hat{T}^\star},\ \iota_{\hat{T}^\star(\bullet)})}& 
\\
\bullet & \ar[l]_{!\circ\pi_1\quad \quad }\hat{T}^\star(\bullet)\diamond(\hat{T}^\star)^2(\bullet) \ar[r]^{\quad \pi_2} \ar[d]|{(\iota_{\hat{T}^\star(\bullet)},\mu^{\hat{T}^\star}_\bullet)} & (\hat{T}^\star)^2(\bullet) 
\\
& \hat{T}^\star(\bullet)\diamond\hat{T}^\star(\bullet) \ar[d]_\nu \ar[d]_\nu \ar[ur]\ar[ul] & 
\\
& \hat{T}^\star(\bullet) \ar[uul]^{!} \ar[uur]_{\eta^{\hat{T}^\star}_{\hat{T}^\star(\bullet)}}& 
} 
\quad 
\xymatrix{
& (\hat{T}^\star)^2(\bullet)\diamond\bullet \ar[dr]^{\pi_2}\ar[dl]_{\pi_1} \ar@/^4cm/[ddd]_\rho \ar[d]|{(\iota_{\hat{T}^\star(\bullet)},\ \eta_\bullet^{\hat{T}^\star})} & 
\\
(\hat{T}^\star)^2(\bullet) & \ar[l]_{\pi_1\quad} (\hat{T}^\star)^2(\bullet)\diamond\hat{T}^\star(\bullet) \ar[r]^{\quad\quad !\circ\pi_2} \ar[d]|{(\mu^{\hat{T}^\star}_\bullet,\iota_{\hat{T}^\star(\bullet)})}
& \bullet 
\\
& \hat{T}^\star(\bullet)\diamond\hat{T}^\star(\bullet) \ar[d]_\nu \ar[ur]\ar[ul] & 
\\
& \ar[uur]_{!} \hat{T}^\star(\bullet) \ar[uul]^{\eta^{\hat{T}^\star}_{\hat{T}^\star(\bullet)}} & 
}
$ 
\\ 
the operadic associativity of $\hat{T}^\star(\bullet)$, that consists of the following identity 
\begin{gather*} 
\mu_{\hat{T}^\star(\bullet)}\circ^2_1(\iota^1_{\hat{T}^\star(\bullet)}\circ^2_0 \mu_{\hat{T}^\star(\bullet)})\circ^2_1 \alpha
=\mu_{\hat{T}^\star(\bullet)}\circ^2_1 (\mu_{\hat{T}^\star(\bullet)}\circ^2_0\iota^1_{\hat{T}^\star(\bullet)}), 
\end{gather*}
can be obtained reconsidering the unicity of $\alpha:=\alpha_{\hat{T}^\star(\bullet)\hat{T}^\star(\bullet)\hat{T}^\star(\bullet)}$ in diagram~\eqref{eq: op-asso-T*}, reproduced here in our case, 
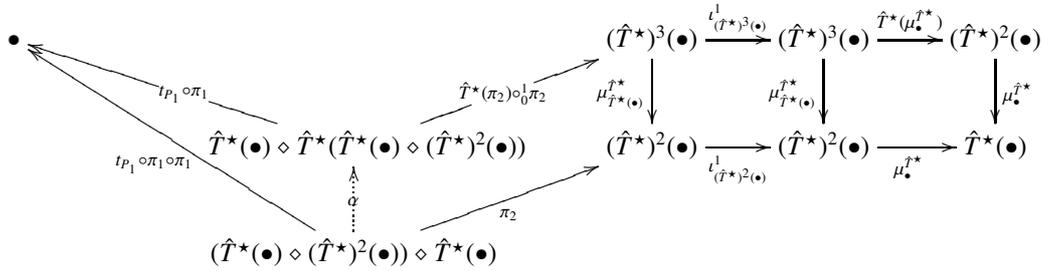
\begin{equation*}
\xymatrix{
\bullet & & & (\hat{T}^\star)^3(\bullet) \ar[d]_{\mu^{\hat{T}^\star}_{\hat{T}^\star(\bullet)}} \ar[r]^{
\iota^1_{(\hat{T}^\star)^3(\bullet)} 
} 
& (\hat{T}^\star)^3(\bullet) \ar[r]^{\hat{T}^\star(\mu^{\hat{T}^\star}_{\bullet})} \ar[d]_{\mu^{\hat{T}^\star}_{\hat{T}^\star(\bullet)}} & (\hat{T}^\star)^2(\bullet) \ar[d]^{\mu^{\hat{T}^\star}_{\bullet}} & 
\\
& & \ar[llu]|{t_{P_1}\circ\pi_1} 
\quad \hat{T}^\star(\bullet)\diamond\hat{T}^\star(\hat{T}^\star(\bullet)\diamond (\hat{T}^\star)^2(\bullet))
\ar[ur]|{\hat{T}^\star(\pi_2)\circ^1_0\pi_2}
& (\hat{T}^\star)^2(\bullet) \ar[r]_{
\iota^1_{(\hat{T}^\star)^2(\bullet)} 
} & (\hat{T}^\star)^2(\bullet) \ar[r]_{\mu^{\hat{T}^\star}_{\bullet}} & \hat{T}^\star(\bullet) 
\\
& & \ar[uull]^{t_{P_1}\circ\pi_1\circ\pi_1} 
(\hat{T}^\star(\bullet)\diamond (\hat{T}^\star)^2(\bullet))\diamond \hat{T}^\star(\bullet) 
\ar@{.>}[u]|{\alpha} \ar[ur]_{\pi_2}
}
\end{equation*}
and simply noting the structural properties of the multiplication maps involved:  
\begin{gather*}
\xymatrix{
\bullet & & \lltwocell^{\hat{T}^\star(\bullet)}_{(\hat{T}^\star(\bullet)\diamond (\hat{T}^\star)^2(\bullet))\diamond \hat{T}^\star(\bullet)}{^} \bullet
}
\quad = \quad 
\vcenter{
\xymatrix{
& & \ar[dll]_{!} (\hat{T}^\star(\bullet)\diamond (\hat{T}^\star)^2(\bullet))\diamond \hat{T}^\star(\bullet) \ar[drr]^{\mu^{\hat{T}^\star}_\bullet\circ \pi_2} 
\ar[d]|{\mu_{\hat{T}^\star(\bullet)}\circ^2_1 (\mu_{\hat{T}^\star(\bullet)}\circ^2_0\iota^1_{\hat{T}^\star(\bullet)})}& & 
\\
\bullet & & \ar[ll]_{!} \hat{T}^\star(\bullet) \ar[rr]^{\iota^1_{\hat{T}^\star(\bullet)}} & & \hat{T}^\star(\bullet), 
}
}
\\
\xymatrix{
\bullet & & \lltwocell^{\hat{T}^\star(\bullet)}_{\hat{T}^\star(\bullet)\diamond\hat{T}^\star(\hat{T}^\star(\bullet)\diamond (\hat{T}^\star)^2(\bullet))}{^} \bullet
}
\quad = \quad 
\vcenter{
\xymatrix{
& & \ar[dll]_{!} 
\hat{T}^\star(\bullet)\diamond\hat{T}^\star(\hat{T}^\star(\bullet)\diamond (\hat{T}^\star)^2(\bullet)) \ar[drr]^{\quad\quad\quad\mu^{\hat{T}^\star}_\bullet\circ \hat{T}^\star(\mu^{\hat{T}^\star}_\bullet)\circ(\pi_2\circ^1_0\hat{T}^\star(\pi_2))} 
\ar[d]|{
\mu_{\hat{T}^\star(\bullet)}\circ^2_1(\mu_{\hat{T}^\star(\bullet)}\circ^2_0\iota^1_{\hat{T}^\star(\bullet)})}
& & 
\\
\bullet & & \ar[ll]_{!} \hat{T}^\star(\bullet) \ar[rr]^{\iota^1_{\hat{T}^\star(\bullet)}} & & \hat{T}^\star(\bullet).
}
}
\end{gather*}

\begin{itemize}
\item[$\blacktriangleright$] 
for any (contracted) globular $\hat{T}^\star$-operadic magma $(M,\pi_M,\eta_M,\mu_M)$, and hence for any (contracted) globular $\hat{T}^\star$-operad, the projection $M\xrightarrow{\pi_M}\hat{T}^\star(\bullet)$ is a morphism in $\Mf^{\hat{T}^\star}_\bullet$ (and respectively in $\Mf^{\hat{T}^\star,\kappa}_\bullet$): 
\begin{gather*}
\vcenter{
\xymatrix{
\bullet \ar[rr]^{\eta_M} \ar[d]_{!} & & M \ar[d]^{\pi_M}
\\
\bullet \ar[rr]^{\eta_{\hat{T}^\star(\bullet)}} & & \hat{T}^\star(\bullet)
}
}
\quad \quad \text{where $\eta_{\hat{T}^\star(\bullet)}:=\eta_\bullet^{\hat{T}^\star}$ and $\bullet\xrightarrow{!}\bullet$ is the terminal morphism},  
\\ 
\begin{aligned}
& \mu_{\hat{T}^\star(\bullet)}:=\nu\circ(\pi_{\hat{T}^\star(\bullet)},\mu^{\hat{T}^\star}_\bullet) 
\\
&\pi_M\circ^2_0\pi_M:(x,y)\mapsto 
\left(\pi_M(x),\hat{T}^\star(\pi_M))(y)\right) 
\end{aligned}
\quad \quad 
\vcenter{
\xymatrix{
M\circ^1_0M \ar[rr]^{\mu_M} \ar[d]_{\pi_M\circ^2_0\pi_M} & & M \ar[d]^{\pi_M}
\\
\hat{T}^\star(\bullet)\circ^1_0 \hat{T}^\star(\bullet)\ar[rr]^{\mu_{\hat{T}^\star(\bullet)}} & & \hat{T}^\star(\bullet)
}
}
\end{gather*}
\end{itemize} 
where we have $M\circ^1_0M=M\times_{\hat{T}^\star(\bullet)}\hat{T}^\star(M)\xrightarrow{\pi_M\circ^2_0\pi_M} \hat{T}^\star(\bullet)\times_{\hat{T}^\star(\bullet)}\hat{T}^\star(\hat{T}^\star(\bullet))
=\hat{T}^\star(\bullet)\circ^1_0\hat{T}^\star(\bullet)$. 

\medskip 

The (contracted) operad $\hat{T}^\star(\bullet)$ is final in $\Of^{\hat{T}^\star,\kappa}_\bullet$: for any other (contrated) $\hat{T}^\star$-operad $(P,\pi^P,\mu^P,\eta^P,\kappa^P)$ the unique morphism of (contracted) $\hat{T}^\star$-operads into $\hat{T}^\star(\bullet)$ is given by the projection  $P\xrightarrow{\pi^P}\hat{T}^\star(\bullet)$.  

\medskip 

Actually $\hat{T}^\star(\bullet)$ is also a $\hat{T}^\star$-operadic contraction since it furthermore satisfies, by direct computation, the following compatibility properties between contraction and operad structures (see diagrams~\eqref{eq: d1} and~\eqref{eq: d2}): 
\begin{equation*}
\mu_{\hat{T}^\star(\bullet)}\circ(\kappa_{\hat{T}^\star(\bullet)}\circ^2_0\kappa_{\hat{T}^\star(\bullet)}) = \kappa_{\hat{T}^\star(\bullet)}\circ \Par_{\mu_{\hat{T}^\star(\bullet)}},
\\
\eta_{\hat{T}^\star(\bullet)}\circ\kappa_{\pi_\bullet} =\kappa_{\hat{T}^\star(\bullet)}\circ\Par_{\eta_{\hat{T}^\star(\bullet)}}.
\end{equation*}
\begin{itemize}
\item[$\blacktriangleright$] 
for any globular $\hat{T}^\star$-operadic contraction magma $(M,\pi_M,\eta_M,\mu_M)$ (and for any globular $\hat{T}^\star$-operadic contraction) the projection $M\xrightarrow{\pi_M}\hat{T}^\star(\bullet)$ is a morphism in $\Mf\Kf^{\hat{T}^\star}_\bullet$ (and respectively in $\Of\Kf^{\hat{T}^\star}_\bullet$): 

\begin{gather}\label{eq: d1}
\vcenter{\xymatrix{
\ar[dr]_{\kappa_\bullet} \Par(\pi_\bullet) \ar[rrrr]^{\Par_{\eta_M}} \ar[ddd]_{\Par_!} & & & & \Par(\pi_M) \ar[ddd]^{\Par_{\pi_M}} \ar[dl]_{\kappa_M}
\\
& \bullet \ar[d]_{!}\ar[rr]^{\eta_M} & & M \ar[d]_{\pi_M} & 
\\
& \bullet \ar[rr]^{\eta_{\hat{T}^\star(\bullet)}} & & \hat{T}^\star(\bullet) & 
\\
\ar[ur]^{\kappa_{\pi_\bullet}} \Par(\pi_\bullet) \ar[rrrr]^{\Par_{\eta_{\hat{T}^\star}}} & & & & \Par(\pi_{\hat{T}^\star(\bullet)}) \ar[ul]^{\kappa_{\hat{T}^\star(\bullet)}}
}
}\quad \quad \quad \text{where:} 
\\ \notag
\Par(\pi_\bullet)=\Big\{(x^+,y,x^-) \ | \ \xymatrix{ \rtwocell^{x^+}_{x^-}{y} & } \ \in \bullet \Big\},
\quad 
\Par_{\eta^{\hat{T}^\star}} : (x^+,y,x^-)\mapsto (x^+,\eta^{\hat{T}^\star}(y),x^-),
\\ \notag
\kappa_{\hat{T}^\star(\bullet)}: (x^+,\eta^{\hat{T}^\star}(y),x^-)\mapsto \eta^{\hat{T}^\star}(y),
\quad 
\kappa_{\pi_\bullet}: (x^+,y,x^-)\mapsto y; 
\end{gather}
\begin{gather}
\label{eq: d2}
\vcenter{
\xymatrix{
\Par(\pi_M\circ^2_0\pi_M) \ar[rrrr]^{\Par_{\mu_M}} \ar[ddd]_{\Par_{\pi_M\circ^2_0\pi_M}} \ar[dr]_{\kappa_M\circ^2_0\kappa_M}
& & & & \Par(\pi_M) \ar[dl]_{\kappa_M}\ar[ddd]^{\Par_{\pi_M}}
\\
& M\circ^1_0 M \ar[rr]^{\mu_M} \ar[d]_{\pi_M\circ^2_0\pi_M} & & M \ar[d]^{\pi_M} & 
\\
& \hat{T}^\star(\bullet)\circ^1_0\hat{T}^\star(\bullet) \ar[rr]^{\mu_{\hat{T}^\star(\bullet)}} & & \hat{T}^\star(\bullet) & 
\\
\Par(\pi_{\hat{T}^\star(\bullet)}\circ^2_0\pi_{\hat{T}^\star(\bullet)})\ar[ur]_{\kappa_{\hat{T}^\star(\bullet)}\circ^2_0\kappa_{\hat{T}^\star(\bullet)}} \ar[rrrr]^{\Par_{\mu_{\hat{T}^\star(\bullet)}}} & & & & \Par(\pi_{\hat{T}^\star(\bullet)}). \ar[ul]_{\kappa_{\hat{T}^\star(\bullet)}}
}
}
\\ \notag 
\Par(\pi_{\hat{T}^\star(\bullet)}\circ^2_0\pi_{\hat{T}^\star(\bullet)})\simeq \Par(\pi_{\hat{T}^\star(\bullet)})\circ^1_0\Par(\pi_{\hat{T}^\star(\bullet)}), 
\\ \notag 
\kappa_{\hat{T}^\star(\bullet)}:(\mu(x^+_1,x^+_2),\mu(y_1,y_2),\mu(x^-_1,x^-_2))\mapsto \mu(y_1,y_2), 
\quad \kappa_{\mu_{\hat{T}^\star(\bullet)}}: (y_1,y_2)\mapsto \mu(y_1,y_2). 
\\ \notag
\kappa_{\hat{T}^\star(\bullet)}\circ^2_0\kappa_{\hat{T}^\star(\bullet)}:((x^+_1,y_1,x^-_1),(x^+_1,y_1,x^-_1))\mapsto (y_1,y_2), 
\\ \notag
\Par_{\mu_{\hat{T}^\star(\bullet)}}:((x^+_1,y_1,x^-_1),(x^+_1,y_1,x^-_1))\mapsto (\mu(x^+_1,x^+_2),\mu(y_1,y_2),\mu(x^-_1,x^-_2)).  
\end{gather}
\end{itemize}
The $\hat{T}^\star$-operadic magma contraction (respectively the $\hat{T}^\star$-operadic contraction) $\hat{T}^\star(\bullet)$ is final in $\Mf\Kf^{\hat{T}^\star}_\bullet$ (respectively in $\Of\Kf^{\hat{T}^\star}_\bullet$): the projection $M\xrightarrow{\pi_M}\hat{T}^\star(\bullet)$ being the terminal morphism from any other object $M$. 
\xqed{\lrcorner}
\end{remark}

The following is the fundamental theorem in our paper, allowing the definition of weak involutive $\omega$-categories.  
\begin{theorem}\label{th: L*}
The category $\Of^{\hat{T}^\star,\kappa}_\bullet$ of globular contracted $\hat{T}^\star$-operads has initial objects. 
\end{theorem}
\begin{proof}
Instead of following Leinster's original line of proof in section~\ref{sec: L-w-T}, we give a direct argument. 
\begin{itemize}
\item[$\blacktriangleright$] 
The category $\Qf^{\hat{T}^\star}_\bullet$ has an initial object $I$: the empty $\hat{T}^\star$-collection $I\xrightarrow{\pi}\hat{T}^\star(\bullet)$ given by $I^n:=\varnothing$, for all $n\in \NN$, where all the source/target maps and the projection $\pi$ are empty functions. 
\item[$\blacktriangleright$] 
Left-adjoint functors preserve colimits (see~\cite[theorem 4.5.3]{Ri16}) and hence they preserve initial objects (that are colimits of the empty diagram). 
\item[$\blacktriangleright$] 
Hence, if $\Ug^\star_\Of\circ\hat{\Ug}^\star_\kappa=\Ug_\kappa^\star\circ\hat{\Ug}^\star_\Of$ has a left-adjoint $\Qf^{\hat{T}^\star}_\bullet\xrightarrow{\Lg}\Of^{\hat{T}^\star,\kappa}_\bullet$, the object $L^\star:=\Lg^0(I)$ is initial in $\Of^{\hat{T}^\star,\kappa}_\bullet$. 
\end{itemize}
The theorem is now reduced to providing the existence of a \textit{free contracted $\hat{T}^\star$-operad over a $\hat{T}^\star$-collection}. This is achieved below, in the theorem~\ref{th: free-c-T*}, by an argument substantially similar to that used in our construction of the free self-dual Penon's contractions in~\cite[proposition~3.3]{BeBe17}. 
\end{proof}

First we need to define suitable free structures over $\hat{T}^\star$-collections. \footnote{
Of course the definitions, that for convenience are here stated for the specific case of $\hat{T}^\star$, work for any Cartesian monad $T$. 
} 
\begin{definition} \label{def: free-c-T*}
A \emph{free globular contracted $\hat{T}^\star$-operad over the $\hat{T}^\star$-collection} $Q:=Q\xrightarrow{\pi}\hat{T}^\star(\bullet)$ consists of a contracted $\hat{T}^\star$-operad $P:=(P,\pi_P,\kappa_P,\mu_P,\eta_P)$ and a morphism $Q\xrightarrow{\zeta}\Ug_\kappa^\star\circ\hat{\Ug}_\Of^\star(P)$ in $\Qf^{\hat{T}^\star}_\bullet$ that satisfies the following universal factorization property: 
for every other morphism $Q\xrightarrow{\phi}\Ug_\kappa^\star\circ\hat{\Ug}_\Of^\star(\hat{P})$ in $\Qf^{\hat{T}^\star}_\bullet$, where $\hat{P}$ is another contracted $\hat{T}^\star$-operad, 
there exists a unique morphism $P\xrightarrow{\hat{\phi}}\hat{P}$ in $\Of^{\hat{T}^\star,\kappa}_\bullet$ such that $\phi=\hat{\phi}\circ\zeta$. 

\medskip 

A \emph{free globular contracted $\hat{T}^\star$-operadic magma} $M:=(M,\pi_M,\kappa_M,\mu_M,\eta_M)$ over a $\hat{T}^\star$-collection, is defined in a similar way, via a morphism $Q\xrightarrow{\xi_Q}\Ug_\kappa^\star\circ\hat{\Ug}_\Mf^\star(M)$, substituting ``operads'' with ``operadic magmas'' above. 
\end{definition}

As any universal factorization property construct, free contracted $\hat{T}^\star$-operad(ic magma)s are unique, modulo a unique isomorphism compatible with the universal factorization property. 
Existence is shown in theorem~\ref{th: free-c-T*}. 

\medskip 

Before embarking on the proof, we need to introduce relevant notions of \textit{congruence} and \textit{quotient structure}. 
\begin{definition}
A \emph{globular $\omega$-equivalence relation} is an equivalence relation $\E$ in a globular $\omega$-set $Q$ that is \textit{graded} \footnote{
This is equivalent to say that $\E$ consists of a sequence $\E^n\subset Q^n\times Q^n$ of equivalence relations in $Q^n$, for all $n\in\NN$. 
} 
$\E\subset Q\times_\NN Q:=\{(x,y)\in Q\times Q \ | \ \exists n\in\NN \st x,y\in Q^n\}$ and source/target preserving: \footnote{
In this way, the globular source and target product maps $(s,s),(t,t):Q\times_\NN Q\to Q\times_\NN Q$ restrict to (necessarily globular) source and target maps on $\Es$ and hence $(\E,(s,s)|_\E^\E,(t,t)|_\E^\E)$ becomes a globular $\omega$-set canonically included in $Q\times_\NN Q$. 
}
\begin{equation}\label{eq: c-st}
(x_1, x_2)\in\E\imp (s(x_1),s(x_2))\in\E, \quad  \quad 
(x_1, x_2)\in\E\imp (t(x_1),t(x_2))\in\E, \quad \quad \forall x,y\in Q. 
\end{equation}

A \emph{$\hat{T}^\star$-collection congruence} is a globular $\omega$-equivalence relation $\E$ in the globular $\omega$-set $Q$ of a globular $\hat{T}^\star$-collection $Q\xrightarrow{\pi}\hat{T}^\star(\bullet)$ that is projection-preserving: 
$(x,y)\in\E\imp \pi(x)=\pi(y)$ and hence $\E\subset Q\times_\pi Q$. \footnote{
This means that the globular $\omega$-set $(\E,(s,s),(t,t))$ is a $\hat{T}^\star$-collection $\E\xrightarrow{\pi_\E}\hat{T}^\star(\bullet)$ with projection $\pi_\E:(x,y)\mapsto \pi(x)=\pi(y)$; furthermore, we have a canonical inclusion morphism of $\hat{T}^\star$-collections $\E\xrightarrow{\varepsilon_\E} Q\times_\pi Q$ into the \emph{product $\hat{T}^\star$-collection $Q\times_\pi Q\xrightarrow{(\pi,\pi)}\hat{T}^\star(\bullet)$}. 
} 

\medskip 

A \emph{congruence of contracted $\hat{T}^\star$-collection} is a $\hat{T}^\star$-collection congruence $\E$ in a contracted $\hat{T}^\star$-collection $\Par(\pi)\xrightarrow{\kappa}Q\xrightarrow{\pi}\hat{T}^\star(\bullet)$ that is also contraction preserving: \footnote{
This entails that the $\hat{T}^\star$-collection $\Es\xrightarrow{\pi_\E}\hat{T}^\star(\bullet)$ has a contraction $\kappa_\E:\Par(\pi_\E)\to\E$ given by $\kappa_\E:=(\kappa,\kappa)$. 
}
\begin{gather} \label{eq: cg-cont} 
\forall (x^+_1,x_2^+),(x_1^-,x_2^-)\in\E \st 
(x_1^+,y_1,x_1^-),(x_2^+,y_2,x_2^-)\in\Par(\pi) \imp (\kappa(x_1^+,y_1,x_1^-),\kappa(x_2^+,y_2,x_2^-))\in\E. 
\end{gather}

\medskip 

A \emph{congruence of (contracted) $\hat{T}^\star$-operadic magma} $(M,\mu_M,\eta_M)$ is a congruence $\E$ of the underlying (contracted) $\hat{T}^\star$-collection $\Par(\pi)\xrightarrow{\kappa_M}M\xrightarrow{\pi_M}\hat{T}^\star(\bullet)$ 
that is unit and multiplication preserving: \footnote{
The congruence $\E$ is always unit preserving since the product morphism 
$\bullet\ \times_{\eta_{\hat{T}^\star}}\bullet\xrightarrow{(\eta_M,\eta_M)} M\times_{\pi_M}M$ has always image inside $\E$; hence $\E$ is equipped with a canonical unit $\bullet\xrightarrow{\eta_\E}\E$ given by composing $\bullet\xrightarrow{!}\bullet\ \times_{\eta_{\hat{T}^\star}}\bullet\xrightarrow{(\eta_M,\eta_M)|^\E}\E$ with isomorphism $!$ of terminal objects. 
} 

\begin{gather}
\label{eq: cg-mu}
(x,y)\in \E\circ^1_0\E=\E\times_{\hat{T}^\star(\bullet)}\hat{T}^\star(\E) \quad \imp \quad 
\left((\mu_M,\mu_M)\circ \tau_M\circ(\varepsilon\circ^1_0\varepsilon)\right)(x,y)\in\Es,
\end{gather}
where 
$\E\circ^1_0\E \xrightarrow{\varepsilon\circ^2_0\varepsilon} (M\times_{\pi_M}M)\circ^1_0(M\times_{\pi_M}M) \xrightarrow{\tau_M}(M\circ^1_0M)\times_{\pi_M\circ^2_0\pi_M}(M\circ^1_0M) \xrightarrow{(\mu_M,\mu_M)}M\times_{\pi_M}M$ with $\tau_M$ denoting the canonical isomorphism of $\hat{T}^\star$-collections between pull-back of products and product of pull-backs in $\Qf^{\hat{T}^\star}_\bullet$ and $!$ as the unique isomorphism of terminal objects.~\footnote{
In this way, denoting by $\E\circ^1_0\E\xrightarrow{\mu_\E}\E$ the restiction $\mu_\E:=(\mu_M,\mu_M)\circ \tau_M\circ(\varepsilon\circ^1_0\varepsilon)|^\E_{\E\circ^1_0\E}$ and by $\bullet\xrightarrow{\eta_\E}\E$ 
the restriction $\eta_\E:=(\eta_M,\eta_M)\circ !|_\bullet^\E$, 
we have that $(\E,\eta_\E,\mu_\E)$ is itself a (contracted) $\hat{T}^\star$-operadic magma and that the inclusion $\E\xrightarrow{\varepsilon}M\times_{\pi_M}M$ is a morphisms of (contracted) $\hat{T}^\star$ operadic magmas. Whenever $(M,\pi_M,\kappa_M,\eta_M,\mu_M)$ is a (contracted) $\hat{T}^\star$-operad, also $(\E,\pi_\E,\kappa_\E,\eta_\E,\mu_\E)$ is. 
} 
\end{definition}
The previous congruences have been defined, for our convenience, for 1-cells in $\Qf^{\hat{T}^\star}_\bullet$, but they actually work for 1-cells in the bicategory $\Qf_T$, where $T$ is any Cartesian monad on the category of globular $\omega$-sets $\Qf$. 

\medskip 

As usual, congruences produce quotients of the corresponding algebraic structures. 
\begin{proposition}\label{prop: q-cong}

Given a globular $\omega$-relation $\E\subset Q\times_\NN Q$ on a globular $\omega$-set $Q$, the family of quotients $Q/\E:=(Q^n/\E^n)_{n\in\NN}$ becomes a \textit{globular $\omega$-set} with well-defined sources and targets $s_{Q/\E}([x]_\E):=[s_Q(x)]$, $t_{Q/\E}([x]_\E):=[t_Q(x)]_\E$; furthermore the quotient map $Q\xrightarrow{\varpi_\E}Q/\E$, defined as usual by $\varpi_\E:x\mapsto [x]_\E$, is a morphism of globular $\omega$-sets. 

\medskip 

Given a congruence $\E\subset Q\times_\pi Q$ of (contracted) $\hat{T}^\star$-collection $Q\xrightarrow{\pi}\hat{T}^\star(\bullet)$, the quotient globular $\omega$-set $Q/\E$ becomes a (contracted) $\hat{T}^\star$-collection $Q/\E\xrightarrow{\pi_{Q/\E}}\hat{T}^\star(\bullet)$ with well-defined projection $\pi_{Q/\E}:[x]_\E\mapsto \pi(x)$ (and with contraction $\kappa_{Q/\E}:([x^+]_\E,y,[x^-]_\E)\mapsto [\kappa(x^+,y,x^-)]_\E$ on $\Par(\pi_{Q/\E})=\left\{([x^+]_\E,y,[x^-]_\E) \ | \ (x^+,y,x^-)\in\Par(\pi)\right\}$); furthermore the quotient map $Q\xrightarrow{\varpi_\E}Q/\E$ is a morphism of (contracted) $\hat{T}^\star$-collections. 

\medskip 

Given a congruence $\E\subset M\times_{\pi_M}M$ of (contracted) $\hat{T}^\star$-operadic magma on $(M,\pi_M,\eta_M,\mu_M)$, the quotient (contracted) $\hat{T}^\star$-collection $M/\E\xrightarrow{\pi_{M/\E}}\hat{T}^\star(\bullet)$ becomes a (contracted) $\hat{T}^\star$-operadic magma with operadic unit $\bullet\xrightarrow{\eta_{M/\E}}M/\E$ given by $\eta_{M/\E}:=\pi_{\E}\circ\eta_M$ and 
with operadic multiplication $M/\E\circ^1_0M/\E\xrightarrow{\mu_{M/\E}}M/\E$ that is well-defined by 
$\mu_{M/\E}:=([x]_\E,[y]_\E)\mapsto [\mu_M(x,y)]_\E$, for all $(x,y)\in M\circ^1_0M=M\times_{\hat{T}^\star(\bullet)}\hat{T}^\star(M)$; furthermore the quotient map $M\xrightarrow{\varpi_\E}M/\E$ is a morphism of (contracted) $\hat{T}^\star$-operadic magmas. 

\medskip 

Given a morphism $Q_1\xrightarrow{\phi}Q_2$ of the categories in proposition~\ref{prop: all-*-cat}, a congruence $\E$ of the respective type in $Q_2$ naturally induces a congruence, of the same type, $\E_\phi:=\big\{(x,y)\in Q_1\times Q_1 \ | \ (\phi(x),\phi(y))\in\E \big\}$ in $Q_1$. 
Furthermore if $\E_1$ is a congruence of the respective type in $Q_1$ such that $\E_1\subset\E_\phi$, there exists a unique well-defined quotient morphism $Q_1/\E_1\xrightarrow{\hat{\phi}}Q_2$ given by $\hat{\phi}:[x]_{\E_1}\mapsto\phi(x)$ such that $\phi=\hat{\phi}\circ\pi_{\E_1}$. 
\end{proposition}
\begin{proof}
By property~\eqref{eq: c-st}, for all $n\in\NN$, the source and target $s^n_{Q/\E},t^n_{Q/\E}:Q^{n+1}/\E^{n+1}\to Q^{n}/\E^{n}$ are well-defined and their globularity property follows from the globularity of $Q$. 
For all $n\in\NN$, the quotient map $\varpi^n_{\E}:Q^n\to Q^n/\E^n$ is well-defined by $x\mapsto [x]_{\E^n}$; and $\varpi_\E:=(\pi^n_{\E})_\NN:Q\to Q/\E$ becomes a morphism in $\Qf$ since 
$s^{n}_{Q/\E}(\varpi^{n+1}_\E(x))=s^{n}_{Q/\E}([x]_{\E^{n+1}})=[s^n_Q(x)]_{\E^{n}}=\varpi^n_\E(s_Q^n(x))$, for all $x\in Q^{n+1}$ and similarly for targets. 

\medskip 

A congruence $\E\subset Q\times_\pi Q$ of a collection $Q\xrightarrow{\pi}\hat{T}^\star(\bullet)$ is necessarily a congruence of $\omega$-globular sets, hence we already have a quotient morphism $Q\xrightarrow{\varpi_\E} Q/\E$ in $\Qf$. Since $\E\subset Q\times_\pi Q$, the assignment $[x]_\E\mapsto \pi(x)$ is a well-defined graded function $\pi_{Q/\E}:Q/\E\to \hat{T}^\star(\bullet)$ that is actually a morphism of globular $\omega$-sets in $\Qf$: for all $[x]_\E\in Q/\E$,   $\pi_{Q/\E}(s_{Q/\E}[x]_\E)=\pi_{Q/\E}([s_Q(x)]_\Es)=\pi(s_Q(x))=s(\pi(x)) =s_{\hat{T}^\star(\bullet)}(\pi_{Q/\E}[x]_\E)$ and similarly for targets. 
Since $\pi_{Q/\E}\circ\varpi_\E(x)=\pi_{Q/\E}([x]_\E)=\pi(x)$, for all $x\in Q$, the quotient $Q\xrightarrow{\varpi_{\E}}Q/\E$ is a morphism of $\hat{T}^\star$-collections in $\Qf^{\hat{T}^\star}_\bullet.$

\medskip 

Suppose now that $\Par(\pi)\xrightarrow{\kappa}Q\xrightarrow{\pi}\hat{T}^\star(\bullet)$ is a contracted $\hat{T}^\star$-collection and $\E\subset Q\times_\pi Q$ is a congruence of contracted $\hat{T}^\star$-collections. For the quotient $\hat{T}^\star$-collection $Q/\E\xrightarrow{\pi_{Q/\E}}\hat{T}^\star(\bullet)$ we have that $[x_1]_\E,[x_2]_\E\in Q/\E$ are $\pi_{Q/\E}$-parallel if and only if $x_1,x_2$ are $\pi$-parallel and hence $\Par(\pi_{Q/\E})=\big\{([x^+]_\E,y,[x^-]_\E) \ | \ (x^+,y,x^-)\in\Par(\pi)\big\}$. 
Furthermore, since $\Par_{\pi_{Q/\E}}:(x^+,y,x^-)\mapsto ([x^+],y,[x^-])$ there exists a unique relation $\Par(\pi_{Q/\E})\xrightarrow{\kappa_{Q/\E}}Q/\E$ that satisfies $\Par_{\pi_\E}\circ\kappa_{Q/\E}=\kappa$ and that is given by $\kappa_{Q/\E}([x^+]_\E,y,[x^-]_\E):=[\kappa(x^+,y,x^-)]_\E$.  
Since $\E$ is a congruence of contracted $\hat{T}^\star$-collections, from equation~\eqref{eq: cg-cont} we see that $\kappa_{Q/\E}$ is actually a well-defined function and hence a contraction on $Q/\E$ and $Q\xrightarrow{\varpi_{\E}}Q/\E$ is a morphism in $\Qf^{\hat{T}^\star,\kappa}_\bullet$. 

\medskip 

Let $\E$ be a congruence of the (contracted) operadic magma $(M,\pi_M,\kappa_M,\eta_M,\mu_M)$. We already know that the quotient $(M/\E,\pi_{M/\E},\kappa_{M/\E})$ is a (contracted) $\hat{T}^\star$-collection and that $M\xrightarrow{\varpi_\E}M/\E$ is a morphism in $\Qf^{\hat{T}^\star,\kappa}_\bullet$. 

The operadic unit map $\eta_{M/\E}:\bullet\mapsto [\eta_M(\bullet)]_\E$ is well defined, and we immediately get $\eta_{M/\E}=\pi_{M/\E}\circ\eta_M$, hence the quotient morphism $\pi_{M/\E}$ preserves units. 

To describe the operadic multiplication, we first notice that we have a canonical exchange isomorphism $(M/\E)\circ^1_0(M/\E)\xrightarrow{\chi}(M\circ^1_0M)/(\E\circ^1_0\E)$; any multiplication morphism $(M/\E)\circ^1_0(M/\E)\xrightarrow{\mu_{M/\E}}M/\E$ such that $\mu_{M/\E}\circ(\varpi_{\E}\circ^2_0\varpi_E)=\varpi_\E\circ\mu_M$, must necessarily be given by $\mu_{M/\E}:=\hat{\mu}_{M/\E}\circ\chi$ where, using the congruence property~\eqref{eq: cg-mu} of $\E$, we have that $\hat{\mu}_{M/\E}:[(x,y)]_{\E\circ^1_0\E}\mapsto\varpi_\E\circ\mu_M(x,y)$, for $(x,y)\in M\circ^1_0M$, is a well-defined morphism of (contracted) $\hat{T}^\star$-collections. 
Since $\mu_{M/\E}\circ(\varpi_{\E}\circ^2_0\varpi_E)=\hat{\mu}_{M/\E}\circ\chi\circ(\varpi_{\E}\circ^2_0\varpi_E)
=\hat{\mu}_{M/\E}\circ \varpi_{\E\circ^1_0\E}=\varpi_\E\circ\mu_M$, we see that $\pi_\E$ is a morphism of (contracted) $\hat{T}^\star$-operadic magmas.

\medskip 

The family $\E_\phi\subset Q_1\times Q_1$ is an equivalence relation in $Q_1$ and, since $\phi$ is grade-preserving, we also have $\E_\phi\subset Q_1\times_\NN Q_1$ and hence $\E_\phi$ consists of a family of equivalence relations $\E_\phi^n\subset Q_1^n\times Q_1^n$, for all $n\in\NN$. 
Since $\phi$ is always a morphism of globular $\omega$-sets, $s_{Q_2}^n\circ\phi^{n+1}=\phi^n\circ s_{Q_1}^n$ (and similarly for targets), for all $n\in\NN$, and hence property~\eqref{eq: c-st} holds and $\E_\phi$ is a globular $\omega$-equivalence relation in $Q_1$. 

If $Q_1\xrightarrow{\phi}Q_2$ is a morphism of $\hat{T}^\star$-collections in $\Qf^{\hat{T}^\star}_\bullet$, $(x,y)\in\E_\phi \imp \pi_{Q_2}(\phi(x))=\pi_{Q_2}(\phi(y))\imp \pi_{Q_1}(x)=\pi_{Q_1}(y)$ and hence $\E_\phi\subset Q_1\times_{\pi_{Q_1}}Q_1$ is a congruence of $\hat{T}^\star$-collections in $Q_1$.

Whenever $\phi$ is a morphism of contracted $\hat{T}^\star$-collections in  $\Qf^{\hat{T}^\star,\kappa}_\bullet$, given $(x^+_1,y_1,x^-_1),(x^+_2,y_2,x^-_2)\in\Par(\pi_1)$ with $(x^+_1,x^+_2),(x^-_1,x^-_2)\in\E_\phi$, we must show that $(\kappa_1(x^+_1,y_1,x^-_1),\kappa_1(x^+_2,y_2,x^-_2))\in\E_\phi$ that is equivalent to show 
$(\kappa_2\circ\Par_\phi(x^+_1,y_1,x^-_1),\kappa_2\circ\Par(x^+_1,y_1,x^-_1))=(\phi\circ\kappa_1(x^+_1,y_1,x^-_1),\phi\circ\kappa_1(x^+_1,y_1,x^-_1))\in\E$ and this is true since we have $(\phi(x_1^+),y_1,\phi(x_1^-)),(\phi(x_2^+),y_2,\phi(x_2^-))\in\Par(\pi_2)$ and $(\phi(x_1^+),\phi(x_2^+)),(\phi(x_1^-),\phi(x_2^-))\in\E$ and hence also 
$(k_2(\phi(x_1)^+,y_1,\phi(x_1)^-),\kappa_2(\phi(x_2)^+,y_2,\phi(x_2)^-)) =(\kappa_2\circ\Par_\phi(x^+_1,y_1,x^-_1),\kappa_2\circ\Par(x^+_1,y_1,x^-_1))\in\E$. 
This shows that $\E_\phi$ is also a congruence of contracted $\hat{T}^\star$-collection in $Q_1$. 

Finally if $\E$ is a congruence of (contracted) $\hat{T}^\star$-operadic magma in $Q_2$ and $Q_1\xrightarrow{\phi}Q_2$ is a morphism in $\Mf^{\hat{T}^\star}_\bullet$: 
$(x,y)\in\E_\phi\circ^1_0\E_\phi\imp (\phi\circ^2_0\phi)(x,y)\in\E\circ^1_0\E\imp 
\mu_\E\circ(\phi\circ^2_0\phi)(x,y)\in\E\imp (\phi\circ^2_0\phi)(\mu_\E(x,y))\in\E\imp \mu_\E(x,y)\in\E_\phi$, hence $\E_\phi$ is a congruence $\hat{T}^\star$-operadic magmas. The same argument assures that, if $\phi$ is a morphism in $\Mf^{\hat{T}^\star,\kappa}_\bullet$, $\E_\phi$ is a congruence of contracted $\hat{T}^\star$-operadic magmas. 

\medskip 

Finally, whenever $\E_1\subset\E_\phi$ is a congruence (of the ``same type'' of the morphism $\phi$) in $Q_1$, any relation $Q_1/\E_1\xrightarrow{\hat{\phi}}Q_2$ such that $\phi=\hat{\phi}\circ\pi_{\E_1}$ must necessarily associate $[x]_{\E_1}\mapsto\phi(x)$ and this is a well defined function since $[x]_{\E_1}\subset[x]_{\E_\phi}$ and hence $\phi(x)$ does not depend on the representative element.  

We must show that $Q_1/\E_1\xrightarrow{\hat{\phi}}Q_2$ is a morphism in the same category of $\phi$. 

From its definition $\hat{\phi}$ is already a graded map: $\hat{\phi}(Q_1^n/\E^n_1)\subset Q_2^n$, for all $n\in\NN$. For all $x\in Q_1$, we have  $s_{Q_2}(\hat{\phi}([x]_{\E_1}))=s_{Q_2}(\phi(x))=\phi(s_{Q_1}(x))=\hat{\phi}([s_{Q_1}(x)]_{\E_1}) =\hat{\phi}(s_{Q_1/\E_1}([x]_{\E_1}))$ and similarly for the target; hence $\hat{\phi}$ is a morphism in $\Qf$. 
Since $\pi_{Q_2}(\hat{\phi}([x]_{\E_1}))=\pi_{Q_2}(\phi(x))=\pi_{Q_1}(x)=\pi_{Q_1/\E_1}([x]_{\E_1})$, for all $x\in Q_1$, we also have that $\hat{\phi}$ is a morphism in $\Qf^{\hat{T}^\star}_\bullet$. 
Since $\hat{\phi}$ is already a morphism of $\hat{T}^\star$-collections it naturally induces a map $\Par_{\hat{\phi}}:\Par(\pi_{Q_1/\E_1})\to\Par(\pi_{Q_2})$ and, whenever $\phi$ is a morphism in $\Qf^{\hat{T}^\star,\kappa}$, we show that $\kappa_{Q_2}\circ\Par_{\hat{\phi}}=\hat{\phi}\circ\kappa_{Q_1/\E_1}$: 
\begin{align*}
\kappa_{Q_2}(\Par_{\hat{\phi}}([x^+]_{\E_1},y,[x^-]_{\E_1}))
&=\kappa_{Q_2}(\Par_{\hat{\phi}}\circ\Par_{\pi_{\frac{Q_1}{\E_1}}}(x^+,y,x^-))
=\kappa_{Q_2}(\Par_\phi(x^+,y,x^-))
\\
&=\phi(\kappa_{Q_1}(x^+,y,x^-))=\hat{\phi}(\pi_{\frac{Q_1}{\E_1}}(x^+,y,x^-))
=\hat{\phi}(\kappa_{\frac{Q_1}{\E_1}}([x^+]_{\E_1},y,[x^-]_{\E_1})), 
\end{align*}
for all $(x^+,y,x^-)\in\Par(\pi_{Q_1})$, and hence $\hat{\phi}$ is a morphism in $\Qf^{\hat{T}^\star}_\bullet$. 

Supposing now that $Q_1\xrightarrow{\phi}Q_2$ is a morphism in $\Mf^{\hat{T}^\star}_\bullet$ and $\E_1$ is a congruence of $\hat{T}^\star$-operadic magmas, we have $\hat{\phi}\circ\eta_{Q_1/\E_1}=\hat{\phi}\circ\pi_{Q_1/\E_1}\circ\eta_{Q_1} =\phi\circ\eta_{Q_1}=\eta_{Q_2}$. 
As regards multiplication, from $\phi=\hat{\phi}\circ\pi_{Q_1/\E_1}$ we obtain $(\phi\circ^2_0\phi)=(\hat{\phi}\circ^2_0\hat{\phi})\circ(\pi_{Q_1/\E_1}\circ^2_0\pi_{Q_1/\E_1})$; since $\mu_{Q_2}\circ(\phi\circ^2_0\phi)=\phi\circ\mu_{Q_1}$ and $\mu_{Q_1/\E_1}\circ(\pi_{Q_1/\E_1}\circ^2_0\pi_{Q_1/\E_1})=\pi_{Q_1/\E_1}\circ\mu_{Q_1}$, we get: 
$\mu_{Q_2}\circ(\hat{\phi}\circ^2_0\hat{\phi})\circ(\pi_{Q_1/\E_1}\circ^2_0\pi_{Q_1/\E_1})
=\mu_{Q_2}\circ(\phi\circ^2_0\phi)=\phi\circ\mu_{Q_1}=\hat{\phi}\circ\pi_{Q_1/\E_1}\circ\mu_{Q_1}
=\hat{\phi}\circ\mu_{Q_1/\E_1}\circ(\pi_{Q_1/\E_1}\circ^2_0\pi_{Q_1/\E_1})$. Since $(\pi_{Q_1/\E_1}\circ^2_0\pi_{Q_1/\E_1})$ is an epimorphism, we finally have  
$\mu_{Q_2}\circ(\hat{\phi}\circ^2_0\hat{\phi})=\hat{\phi}\circ\mu_{Q_1/\E_1}$ and hence $\hat{\phi}$ is a morphism of (contracted) $\hat{T}^\star$-operadic magmas. 
\end{proof}

\begin{theorem} \label{th: free-c-T*}
A free contracted $\hat{T}^\star$-operad over a $\hat{T}^\star$-collection always exists. 
\end{theorem}
\begin{proof}
We proceed with a low-tech quite long, but direct, iterative construction followed by a quotient. 
\begin{itemize}
\item[a.]
starting from a globular $\hat{T}^\star$-collection $Q\xrightarrow{\pi_Q}\hat{T}^\star(\bullet)$ a new globular $\hat{T}^\star$-collection $\Mg(Q)\xrightarrow{\pi_{\Mg(Q)}}\hat{T}^\star(\bullet)$ is inductively constructed together with a morphism $\xi_Q:Q\to\Mg(Q)$ of globular $\hat{T}^\star$-collections; 
\item[b.] 
we show that $\Mg(Q)\xrightarrow{\pi_{\Mg(Q)}}\hat{T}^\star(\bullet)$ can be equipped with a contraction $\kappa_{\Mg(Q)}:\Par(\pi_{\Mg(Q)})\to \Mg(Q)$, 
a unit $\eta_{\Mg(Q)}:\bullet\to \Mg(Q)$ and a multiplication $\mu_{\Mg(Q)}: \Mg(Q)\circ^1_0\Mg(Q)=\Mg(Q)\times_{\hat{T}^\star(\bullet)}\hat{T}^\star(\Mg(Q))\to\Mg(Q)$; 
\item[c.] 
it is shown that $Q\xrightarrow{\xi_{Q}}\Mg(Q)$ is a free contracted $\hat{T}^\star$-operadic-magma over $Q\xrightarrow{\pi_Q}\hat{T}^\star(\bullet)$; 
\item[d.]
we establish the existence of the smallest congruence $\E_\X$ generated by operadic associativity and unitality axioms on the contracted operadic $\hat{T}^\star$-magma $\Mg(Q)$, so that the quotient $\Mg(Q)\xrightarrow{\pi_{\E_\X}}\Pg(Q):=\Mg(Q)/\E_\X$ becomes a morphism of contracted operadic $\hat{T}^\star$-magmas onto a contracted $\hat{T}^\star$-operad $\Pg(Q)\xrightarrow{\pi_{\Pg(Q)}}\hat{T}^\star(\bullet)$ with 
operadic unit $\eta_{\Pg(Q)}: \bullet\to \Pg(Q)$ operadic multiplication $\mu_{\Pg(Q)}:\Pg(Q)\circ^1_0\Pg(Q)\to\Pg(Q)$ and contraction $\kappa_{\Pg(Q)}:\Par(\pi_{\Pg(Q)})\to\Pg(Q)$, as explained in proposition~\ref{prop: q-cong};   
\item[e.]
the universal factorization property for free contracted $\hat{T}^\star$-operads is proved for $Q\xrightarrow{\zeta_Q:=\pi_{\E_\X}\circ\xi_Q}\Pg(Q)$. 
\end{itemize}
a. 

We start by explicitly providing $\Mg(Q)^j\xrightarrow{\pi_{\Mg(Q)}^j}\hat{T}^\star(\bullet)^j$, for $j=0,1$.
Define $\Mg(Q)^0:=Q^0\biguplus Q_\eta^0$, disjoint union of $Q^0$ and a singleton $Q^0_\eta:=\{\bullet^0_\eta\}$, and  $\pi_{\Mg(Q)}^0:\Mg(Q)^0\to\hat{T}^\star(\bullet)^0=\{\bullet^0\}$ as the terminal map, necessarily coinciding with the terminal map $\pi_Q^0$ on $Q^0$ and given by $\bullet^0_\eta\mapsto\bullet^0$ on the singleton. 
Notice that, at this level-$0$ stage, no contraction-cells or operadic multiplication-cells are added, but only a ``free operadic unity-cell'' $\bullet^0_\eta$. 

Passing to the level-$1$, we define $\Mg(Q)^1[0]:=Q^1\biguplus Q^1_\eta\biguplus Q^1_\kappa$, 
where $Q^1_\eta:=\{\bullet^1_\eta\}$ is again a singleton (corresponding to a free operadic-unit 1-cell); $Q^1_\kappa:=\Par(\pi^0_{\Mg(Q)})$ consists of a copy of the set of free contraction 1-cells for the parallel $0$-cells induced by the projection $\pi_{\Mg(Q)}^0$ at the previous 0-level; 
we set the map $\pi_{\Mg(Q)}^1[0]:\Mg(Q)^1[0]\to\hat{T}^\star(\bullet)^1$ coinciding with $\pi_Q^1$ on $Q^1$, as a terminal map $\bullet^1_\eta\mapsto \bullet^1\in\hat{T}^\star(\bullet)^1$ on $Q^1_\eta$, and as the map $(x^+,y,x^-)\mapsto y$ on all $(x^+,y,x^-)\in Q^1_\kappa$; furthermore we introduce new source/target maps $s_{\Mg(Q)}^0[0],t_{\Mg(Q)}^0[0]:\Mg(Q)^1[0]\to\Mg(Q)^0$ coinciding with the original source/target maps on $Q^1$; as $\bullet^1_\eta\mapsto\bullet^0_\eta$ on $Q^1_\eta$; and as $s_{\Mg(Q)}^0[0]:(x^+,y,x^-)\mapsto x^-$, respectively $t_{\Mg(Q)}^0[0]:(x^+,y,x^-)\mapsto x^+$ on $Q^1_\kappa$. 

Finally we introduce additional free operadic multiplication-1-cells by   $\Mg(Q)^1:=\biguplus_{k=0}^{+\infty}\Mg(Q)^1[k]$, where we inductively have  $\Mg(Q)^1[k+1]:=\Big\{(x,\mu,y) \ | \ (x,y)\in 
\Mg(Q)^1[k_1]\times_{\hat{T}^\star(\bullet)}\hat{T}^\star(\Mg(Q)^1[k_2]), \ k_1+k_2=k\Big\}$;  
we recursively define the projection maps $\pi^1_{\Mg(Q)}[k+1]:(x,\mu,y)\mapsto \pi^1_{\Mg(Q)}[k](x)=\mu^\star_{\bullet}\left((\hat{T}^\star(\pi^1_{\Mg(Q)}[k]))(y)\right)$, targets $t^0_{\Mg(Q)}[k+1]:(x,\mu,y)\mapsto (t^0_{\Mg(Q)}[k])(x)$ and sources 
$s^0_{\Mg(Q)}[k+1]:(x,\mu,y)\mapsto \mu^\star_{P}\left((\hat{T}^\star(s^0[k]))(y)\right)$, for all 
$(x,y)\in \Pg(Q)^1[k]\times_{\hat{T}^\star(\bullet)}\hat{T}^\star(\Mg(Q)^1[k])$. 

\medskip 

Assuming now inductively the existence of $\Mg(Q)^n\xrightarrow{\pi_{\Mg(Q)}^n}\hat{T}^\star(\bullet)^n$, we construct $\Mg(Q)^{n+1}\xrightarrow{\pi_{\Mg(Q)}^{n+1}}\hat{T}^\star(\bullet)^{n+1}$: starting from $\Mg(Q)^{n+1}[0]:=Q^{n+1}\biguplus Q^{n+1}_\eta\biguplus Q^{n+1}_\kappa$, where
$Q^{n+1}_\eta:=\{\bullet^{n+1}_\eta\}$ is a singleton free operadic-unit \hbox{$(n+1)$-cell} and 
$Q^{n+1}_\kappa:=\Par(\pi^n_{\Mg(Q)})$ is a family of free contraction $(n+1)$-cells induced by the already defined projection map $\pi^n_{\Mg(Q)}:\Mg(Q)^n\to\hat{T}^\star(\bullet)$, 
we introduce free operadic-multiplication $(n+1)$-cells by the recursive nesting $\Mg(Q)^{n+1}[k+1]:=\Big\{
(x,\mu,y) \ | \ (x,y)\in 
\Mg(Q)^n[k_1]\times_{\hat{T}^\star(\bullet)}\hat{T}^\star(\Mg(Q)^n[k_2]), \ k_1+k_2=k
\Big\}$ and we get $\Mg(Q)^{n+1}:=\biguplus_{k=0}^{+\infty}\Mg(Q)^{n+1}[k]$. 
The projection map $\pi^{n+1}_{\Mg(Q)}$ is separately defined on each set of the disjoint union: it coincides with $\pi^{n+1}_Q$ on $Q^{n+1}\subset\Mg(Q)^{n+1}[0]$; it is $\bullet_\eta^{n+1}\mapsto \bullet^{n+1}$ on $Q^{n+1}_\eta$; it is $(x^+,y,x^-)\mapsto y$ on $Q^{n+1}_\kappa$; 
and it is recursively given  
by $\pi_{\Mg(Q)}^{n+1}:(x,\mu,y)\mapsto
\pi^{n}_{\Mg(Q)}[k](x)= \mu_{\hat{T}^\star}\left(\hat{T}^\star(\pi^{n}_{\Mg(Q)}[k])(y)\right)$
on the elements $(x,\mu,y)\in\Mg(Q)^{n+1}[k+1]$. 
Finally we obtain a globular $\hat{T}^\star$-collection with target/source maps given by  $t^{n+1}_{\Mg(Q)}[k+1]:(x,\mu,y)\mapsto (t^{n+1}_{\Mg(Q)}[k])(x)$ and  $s^{n+1}_{\Mg(Q)}[k+1](x,\mu,y):(x,\mu,y)\mapsto\mu^\star_{P}\left((\hat{T}^\star(s^{n+1}_{\Mg(Q)}[k]))(y)\right)$, where $(x,y)\in \Mg(Q)^n[k]\times_{\hat{T}^\star(\bullet)}\hat{T}^\star(\Mg(Q)^n[k])$. 

\medskip 

The inclusions $\xi_Q^n:Q^n\to Q^n\biguplus Q^n_\eta\biguplus Q^n_\kappa=\Mg(Q)^n[0]\subset \Mg(Q)^n=\biguplus_{k=0}^{+\infty}\Mg(Q)^n[k]$, for all $n\in\NN$ define level-by-level the map $\xi_Q:Q\to\Mg(Q)$ that is already a morphism of globular $\hat{T}^\star$-collections. 

\medskip 

b. 

We construct level-by-level the several structural maps involved in the definition of contracted $\hat{T}^\star$-operadic magma: the contraction $\kappa_{\Mg(Q)}$, the unit $\eta_{\Mg(Q)}$, the multiplication $\mu_{\Mg(Q)}$. 

\medskip 
 
The operadic unit $\eta_{\Mg(Q)}: \bullet \to \Mg(Q)$ is defined as $\eta^n_{\Mg(Q)}: \bullet^n\mapsto \bullet_\eta^n$, for all $n\in\NN$. 

\medskip 

The operadic multiplication 
$\mu_{\Mg(Q)}: \Mg(Q)\times_{\hat{T}^\star(\bullet)}\hat{T}^\star(\Mg(Q))\to \Mg(Q)$ is given, at each level $n\in\NN$, for all $k\in\NN$, by the maps 
$\mu^n_{\Mg(Q)}[k]: \Mg(Q)^n[k]\times_{\hat{T}^\star(\bullet)}\hat{T}^\star(\Mg(Q)^n[k])\to \Mg(Q)^n[k+1]$  given as $\mu^n_{\Mg(Q)}[k]: (x,y)\mapsto (x,\mu,y)$. 

\medskip 

The contraction $\kappa_{\Mg(Q)}: \Par(\pi_{\Mg(Q)})\to\Mg(Q)$ is provided by the maps $\kappa^n_{\Mg(Q)}: \Par(\pi^n_{\Mg(Q)})\to\Mg(Q)^{n+1}$ defined, for all $n\in\NN$, as inclusions  $\kappa^n_{\Mg(Q)}:Q_\kappa^n\to\Mg(Q)^{n+1}$. 

\medskip 

c. 

Here we deal with the universal factorization property of $Q\xrightarrow{\xi_Q}\Mg(Q)$: 
given another contracted $\hat{T}^\star$-operadic magma $\hat{M}\xrightarrow{\hat{\pi}}\hat{T}^\star(\bullet)$ with contraction $\hat{\kappa}: \Par(\hat{\pi})\to \hat{M}$, operadic unit $\hat{\eta}:\bullet\to\hat{M}$ and operadic multiplication $\hat{\mu}:\hat{M}\times_{\hat{T}^\star(\bullet)}\hat{T}^\star(\hat{M})\to\hat{M}$, we show the existence of a unique morphism $\hat{\phi}:\Mg(Q)\to \hat{M}$ of contracted $\hat{T}^\star$-operadic magmas such that $\phi=\hat{\phi}\circ\xi_Q$. 

\medskip 

Since, by construction, the inclusion $Q\xrightarrow{\xi_Q}\Mg(Q)$, maps every element $x\in Q^n$ to the same element $x\in\Mg(Q)^n$, for all $n\in\NN$, we must necessarily have that $\hat{\phi}(x):=x$, for all $x\in Q\subset\Mg(Q)$. 
Since $\hat{\phi}$ should preserve the unit, $\hat{\phi}\circ\eta_{\Mg(Q)}=\hat{\eta}$, the explicit construction of $\eta_{\Mg(Q)}$, necessarily entails  $\hat{\phi}:\eta_\Mg(\bullet)\mapsto\hat{\eta}(\bullet)$, for all the elements $\eta_{\Mg(Q)}(\bullet)\subset\Mg(Q)$. 
Similarly, since $\hat{\phi}$ should be contraction preserving, $\hat{\phi}\circ\kappa_{\Mg(Q)}=\hat{\kappa}\circ(\hat{\phi},\hat{\phi})$, the only possible choice for the restriction of $\hat{\phi}$ on $\kappa_{\Mg(Q)}(\Par(\pi_{\Mg(Q)}))\subset\Mg(Q)$ is: 
$\hat{\phi}(x^+,y,x^-)\mapsto \hat{\kappa}(\hat{\phi}(x^+),y,\hat{\phi}(x^-))$. \footnote{
Notice that, if $(x^+,y,x^-)\in\Par(\pi_{\Mg(Q)})^n$, we have $x^+,x^-\in\Mg(Q)^{n-1}$ and, since the construction of $(\Mg(Q),\pi_{\Mg(Q)},\kappa_{\Mg(Q)})$, is produced inductively, for all $n\in\NN$, the definition of $\hat{\phi}$ on the elements $\kappa_{\Mg(Q)}(\Par(\pi_{\Mg(Q)}))^n\subset\Mg(Q)^n$ requires only the knowledge of the already available map $\hat{\phi}^{n-1}$. 
} 
Since $\hat{\phi}$ should preserve multiplications $\hat{\phi}^n\circ(\mu_{\Mg(Q)})^n=\hat{\mu}^n\circ(\hat{\phi}^n,(\hat{T}^\star(\hat{\phi}))^n)$, for all $n\in\NN$, the map $\hat{\phi}^n$ should be uniquely defined as $(x,\mu,y)\mapsto \hat{\mu}(\hat{\phi}^n,\hat{T}^\star(\hat{\phi})^n(y))$ on the elements in $\mu_{\Mg(Q)}\left(\Mg(Q)\times_{\hat{T}^\star(\bullet)}\hat{T}^\star(\Mg(Q))^n\right)\subset \Mg(Q)^n$. \footnote{
Notice again that, because of the inductive definition of $\Mg(Q)$, the elements in $\Mg(Q)\times_{\hat{T}^\star(\bullet)}\hat{T}^\star(\Mg(Q))^n$ only require the knowledge of already defined $\hat{\phi}^k$, for all $k\leq n$ on . 
} 
The already uniquely constructed $\hat{\phi}:\Mg(Q)\to\hat{M}$ is a morphism of globular contracted $\hat{T}^\star$-operadic magmas that also satisfies $\phi=\hat{\phi}\circ\xi_Q$.

\medskip 

d. 

The previously constructed free contracted $\hat{T}^\star$-operadic magma $\Mg(Q)$ is not yet a free contracted operad because its free unit $\eta_{\Mg(Q)}:\bullet\to\Mg(Q)$ and free multiplication $\mu_{\Mg(Q)}:\Mg(Q)\circ^1_0\Mg(Q)\to\Mg(Q)$ fail to satisfy the unitality and associativity axioms for a monad in the bicategory $\Qf_{\hat{T}^\star}$ as specified by the commuting diagrams in definition~\ref{def: bicat-monad}, in detail, denoting 
the associators and left/right unitors morphisms by 
$(\Mg(Q)\circ^1_0\Mg(Q))\circ_0^1\Mg(Q)\xrightarrow{\alpha_{\Mg(Q)}}\Mg(Q)\circ^1_0(\Mg(Q)\circ_0^1\Mg(Q))$ and $\iota^1(\bullet)\circ^1_0\Mg(Q) \xrightarrow{\lambda_{\Mg(Q)}}\Mg(Q) \xleftarrow{\rho_{\Mg(Q)}}\Mg(Q)\circ^1_0\iota^1(\bullet)$, we need to get identified all the pairs of terms in $\Mg(Q)$ contained in the following family $\X\subset \Mg(Q)\times\Mg(Q)$: 
\begin{align}\notag 
\X&:=\Big\{
\left(\lambda_{\Mg(Q)}(x_1)\ ,\ \mu_{\Mg(Q)}\left((\eta_{\Mg(Q)}\circ^2_0\iota^1_{\Mg(Q)})(x_1)\right)\right) \ \big| \ x_1\in \iota^1(\bullet)\circ^1_0\Mg(Q) 
\Big\} 
\ \cup
\\ \label{eq: X}
&\Big\{ \left(
\rho_{\Mg(Q)}(x_2)
\ , \ 
\mu_{\Mg(Q)}\left((\iota^1_{\Mg(Q)}\circ^2_0\eta_{\Mg(Q)})(x_2)\right)
\right)
\ \big| \ x_2\in \Mg(Q)\circ^1_0\iota^1(\bullet) 
\Big\}
\ \cup
\\ \notag 
&\Big\{
\left(
\mu_{\Mg(Q)}((\iota^1_{\Mg(Q)}\circ^2_0\mu_{\Mg(Q)})\circ\alpha_{\Mg(Q)}(x_3)) 
\ , \ 
\mu_{\Mg(Q)}((\mu_{\Mg(Q)}\circ^2_0\iota^1_{\Mg(Q)})(x_3))
\right)
\ \big| \ 
x_3\in (\Mg(Q)\circ^1_0\Mg(Q))\circ_0^1\Mg(Q)  
\Big\}.
\end{align} 

\medskip 

To solve this problem, we define $\E_\X$ as the \textit{smallest congruence of contracted $\hat{T}^\star$-operadic magma} in $\Mg(Q)$ containing all the pairs of terms in $\X$. 
From remark~\ref{rem: q-cong}  $\hat{T}^\star(\bullet)\xrightarrow{\pi_{\hat{T}^\star(\bullet)}}\hat{T}^\star(\bullet)$ is a contracted $\hat{T}^\star$-operad and, since $\Mg(Q)\xrightarrow{\pi_{\Mg(Q)}}\hat{T}^\star(\bullet)$ is morphism of contracted $\hat{T}^\star$-operadic magmas, by proposition~\ref{prop: q-cong} we know that the canonical equivalence relation $\E_{\pi_{\Mg(Q)}}:=\Mg(Q)\times_{\pi_{\Mg(Q)}}\Mg(Q)=\left\{(x,y)\in\Mg(Q)\times\Mg(Q) \ | \ \pi_{\Mg(Q)}(x)=\pi_{\Mg(Q)}(y) \right\}$ is itself a congruence of contracted $\hat{T}^\star$-operadic magmas with $\X\subset\E_{\pi_{\Mg(Q)}}$ and hence $\E_\X$ can be taken as the intersection of the (non-empty) family of all such congruences containing $\X$ in $\Mg(Q)$. 

\medskip 

By proposition~\ref{prop: q-cong}, the quotient map $\Mg(Q)\xrightarrow{\pi_{\E_\X}}\Pg(Q):=\Mg(Q)/\E_\X$ is a morphism of contracted $\hat{T}^\star$-operadic magmas and, since $\X\subset\E_\X$, the quotient $\Pg(Q):=\Mg(Q)/\E_\X\xrightarrow{\pi_{\Pg(Q)}}\hat{T}^\star(\bullet)$, where $\pi_{\Pg(Q)}$ is the only morphism of contracted $\hat{T}^\star$-operadic magmas such that $\pi_{\Mg(Q)}=\pi_{\Pg(Q)}\circ\pi_{\E_\X}$, is not only a contracted $\hat{T}^\star$-operadic magma, but it is already a contracted $\hat{T}^\star$-operad. 

\medskip 

The inclusion $Q\xrightarrow{\zeta_Q}\Pg(Q)$ given by $\zeta_Q:=\pi_{\E_\X}\circ\xi_Q$ is, by composition, a morphism of globular $\omega$-sets. 

\bigskip 

e. 

Suppose that $Q\xrightarrow{\phi}\hat{P}$ is a morphism of $\hat{T}^\star$-collections into another contracted $\hat{T}^\star$-operad $\hat{P}\xrightarrow{\hat{\pi}_P}\hat{T}^\star(\bullet)$ with contraction $\hat{\kappa}_P:\Par(\hat{\pi}_P)\to\hat{P}$, operadic unit $\hat{\eta}_P:\bullet\to \hat{P}$ and operadic multiplication $\hat{\mu}_{P}:\hat{P}\circ^1_0\hat{P}\to \hat{P}$. 

\medskip 

Since (forgetting the associativity and unitality axioms) every contracted $\hat{T}^\star$-operad is a contracted $\hat{T}^\star$-operadic magma, by the previous point c.~above, there exists a unique morphism $\Mg(Q)\xrightarrow{\tilde{\phi}}\hat{P}$ of contracted $\hat{T}^\star$-operadic magmas, defined on the free contracted $\hat{T}^\star$-operadic magma $\Mg(Q)$, such that $\phi=\tilde{\phi}\circ\xi_Q$. 

\medskip 

The equivalence relation $\E_{\tilde{\phi}}:=\Mg(Q)\times_{\tilde{\phi}}\Mg(Q)=\{(x,y)\in\Mg(Q)\times\Mg(Q) \ | \ \tilde{\phi}(x)=\tilde{\phi}(y)\}$ induced by the morphism $\Mg(Q)\xrightarrow{\tilde{\phi}}\hat{P}$ into the contracted $\hat{T}^\star$-operad $\hat{P}$, is a congruence of contracted $\hat{T}^\star$-operadic magma that contains all the terms $\X$ generating $\E_\X$ and hence, by the minimality of $\E_\X$, we have $\E_\X\subset\E_{\tilde{\phi}}$. It follows, by proposition~\ref{prop: q-cong}, that there exists a unique quotient morphism $\Pg(Q)\xrightarrow{\hat{\phi}}\hat{P}$ of contracted $\hat{T}^\star$-operads, given by $\hat{\phi}(\pi_{\E_\X}(x)):=\tilde{\phi}(x)$, for all $x\in\Mg(Q)$. 
The universal factorization property for free $\hat{T}^\star$-operads over $Q$ is satisfied since: 
$\hat{\phi}\circ\zeta_Q=\hat{\phi}\circ\pi_{\E_\X}\circ\xi_Q=\tilde{\phi}\circ\xi_Q=\phi$. 
\end{proof}

\begin{remark}
The proof of theorem~\ref{th: free-c-T*} could be obtained mimicking Leinster's techniques in section~\ref{sec: L-w-T}. Specifically, applying lemma~\ref{lem: kelly}
after showing that:  
a) the category $\Qf^{\hat{T}^\star}_\bullet$ is locally finitely presentable; 
b) $\hat{T}^\star$ is a finitary Cartesian monad (and hence 
$\Ug^\star_\Of$ is monadic)  
c) $\Ug^\star_\Of$ is finitary
d) $\Ug^\star_\kappa$ is monadic and finitary. 
\xqed{\lrcorner}
\end{remark}

\medskip 

Our main definition in the $T^\star$ case, is in perfect analogy with Leinsteir's definition~\ref{def: w-omega}: 
\begin{definition}
A \emph{weak involutive globular $\omega$-category} is an algebra for an initial object $L^\star$ in  $\Of^{T^\star, \kappa}_\bullet$. 
\end{definition}

\begin{remark} \label{rem: ope-con-fin}
It is perfectly possible to utilize the category $\Of\Kf^{\hat{T}^\star}_\bullet$ introducted in remark~\ref{rem: ope-con*} instead of $\Of^{\hat{T}^\star,\kappa}_\bullet$ in order to define a slightly more restrictive notion of involutive weak globular $\omega$-category as an algebra for the initial object in $\Of\Kf^{\hat{T}^\star}_\bullet$.  

\medskip 

The existence of such initial object can be obtained with techniques perfectly similar to those utilized in theorems~\ref{th: L*} and in the proof of theorem~\ref{th: free-c-T*}, just adding to the family $\X$ in equation~\eqref{eq: X} all the pairs of terms of $\Mg(Q)$ required for the validity of the additional axioms imposed by the commuting diagrams~\eqref{eq: L*+} and quotienting for the smallest congruence of $\hat{T}^\star$-operadic contraction containing $\X$ (such minimal congruence always exists because $\hat{T}^\star(\bullet)$ is a terminal object in $\Of\Kf^{\hat{T}^\star}_\bullet$). 
\xqed{\lrcorner}
\end{remark}

\subsection{Examples} 

The strategy used to provide examples of weak involutive globular $\omega$-categories is perfectly parallel to the one described in section~\ref{sec: ex-T} and consists in producing: 
a contracted $\hat{T}^\star$-operad $P^\star$ and an algebra $X^\star$ over it.  

\medskip 

We just mention here some immediate available examples of involutive weak globular $\omega$-categories. 
\begin{itemize}
\item[$\blacktriangleright$] 
\emph{strict involutive globular $\omega$-categories}: 
again, these coincide with algebras for the monad given by the terminal contracted-$\hat{T}^\star$-operad $\hat{T}^\star(\bullet)\in\Of^{\hat{T}^\star,\kappa}_\bullet$ described in remark~\ref{rem: q-cong}. 
\item[$\blacktriangleright$] 
\emph{globular $\omega$-spans}: following the notation introduced in~\cite[example 4.7]{BCLS20} we define an $\omega$-span as a sequence $(A^{n}\xleftarrow{t^{n}}Q^{n+1}\xrightarrow{s^{n}}B^{n})_{n\in\NN}$ of 1-spans, where $Q^n:=A^n\cup B^n$; a globular $\omega$-span is an $\omega$-span $Q:=(Q^{n+1}\rightrightarrows Q^{n})$ that is a globular $\omega$-set (in this way any $x\in Q^n$ determines, with all its sources/targets a unique globular $n$-cell). 
Introducing level-by-level the equivalence relation that identifies $(n+1)$-cells having the same source/target sets, we obtain a quotient $\omega$-globular set $X$. A family of globular $\omega$-spans is hence determined by the quotient morphism $Q\xrightarrow{\chi}X$.  
Considering such family $Q\xrightarrow{\chi}X\xrightarrow{!}\bullet$ as ``generators'', we apply the free involutive magma functor (as constructed in point a.~of the proof of proposition~\ref{prop: free-inv-omega-cat}) to get $\Mg(Q)\xrightarrow{\Mg(\chi)}\Mg(X)\xrightarrow{\Mg(!)}\Mg(\bullet)$, applying the forgetful functor to the category of globular $\omega$-sets (that we omit to indicate) and using the quotient projection (constructed in point b.~of proposition~\ref{prop: free-inv-omega-cat}) onto $\hat{T}^\star(Q)\xrightarrow{\hat{T}^\star(\chi)}\hat{T}^\star(X)\xrightarrow{\hat{T}^\star(!)}\hat{T}^\star(\bullet)$, we obtain a morphism of $\hat{T}^\star$-collections as in the first diagram below: 
\begin{equation*}\vcenter{
\xymatrix{
\Mg(Q)\ar[rr]^{\Mg(\chi)}\ar[dr]^{} & &\Mg(X)\ar[dl]_{} 
\\ 
& \hat{T}^\star(\bullet) &
}
}
\quad \mapsto \quad 
\vcenter{
\xymatrix{
\Pg(Q)\ar[rr]^{\Pg(\chi)}& &\Pg(X) 
}
} 
\end{equation*}
The \textit{involutive weak globular $\omega$-category of globular $\omega$-spans generated by $Q\xrightarrow{\chi}X$} is given by the morphism of free globular contracted $\hat{T}^\star$-operads to the right of the diagram above (considered as algebras over themselves). Keeping track of the projection onto $X$ and $\Pg(X)$ is necessary to ``coarse grain'' recovering the original spans and operations between them. 
\item[$\blacktriangleright$] 
\emph{homotopy $\omega$-groupoids} $\Pi_\omega(X)$ of a topological space $X$: here, since the involutions coincide with the weak inverse homotopies, one just show that the contracted $\hat{T}$-operad utilized in~\cite[example 9.2.7]{Lei04} is actually a contracted $\hat{T}^\star$-operad. 
\end{itemize}

\section{Outlook}\label{sec: out} 

The construction of algebraic involutive versions of weak (globular) $\omega$-categories (either in Penon's or in Leinster's approaches), as done in our previous work~\cite{BeBe17} and in the present paper, is only the very first step in the direction of a full operator algebraic categorical environment suitable for the needs of \textit{categorical non-commutative geometry}~\cite{BCL08}, \cite{BCL12}. 

\medskip 

Involutions, in the case of (weak) cubical $\omega$-categories, are currently under investigation,\footnote{
Bejrakarbum P, Bertozzini P, Theesoongnern S, \textit{Involutive Weak Cubical/Globular $\omega$-categories} (work in progress).
} 
following recent work by~\cite{Ka22}, 
including the study of conditions (see~\cite{ABS02}) assuring the equivalence between cubical and globular \textit{involutive} settings, extending previous still unpublished results already achieved in the case of involutive 2-categories / double categories~\cite{BCDM14}. 

\medskip

Immediate further developments of the present work will concentrate on possible algebraic definitions and examples of involutive weak $\omega$-algebroids; subsequently the treatment of uniform structures related to completeness and norms necessary for the formulation of weak $\omega$-C*-algebroids will have to be addressed, generalizing (and possibly modifying) the strict $n$-C*-categorical notions tentatively put forward in~\cite[section~5]{BCLS20}. 

\medskip 

In the present paper, for simplicity, we have only considered monads and operads that do not possess involutive symmetries, but it is already evident that certain ``covariant'' involutions could have been introduced at the level of the operads in $\Of^{\hat{T}^\star}$. Certain involutive monads and operads (see~\cite[chapter 4]{Ya20} and references therein) can be used for this purpose. A full treatment of \textit{involutive bicategories} and the discussion of \textit{covariant vs contravariant involutive monads and operads in a bicategory} will be separately addressed in a companion paper~\footnote{
Bejrakarboom P, Bertozzini P, Puttirungroj C, \textit{Involutive Monads and Hybrid Categories} (work in progress). 
} 
making direct use of hybrid-categories as put forward in~\cite{BePu14}. 

\medskip 

In this work we have not considered any relaxing of the usual axioms for globular higher categories, in particular we did not formulate a definition of weak involutive globular $\omega$-categories with \textit{non-commutative exchange property}, as already proposed in~\cite[section 3.3]{BCLS20} for strict globular $n$-categories. We suspect that treatments of versions  of \textit{non-commutative derived geometries} (homotopies, cobordisms, holonomies) will require some axiomatic modification in such direction. 

\bigskip 

\emph{Notes and Acknowledgments:} 
P. Bejrakarbum is supported by the ``Research Professional Development Project'' under the ``Science Achievement Scholarship of Thailand'' (SAST).  

P. Bertozzini thanks Starbucks Coffee (Langsuan, Jasmine City, Gaysorn Plaza, Emquartier Sky Garden) where he spent most of the time dedicated to this research project; he thanks Fiorentino Conte of ``The Melting Clock'' for the great hospitality during many crucial on-line dinner-time meetings.

\end{document}